\documentclass[10pt]{amsart}  

\usepackage[left=1in,right=1in,top=1in,bottom=1.17in]{geometry}

\usepackage[utf8]{inputenc}

\usepackage{graphicx,amssymb,amsmath,amsthm,wrapfig}
\usepackage{caption}
\usepackage{subcaption}
\usepackage{indentfirst}
\usepackage{cancel}
\usepackage{lipsum}
\usepackage{epstopdf}
\usepackage{epsfig}
\usepackage{comment}
\usepackage{enumerate,color}
\usepackage{empheq}

\usepackage{hyperref}
\usepackage{xcolor}

\usepackage{soul}      
\sethlcolor{yellow}    

\everymath{\displaystyle}
\theoremstyle{plain}
\newtheorem{theorem}{Theorem}[section]

\newtheorem{lemma}[theorem]{Lemma}

\newtheorem{proposition}[theorem]{Proposition}
\theoremstyle{remark}
\newtheorem{remark}{Remark}

\numberwithin{equation}{section}

\newcommand{\UU}{\overline{W}}

\newcommand{\veps}{\varepsilon}
\newcommand{\ve}{\varepsilon}
\newcommand{\wt}{\widetilde}

\newcommand{\RNum}[1]{\uppercase\expandafter{\romannumeral #1\relax}}

\title[Blow-up in the regularized Saint--Venant equations]{Asymptotic self-similar blow-up for the regularized Saint--Venant equations}  

\author[Y. Kim, B. Kwon, and W. Shim]{Yunjoo Kim, Bongsuk Kwon, and Wanyong Shim}

\address{(Yunjoo Kim) Department of Mathematical Sciences, Ulsan National Institute of Science and Technology, Ulsan, 44919, Korea}
\email{gomuli3@unist.ac.kr}

\address{(Bongsuk Kwon) Department of Mathematical Sciences, Ulsan National Institute of Science and Technology, Ulsan, 44919, Korea}
\email{bkwon@unist.ac.kr}

\address{(Wanyong Shim) Department of Mathematical Sciences, Korea Advanced Institute of Science and Technology, Daejeon, 34141, Korea}
\email{wyshim25@kaist.ac.kr}

\thanks{\textbf{Acknowledgment.}
B.K. was supported by the National Research Foundation of Korea(NRF) grant funded by the Korea government(MSIT)(00560003).}

\begin{document}

\begin{abstract}
We investigate singularity formation in the regularized Saint--Venant (rSV) equations, a conservative, non-dispersive shallow water system that is formally regarded as a Hamiltonian regularization of the isentropic Euler equations. While it is known that smooth solutions to the rSV system can develop gradient blow-up in finite time, the precise structure of such singularities has not been rigorously characterized. In this work, we establish stability of self-similar blow-up profiles of the Hunter--Saxton equation within the rSV framework, using a nonlinear bootstrap argument in dynamically rescaled coordinates. Our analysis captures the detailed space-time dynamics of solutions near the singularity, and proves their sharp $C^{3/5}$ H\"older regularity up to the singular time. This regularity differs from the $C^{1/3}$ H\"older regularity of the cubic-root singularities found in the compressible Euler and inviscid Burgers equations. This contrast highlights the structural influence of the Hamiltonian regularization on singularity formation. To illuminate this effect, we also show that the same $C^{3/5}$ blow-up profile emerges in the regularized Burgers equation, a scalar analogue of the rSV system. ~\\


\end{abstract}

\maketitle



\section{Introduction}
We consider the regularized Saint--Venant (rSV) equations:
\begin{equation} \label{rSV1}
	\begin{split}
		& h_t + (hu)_x =0, \\
		& (hu)_t + \left( hu^2 + \frac{1}{2} gh^2 + \gamma \mathcal{S} \right)_x = 0,  \quad \text{where}\\
		& \mathcal{S} := h^3 \left( -u_{tx} - uu_{xx} + u_x^2 \right) -gh^2 \left( hh_{xx} + \frac{1}{2}h_x^2 \right)
	\end{split}
\end{equation}
for $x \in \mathbb{R}$ and $t > -\varepsilon$ with $\varepsilon>0$. This system serves as a higher-order regularized model for shallow water flows. Here the unknowns $h=h(x,t)>0$ and $u=u(x,t)$ represent the fluid depth and the depth-averaged horizontal velocity, respectively, and the positive constants $g$ and $\gamma$ denote the gravitational acceleration and a dimensionless regularization parameter. For simplicity, we normalize the parameters by setting $g=\gamma=1$ unless stated otherwise.

In the formal limit $\gamma \to 0$, the system \eqref{rSV1} reduces to the classical Saint--Venant (cSV) equations, a hyperbolic system that may admit non-unique weak solutions and develop discontinuous shock waves. Motivated by these challenging features, several regularization methods for such hyperbolic systems have been proposed, typically involving dissipative or dispersive corrections \cite{KL, Le}. These approaches, however, may compromise physical fidelity, for instance by violating energy conservation, or introduce artificial oscillations \cite{CL, GH, LL}. In view of these considerations, the rSV system was introduced in \cite{CD} as a conservative, non-dispersive variant of the Green--Naghdi equations \cite{GLN, GN} with a Hamiltonian structure. It preserves the non-dissipativity of the cSV system while ensuring the conservation of mass, momentum, and energy.

Despite being a regularized model, the rSV system exhibits a range of singular behaviors that distinguish it from other regularized systems. Notably, smooth solutions may still develop finite-time gradient blow-up, i.e., loss of \( C^1 \) regularity, as shown in \cite{LPP}. In addition, the system admits weakly singular traveling wave solutions that are continuous but not differentiable, such as piecewise smooth shocks and cusped solitary waves \cite{PPDC}. The shock profiles satisfy the Rankine--Hugoniot condition and propagate at the same speed as hydraulic shocks in the cSV system. They may therefore be regarded as regularized counterparts of the cSV shocks. These distinct nonlinear phenomena have motivated further studies of the rSV system and other Hamiltonian regularized models; see \cite{CDM, Gu, Gu1, GCJ, GCJ1, GJCP}.

These singular behaviors place the rSV system within a broader class of nonlinear fluid models that develop singularities in finite time. Prototypical examples include the compressible Euler equations, the inviscid Burgers equation, and various water wave models. In such systems, recent advances have revealed that the singularity profile generically attains $C^{1/3}$ H\"older regularity at the blow-up time. See \cite{BSV, BSV1} for the multi-dimensional compressible Euler equations, \cite{CGIM, CGM} for the two-dimensional unsteady Prandtl system and the Burgers equation with transverse viscosity, respectively, \cite{BKK1, BKK2} for the Euler--Poisson system, and \cite{OP} for some perturbations of the Burgers equation, where the blow-up dynamics have been rigorously studied. More recently, a different type of singularity with $C^{3/5}$ regularity has been identified in the Hunter--Saxton (HS) and Camassa--Holm (CH) equations in \cite{KKY}. In this context, our results show that the rSV system provides a new example of the $C^{3/5}$-type singularity. Here we note that the cSV system is formally equivalent to the isentropic Euler equations with a specific pressure law, and hence gives rise to $C^{1/3}$ regularity at blow-up. The observed gap between the $\textstyle \frac13$ and $\textstyle  \frac35$ H\"older exponents indicates a structural change in blow-up dynamics resulting from the regularization.

In this paper, we establish the sharp $C^{3/5}$ H\"older regularity of gradient blow-up solutions for the rSV equations, along with the detailed spatial and temporal dynamics of the blow-up. This provides rigorous confirmation of the heuristic $x^{3/5}$-root singularity discussed in \cite{LPP}, near the blow-up point. To support a structural interpretation of the observed $C^{3/5}$ regularity, we also examine the regularized Burgers (rB) equation derived in \cite{GJCP},
\begin{equation} \label{rB'}
v_t + vv_x = v_{xxt} + 2v_xv_{xx} +vv_{xxx},
\end{equation}
a scalar analogue of the rSV system that preserves similar conservative and non-dispersive regularization properties. Following the analysis for rSV, we verify that the rB equation also exhibits the $C^{3/5}$-type singularity. In the Hamiltonian regularizations considered here, the non-dispersive corrections ensure conservation of $H^1$-like energy. This structural change modifies the leading-order dynamics underlying singularity formation and results in enhanced H\"older regularity. A detailed discussion of this observation is given in Section~\ref{Discussion}.

\subsection{Reformulation using Riemann functions}
We begin by introducing a nonlocal operator. For a smooth function $w : \mathbb{R} \to \mathbb{R}$, we define
\begin{equation} \label{I_h}
\mathcal{I}_h(w) := h w - \left( h^3 w_x \right)_x.
\end{equation}
Provided that $\textstyle \inf_{x \in \mathbb{R}} h >0$ and $h \in W^{1,\infty}(\mathbb{R})$, this operator defines an isomorphism from $H^2(\mathbb{R})$ to $L^2(\mathbb{R})$; see \cite[Lemma~3.3]{LPP}. Consequently, the inverse $\mathcal{I}_h^{-1}$ is well-defined as a nonlocal operator. With this notation, the rSV system \eqref{rSV1} can be rewritten as
\begin{equation}\label{eta_u}
	\begin{split}
		& h_t + (hu)_x = 0,\\
		& u_t + uu_x + h_x = - \mathcal{I}_h^{-1} \partial_x \left( 2 h^3 u_x^2 - \frac{1}{2}h^2 h_x^2 \right). 
	\end{split}
\end{equation}
Here we note that the right-hand side of the second equation can be expressed using the definition \eqref{I_h} as
\begin{equation}\label{G-q}
	\begin{split}
		-\mathcal{I}_h^{-1} \partial_x \left( 2 h^3 u_x^2 - \frac{1}{2}h^2 h_x^2 \right) = 2 G - 2 q,
	\end{split}
\end{equation}
where
\begin{equation}\label{G}
	G(x,t) := \int_{-\infty}^x (u + \sqrt{h})_x (u - \sqrt{h})_x \, dx', \qquad q(x,t) := \mathcal{I}_h^{-1} (hG).
\end{equation}
If $h-h_*,u\in H^s(\mathbb{R})$ for some $h_*>0$ and $s>\tfrac{3}{2}$, then the positivity assumption on $h$ implies that $\textstyle G_x=u_x^2-\frac{h_x^2}{4h}\in L^1(\mathbb{R})$, and hence $G,hG\in C_{\mathrm{lim}}(\mathbb{R})$. Here $C_{\mathrm{lim}}(\mathbb{R})$ denotes the space of continuous functions having finite limits at $\pm\infty$. In this case, $q$ is well-defined by \cite[Lemma~4.4]{LPP}.

We next introduce the \emph{Riemann functions}
\begin{equation} \label{wz}
	w := u + 2\sqrt{h}, \quad z := u - 2\sqrt{h},
\end{equation}
associated to the characteristic speeds
\begin{equation*}
\lambda_+ := \frac{1}{4} (3w+z), \quad \lambda_- : = \frac{1}{4} (w+3z),
\end{equation*}
respectively. The functions $w$ and $z$ coincide with the Riemann invariants of the classical shallow water system, which corresponds to \eqref{rSV1} with $\gamma = 0$ and $g=1$. 

Applying the identity \eqref{G-q} and the temporal scaling \( \textstyle \tilde{t} = \frac{3}{4}t \), we rewrite the system \eqref{eta_u} in terms of the Riemann functions:
\begin{subequations} \label{rSV'}
\begin{align}
\label{rSV'1} w_{t} + \left( w + \frac{1}{3} z \right) w_x = \frac{8}{3}G - \frac{8}{3}q, \\
\label{rSV'2} z_{t} + \left( z + \frac{1}{3} w \right) z_x = \frac{8}{3}G - \frac{8}{3}q,
\end{align}
\end{subequations}
where, for notational simplicity, we continue to write $t$ in place of \( \tilde{t} \).

\newcommand{\tw}{v}

\subsection{Self-similar Hunter--Saxton profiles} \label{sec:HSpro}
Following \cite{KKY}, we consider the stable self-similar blow-up profile $\overline{W}(y)$ for the Hunter--Saxton equation, satisfying the following ODE in similarity variables:
\begin{equation} \label{Weq}
\left( 1 + \frac{1}{2} \overline{W}'(y) \right) \overline{W}'(y) + \left( \overline{W}(y) + \frac{5}{2}y \right) \overline{W}''(y) = 0, \quad y \in \mathbb{R},
\end{equation}
along with the asymptotic growth condition
\begin{equation} \label{decay-infty}
\lvert \overline{W} (y) \rvert \lesssim \lvert y \rvert^{a}, \quad \lvert y \rvert \geq 1
\end{equation}
for some constant $a \in [0,1]$. We recall below the existence and properties of such profiles.

\begin{proposition}[\cite{KKY}, Proposition~2.1]\label{Profile-construct}
The ODE problem \eqref{Weq}--\eqref{decay-infty} admits a one-parameter family of smooth solutions $\{ \overline{W}_\beta(y) : \beta>0 \}$ such that for each $\beta>0$, $\overline{W}_\beta(y)$ is monotonically decreasing on $\mathbb{R}$ and is odd, i.e., $\overline{W}_\beta(-y) = - \overline{W}_\beta(y)$ for all $y\in\mathbb{R}$. Furthermore, it holds that
\begin{equation}\label{W_taylor_small}
\overline{W}_\beta (y) = -2y+\frac{128 \beta}{3} y^3+O(y^5) \quad\text{for }|y|\ll 1
\end{equation}
and
\begin{equation} \label{W_taylor_far}
\overline{W}_\beta (y) = -\frac{5}{3(50\beta)^{1/5}}y^{3/5}+ o(y^{3/5}) \quad \text{for } |y|\gg 1.
\end{equation}
In particular,
\begin{equation}\label{WW-0}
\overline{W}_\beta(0) = 0, \quad \overline{W}_\beta'(0) =-2, \quad \overline{W}_\beta''(0) = 0, \quad \overline{W}_\beta^{(3)} (0) = 256 \beta, 
\end{equation} 
and, as $|y|\rightarrow \infty$,
\begin{equation}\label{asymp-y-infty}
|y|^{-3/5} | \overline{W}_\beta(y) | \to \frac{5}{3}(50\beta)^{-1/5}, \quad |y|^{2/5} | \overline{W}_\beta'(y) | \to (50\beta)^{-1/5}, \quad |y|^{7/5} | \overline{W}_\beta''(y) | \to \frac{2}{5}(50\beta)^{-1/5}. 
\end{equation}
\end{proposition}

For definiteness, we fix $\beta=1$, and denote the corresponding profile by $\overline{W}(y):=\overline{W}_1(y)$ throughout this paper.

\subsection{Main result} \label{sec:main}
We formulate an initial value problem for the rSV system \eqref{rSV'} and describe the assumptions on the initial data that lead to the formation of a finite-time singularity. Let us set the initial time $t=-\veps$, where $\veps$ is to be chosen sufficiently small. We consider the initial data $(w_0,z_0)(x) := (w,z)(x,-\veps)$ satisfying
\begin{equation} \label{init_h}
\inf_{x\in\mathbb{R}}(w_0 - z_0)(x) >0
\end{equation}
and
\begin{equation} \label{IC}
(w_0-2\sqrt{h_*},z_0+2\sqrt{h_*})(x) \in H^5(\mathbb{R}) \times H^5(\mathbb{R})  \subset C^4(\mathbb{R}) \times C^4(\mathbb{R}),
\end{equation}
where $h_*>0$ is a fixed constant representing the background fluid depth. Note that the condition \eqref{init_h} is equivalent to $\textstyle \inf_{x\in\mathbb{R}}h_0(x)>0$, where $h_0(x) := h(x,-\varepsilon)$.

We now list further assumptions on $(w_0,z_0)$, which are satisfied on a non-empty open set of initial data in an appropriate topology; see Appendix~\ref{app:admissible-data}. We first impose the conditions
\begin{equation}\label{init_w0}
    \partial_xw_0(0)=-2\varepsilon^{-1},\quad
    \partial_x^2w_0(0)=0,\quad
    \partial_x^3w_0(0)=256\varepsilon^{-6}
\end{equation}
together with the following uniform and $L^2$ bounds:
\begin{equation}\label{init_wbound}
    \begin{split}
        &\lVert \partial_xw_0\rVert_{L^\infty}\leq 2\varepsilon^{-1},
        \quad \lVert \partial_x^4w_0\rVert_{L^\infty}\leq C_1\varepsilon^{-17/2}, \quad \lVert \partial_x^jw_0\rVert_{L^2}
        \leq C_j\varepsilon^{-(10j-11)/4},
        \quad j=2,3,4,\\
        &\lVert z_0+2\sqrt{h_*}\rVert_{L^\infty}\leq 1,
        \quad \lVert \partial_xz_0\rVert_{L^\infty}\leq \frac{1}{\sqrt{h_*}},
        \quad \lVert \partial_x^jz_0\rVert_{L^\infty}\leq 1,
        \quad j=2,3,4,
    \end{split}
\end{equation}
where \(C_j>0\), $j=1,2,3,4$, are fixed constants chosen sufficiently large; for definiteness, we take \(C_1=2^{13}\), \(C_2=2^2\), \(C_3=2^6\), and \(C_4=2^{11}\). These assumptions ensure that $\partial_x w_0$ attains its global minimum at $x=0$, and notably the constraint $\textstyle \lVert \partial_x z_0 \rVert_{L^\infty} \leq \frac{1}{\sqrt{h_*}}$ is consistent with the setup employed in \cite[Section~5.2]{LPP}. We note that no uniform bound is imposed on $\partial_x^2w_0$ or $\partial_x^3w_0$ in \eqref{init_wbound}. The $\varepsilon$-dependent $L^2$ bounds yield the required $L^\infty$ control through the one-dimensional Gagliardo--Nirenberg inequality; see Remark~\ref{init_rmk}. The fixed constants above are compatible with \eqref{init_w0}; see Lemma~\ref{lem:profile-norms} and the explicit construction of admissible initial data in Appendix~\ref{app:data-compatibility}.

To obtain a sharp description of the blow-up regularity of $w$, we assume that
\begin{equation}\label{init_wx_weight}
	\left\lvert \varepsilon (\partial_xw_0)(x) - \overline{W}' \left( \frac{x}{\varepsilon^{5/2}} \right) \right\rvert \leq \min{\left\{ \frac{\left( \frac{x}{\varepsilon^{5/2}} \right)^2}{3000 \left( 1 +  \left( \frac{x}{\varepsilon^{5/2}} \right)^2 \right) }, \frac{\Theta}{1+\left( \frac{x}{\varepsilon^{5/2}} \right)^{2/5}} \right\} }
\end{equation}
for all $x \in \mathbb{R}$ and
\begin{equation}\label{init_wx_dec}
	\limsup_{\lvert x \rvert \to \infty}{\lvert x^{2/5} \partial_xw_0(x)\rvert} \leq \frac{\theta}{2}, \quad \theta := \frac{1}{3} \left( \frac{6}{13} - \Theta \right).
\end{equation}
Here $\overline{W}$ is the smooth solution of \eqref{Weq} described in Proposition \ref{Profile-construct} and $\Theta >0$ is any fixed number such that $\textstyle \Theta \in (50^{-1/5}, \frac{6}{13})$. As for $z_0$, we impose the following weighted condition on $\partial_x z_0$:
\begin{equation}\label{init_zx_weight}
	(\veps+x^{2/5})|\partial_xz_0(x)|\leq \frac{1}{\sqrt{h_*}}.
\end{equation}

We define the energy $E$ associated with \eqref{eta_u} as
\begin{equation} \label{def_E}
	E(t) := \int_{\mathbb{R}}  \frac{1}{2} h u^2 + \frac{1}{2}(h - h_*)^2 + \frac{1}{2} h^3 \left( u_x^2 + \frac{h_x^2}{h} \right) \, dx.
\end{equation}
Note that as long as the smooth solution persists, the energy $E$ is conserved for all $t \geq -\varepsilon$, that is,
\begin{equation} \label{Energy}
	E(t) = E(-\varepsilon) =: E_0.
\end{equation}
We assume that the initial energy satisfies
\begin{equation} \label{E0_bound}
	E_0\leq \frac{h_*^3}{6}.
\end{equation}
It will be shown in Lemma~\ref{h_max} that the condition \eqref{E0_bound} guarantees that the positivity condition \eqref{init_h} for $h$ holds, as long as the smooth solution exists.

In what follows, $C^\alpha(\Omega)$ denotes  the H\"older space with exponent $\alpha\in(0,1]$ for a bounded open set $\Omega \subset \mathbb{R}$, and $[w]_{C^{\alpha}(\Omega) }$ represents the corresponding H\"older semi-norm, defined as 
\begin{equation*}
	[w]_{C^{\alpha}(\Omega) } := \sup_{x,y\in \Omega, x\ne y} \frac{|w(x) - w(y) | }{|x-y|^{\alpha}}.
\end{equation*}
We use the convention that for any scalar-valued function space $X(\mathbb{R})$, the notation $(u,v)\in X(\mathbb{R})$ is understood componentwise, i.e., $u,v\in X(\mathbb{R})$.

The main theorem is stated as follows.

\begin{theorem}\label{mainthm}
Fix $\textstyle \Theta\in (50^{-1/5},\frac{6}{13})$. There exists a constant $\ve_0 >0$ such that for each $\veps\in(0,\veps_0)$, if the initial data $(w_0-2\sqrt{h_*},z_0+2\sqrt{h_*})\in H^5(\mathbb{R})$ satisfies \eqref{init_h}--\eqref{init_zx_weight} and \eqref{E0_bound}, then the initial value problem for \eqref{rSV'} admits a unique classical solution $(w,z)$ satisfying
\begin{equation} \label{sol_reg}
(w-2\sqrt{h_*}, z+2\sqrt{h_*}) \in C([-\varepsilon,T_*); H^5(\mathbb{R})) \cap C^1([-\varepsilon,T_*); H^4(\mathbb{R}))
\end{equation}
and
\begin{equation} \label{sol_reg1}
(w,z) \in C^4([-\varepsilon,T_*) \times \mathbb{R}),
\end{equation}
which blows up in the $C^1$-norm at $t=T_*$ for some $T_*<\infty$. Moreover, the following holds:
\begin{enumerate}[(i)]
\item For any bounded open set $\Omega\subset\mathbb{R}$,
\begin{equation} \label{no-blow}
\sup_{t<T_*} [w(\cdot,t)]_{C^{3/5}(\Omega)}<\infty.
\end{equation}
\item For any $\textstyle \alpha \in (\frac{3}{5},1]$ and any bounded open set $\Omega\subset\mathbb{R}$,
\begin{equation} \label{blow-up-result1}
\begin{cases}
\lim_{t\nearrow T_*} [w(\cdot,t)]_{C^\alpha(\Omega)} = \infty& \text{if } x_*\in\Omega, \\
\limsup_{t\nearrow T_*} [w(\cdot,t)]_{C^\alpha(\Omega)} <\infty& \text{if } x_*\notin\overline{\Omega},
\end{cases}
\end{equation}
where $x_*$ is the blow-up location.\footnote{$x_*$ is defined in the proof of Theorem~\ref{mainthm}.}
\item For any bounded open set $\Omega$ containing $x_*$ and any $\textstyle \alpha \in (\frac{3}{5},1]$,
\begin{equation} \label{blow-up-rate1}
[w(\cdot,t)]_{C^\alpha(\Omega)}\sim (T_*-t)^{-(5\alpha-3)/2}
\end{equation}
for all $t$ sufficiently close to $T_*$.\footnote{Here, $A(t) \sim B(t) $ means that $C^{-1} B(t) \le A(t) \le C B(t)$ for some $C>0$ independent of $t$.}
\item For $z$, we have
\begin{equation} \label{z_result}
\sup_{t<T_*}[z(\cdot,t)]_{C^1(\mathbb{R})}<\infty.
\end{equation}
\end{enumerate}
\end{theorem}

In view of \eqref{wz} and the positivity of $h$ (cf. Lemma~\ref{h_max}), Theorem~\ref{mainthm} implies that both $h$ and $u$ exhibit $C^{3/5}$ regularity up to the blow-up time and satisfy statements corresponding to \textit{(i)--(iii)}. 

\subsection{Overview of the proof of Theorem~\ref{mainthm}}\label{sec:intro-proof-outline}

We briefly describe the main ideas of the proof. As will be discussed in Section~\ref{Discussion}, after differentiating \eqref{rSV'1} in $x$, the leading-order blow-up dynamics are governed by the Hunter--Saxton equation. This motivates the use of the profile \(\overline W\) in Proposition~\ref{Profile-construct} and of modulated similarity variables adapted to its scaling; precisely, we introduce
\begin{equation*}
y=\frac{x-\xi(t)}{(\tau(t)-t)^{5/2}}, \quad s=-\log(\tau(t)-t), \quad w-\kappa=e^{-3s/2}W, \quad z=Z, \quad q=Q, \quad G=\widetilde G.
\end{equation*}
The modulation parameters \(\tau,\kappa,\xi\) track the blow-up time, the height, and the center of the singularity. They are chosen so that the normalization conditions
\begin{equation*}
W(0,s)=0, \quad W_y(0,s)=-2, \quad W_{yy}(0,s)=0
\end{equation*}
are preserved along the flow. In these variables, the leading part of the \(W_y\)-equation is transported by \(\textstyle \frac52 y+\overline W\), while the modulation terms, the coupling to \(Z\), and the nonlocal quantity \(Q\) appear as perturbative errors.

The main part of the proof is a nonlinear stability argument around \(\overline W\). We formulate bootstrap assumptions that control \(W_y-\overline W'\) both near the blow-up center and in the far field, together with bounds for \(W_{yy}\), \(\partial_y^3W\), \(\partial_y^4W\), and the derivatives of \(Z\). The crucial estimate is the weighted bound
\begin{equation*}
(1+y^{2/5})|W_y(y,s)-\overline W'(y)|\leq \frac{6}{13}-\theta,
\end{equation*}
which propagates the \(|y|^{-2/5}\) decay of the slope and ultimately yields the \(C^{3/5}\) cusp. The remaining estimates ensure that the perturbative terms in the \(W\)-equation stay lower order for all self-similar times.

Compared with the analysis for the Camassa--Holm case \cite{KKY}, two additional difficulties arise in the rSV system. The first is the treatment of the nonlocal term \(q=\mathcal I_h^{-1}(hG)\). Since the operator depends on the evolving depth \(h\), the associated kernel is no longer the fixed convolution kernel appearing in scalar models. We therefore construct the Green function for the variable-coefficient operator \(1-\partial_x\circ h^{-1}\partial_x\circ h^3\) and establish pointwise exponential decay bounds for it; see Appendix~\ref{App:K}. This Green function estimate gives the weighted control of \(Q_y\) needed to close the far-field estimates for \(W_y-\overline W'\).

The second difficulty comes from the system structure. Although \(z\) remains \(C^1\)-bounded up to the singular time, its derivatives enter the equations for \(W_y\), \(W_{yy}\), \(\partial_y^3W\), and \(\partial_y^4W\). Direct estimates along the \(Z\)-characteristics do not give decay strong enough to close the bootstrap. To overcome this, we carefully compare the characteristic flows associated with the \(W\)- and \(Z\)-equations. Using the positivity relation \(w-z=4\sqrt h\), we convert certain integrals along \(Z\)-characteristics into integrals along \(W\)-characteristics; see Lemma~\ref{Lemma:Z_high} and Appendix~\ref{App:Z}. This yields the sharper decay estimates for \(\partial_y^nZ\), \(n=2,3,4\), and then the bounds for \(Z_y\) and its weighted far-field norm.

Once the bootstrap estimates are closed, the proof of Theorem~\ref{mainthm} follows from the self-similar scaling. The estimate \(|W_y(y,s)|\lesssim (1+y^{2/5})^{-1}\) gives uniform \(C^{3/5}\) control of \(w\) on bounded sets, while the normalization \(W_y(0,s)=-2\) gives blow-up of every local \(C^\alpha\) seminorm with \( \textstyle \alpha \in (\frac35,1]\) at sets containing \(x_*\). The same scaling gives the temporal rate
\begin{equation*}
[w(\cdot,t)]_{C^\alpha(\Omega)}\sim (T_*-t)^{-(5\alpha-3)/2}
\end{equation*}
whenever \(x_*\in\Omega\) and \( \textstyle \alpha\in(\frac35,1]\). Away from \(x_*\), the far-field decay of \(W_y\) yields bounded \(C^\alpha\) seminorms, and the \(Z_y\)-estimates imply that \(z\) remains uniformly \(C^1\)-bounded up to the blow-up time.

\subsection{Blow-up for the regularized Burgers equation}\label{sec1.5}

We now turn our attention to the regularized Burgers (rB) equation. While some notation may coincide with that used in the analysis of the rSV system, the meaning will be clear from the context.

Let $\textstyle p:= - v_{xt} - \frac{v_x^2}{2} -vv_{xx}$. Then, the rB equation \eqref{rB'} can be rewritten as
\begin{equation} \label{rB}
\begin{split}
& v_t + vv_x = -p_x, \\
& p-p_{xx} = \frac{1}{2}v_x^2.
\end{split}
\end{equation}
We consider the initial value problem for \eqref{rB} with the initial data
\begin{equation} \label{rBic}
v_0(x) := v(x,-\varepsilon) \in H^5(\mathbb{R}) \subset C^4(\mathbb{R}).
\end{equation}
We impose the following assumptions on the initial data:
\begin{equation} \label{rBIC}
\begin{split}
& \partial_x v_0 (0) = - 2 \varepsilon^{-1}, \quad \partial_x^2 v_0(0) = 0, \quad \partial_x^3 v_0(0) = 256\veps^{-6}, \\
& \lVert \partial_x v_0 \rVert_{L^\infty} \leq 2 \varepsilon^{-1}, \quad \lVert \partial_x^2 v_0 \rVert_{L^\infty} \leq 2^4 \veps^{-7/2}, \quad \lVert \partial_x^3 v_0 \rVert_{L^\infty} \leq 2^9 \veps^{-6}, \quad \lVert \partial_x^4 v_0 \rVert_{L^\infty} \leq 2^{13} \veps^{-17/2}, \\
& \left| \varepsilon (\partial_x v_0)(x) - \overline{W}'\left( \frac{x}{\varepsilon^{5/2}} \right) \right| \leq \min{\left\{ \frac{\left( \frac{x}{\varepsilon^{5/2}} \right)^2}{3000\left(1+\left( \frac{x}{\varepsilon^{5/2}} \right)^2 \right)} , \frac{\Theta}{1+ \left( \frac{x}{\varepsilon^{5/2}} \right)^{2/5}} \right\}}, \\
& \limsup_{|x| \to \infty}{|x^{2/5}\partial_x v_0(x)|} \leq \frac{\theta}{2}, \quad \theta := \frac{1}{3} \left( \frac{6}{13} - \Theta \right)
\end{split}
\end{equation}
for a fixed $\textstyle \Theta\in ( 50^{-1/5}, \frac{6}{13})$, where $\overline{W}$ is defined in Section~\ref{sec:HSpro}.

\begin{theorem}\label{mainthm2}
		Fix $\textstyle \Theta\in (50^{-1/5},\frac{6}{13})$. There exists a constant 
		$\ve_0 >0$ such that for each $\veps\in(0,\veps_0)$, if the initial data $v_0$ satisfies \eqref{rBic}--\eqref{rBIC},
		then the initial value problem for \eqref{rB} admits a unique classical solution $v$ satisfying
\begin{equation*}
v \in C([-\varepsilon,T_*); H^5(\mathbb{R})) \cap C^1([-\varepsilon,T_*); H^4(\mathbb{R}))
\end{equation*}
and
\begin{equation*}
v \in C^4([-\varepsilon,T_*) \times \mathbb{R}),
\end{equation*}
		which blows up in the $C^1$-norm at $t=T_\ast$ for some $T_\ast<\infty$. Moreover, the exact same assertions \eqref{no-blow}--\eqref{blow-up-rate1} as in Theorem~\ref{mainthm} also hold for $v$.
\end{theorem}

Appendix~\ref{app:admissible-data} also constructs a non-empty class of data satisfying \eqref{rBic}--\eqref{rBIC}. Remark~\ref{rem_rel} records a $\beta$-dependent variant of this initial-data construction; the statement and proof of Theorem~\ref{mainthm2} below retain the normalization $\beta=1$. The proof of Theorem~\ref{mainthm2} follows a similar bootstrap framework to that in the analysis of the rSV system, but is considerably simpler due to its scalar structure. Moreover, as noted in Section~\ref{Discussion}, the CH equation can be seen as a perturbation of rB in the blow-up regime, and its singularity formation has already been treated in \cite{KKY}. Thus, we provide an outline of the proof in Section~\ref{sec4}.

\subsection{Discussion} \label{Discussion}
In several nonlinear fluid models exhibiting finite-time singularities, blow-up profiles have been observed to attain \( C^{1/3} \) spatial H\"older regularity at the singular time. The regularity \( C^{1/3} \) arises in classical models of fluid dynamics such as the inviscid Burgers, compressible Euler, and Euler--Poisson (EP) equations \cite{BKK1, BKK2, BSV, BSV1}. In contrast, \( C^{3/5} \) regularity has been recently established for the Camassa--Holm (CH) equation \cite{KKY}, and is further confirmed in this work for the regularized Burgers (rB) and regularized Saint--Venant (rSV) equations. It is worth noting that these models differ in the form of conserved energy. The Burgers, Euler, and EP equations conserve an \( L^2 \)-type energy, whereas the CH, rB, and rSV equations, all of which are Hamiltonian systems, conserve an \( H^1 \)-type energy. This alignment between energy structure and observed regularity suggests that the form of conserved energy may constrain the regularity of blow-up profiles.

Indeed, the conservation of an \( H^1 \)-type energy imposes a strict lower bound on spatial regularity. As noted in \cite{KKY}, this follows from a direct application of Morrey’s inequality: any function in \( H^1(\mathbb{R}) \) must be at least locally \( C^{1/2} \). If a blow-up profile were less regular than \( C^{1/2} \), then its \( H^1 \)-norm would diverge near the singularity, contradicting the boundedness of the $H^1$-norm. However, while Morrey’s inequality rules out singularities less regular than \( C^{1/2} \) for the CH, rB, and rSV equations, it does not explain why the specific H\"older exponent \( \textstyle \frac{3}{5} \) appears. To better understand the emergence of this regularity, we first examine the rB equation.

From a structural perspective, the rB equation can be formally viewed as the inviscid Burgers equation modified by a correction involving the derivative of the Hunter--Saxton (HS) equation: 
\begin{equation*}
\underbrace{v_t + vv_x}_{\text{Burgers}} = \underbrace{v_{xxt} + 2v_xv_{xx} + vv_{xxx}}_{\text{Hunter--Saxton}} \tag*{(rB)}.
\end{equation*}
While the Burgers and HS equations conserve, respectively, the \( L^2 \) and \( \dot{H}^1 \) norms, the rB equation conserves the full \( H^1 \)-norm and thus serves as a simple model that captures both \( H^1 \)-conservation and finite-time blow-up. Upon differentiation, the leading-order structure of rB coincides with that of the HS equation, with the remaining term forming a bounded nonlocal correction; see \eqref{rB} together with Section~\ref{sec4}. Note that the HS equation itself exhibits a generic $C^{3/5}$ blow-up profile \cite{KKY}, which rB inherits while conserving the $H^1$-energy. The CH equation may then be viewed as a weakly dispersive perturbation of rB and exhibits the same critical H\"older regularity \( C^{3/5} \) at the blow-up time. We note that the correction terms in rB are motivated by the form of the regularization terms introduced in the rSV system \cite{GJCP}.

Recall that the rSV system is derived as a variational regularization of the classical Saint--Venant (cSV) equations \cite{CD}. In that derivation, the variational formulation determines the regularization ansatz up to two free parameters, and the non-dispersion constraint then identifies $\gamma\mathcal{S}$ in \eqref{rSV1}, thereby fixing the coefficients of the quadratic gradient terms in $\mathcal{S}$. Consequently, this specific choice fixes the leading-order structure in a gradient blow-up regime: differentiating \eqref{rSV'1} with respect to $x$, we obtain
\begin{equation*}
\underbrace{w_{xt} + \frac{1}{2} w_x^2  + w w_{xx}}_{\text{Hunter--Saxton}} = \mathcal{F}(w_x,w_{xx},z,z_x,q_x).
\end{equation*}
In a suitable moving frame, the right-hand side contributes only at lower order in the blow-up regime considered in this paper, so that the blow-up dynamics are governed by the HS equation and the $C^{3/5}$ blow-up profile emerges. By contrast, in an analogous blow-up regime, the cSV system  (equivalently, the isentropic Euler system) reduces at leading order to the inviscid Burgers equation.

These observations, together with the analysis carried out in this work, demonstrate that the Hamiltonian regularizations considered here for the Burgers and Euler equations alter the nature of singularity formation, giving rise to $C^{3/5}$-type singularities. However, it remains an open question whether other types of singular behavior or regularity exponents may arise within the rSV framework under different settings.
~\\

\emph{Plan of the paper}. In Section~\ref{sec2}, we introduce a bootstrap argument in self-similar variables and prove Theorem~\ref{mainthm}. Section~\ref{sec3} is devoted to closing the bootstrap assumptions. In Section~\ref{sec4}, we outline the proof of Theorem~\ref{mainthm2}. Appendix~\ref{app:admissible-data} constructs a non-empty relatively open class of admissible initial data. The remaining appendices collect estimates for the self-similar Hunter--Saxton profile, additional technical material used throughout the paper, and the deferred proof of Lemma~\ref{Lemma:Z_high}.

\section{Stability estimates and Proof of Theorem~\ref{mainthm}}\label{sec2}
In this section, we establish global stability estimates in self-similar time by employing a bootstrap procedure and prove Theorem~\ref{mainthm}.

\subsection{Self-similar variables and modulations} \label{sec2.2}
We define dynamic modulation functions $\tau,\kappa,\xi:[-\varepsilon,T)\to\mathbb{R}$ for some $T>-\varepsilon$, satisfying a system of ODEs:
\begin{equation} \label{modul_xt}
	\begin{split}
		\dot{\tau}(t) & = \frac{(\tau(t) - t)^2 z_x^2(\xi(t),t)}{4} -  \frac{4 (\tau(t) - t) z_x(\xi(t),t)}{3} - \frac{4 (\tau(t) - t)^2 q_x (\xi(t),t)}{3}, \\
		\dot{\kappa}(t) & = \frac{2(\tau(t) - t)^{-1}}{ \partial_x^3 w (\xi(t),t)} \left( z_x(\xi(t),t) z_{xx}(\xi(t),t) - \frac{8 q_{xx}(\xi(t),t)}{3} - \frac{8 (\tau(t) - t)^{-1} z_{xx}(\xi(t),t)}{3} \right) \\
		& \quad  + \frac{8 {G}(\xi(t),t)}{3} - \frac{8 q(\xi(t),t)}{3}, \\
		\dot{\xi}(t) & = - \frac{1}{\partial_x^3 w (\xi(t),t)} \left( z_x(\xi(t),t) z_{xx}(\xi(t),t) - \frac{8 q_{xx}(\xi(t),t)}{3} - \frac{8 (\tau(t) - t )^{-1} z_{xx}(\xi(t),t)}{3} \right) \\
		& \quad + \frac{z(\xi(t),t)}{3} + \kappa(t),
	\end{split}
\end{equation}
with the initial data
\begin{equation}\label{modul_init}
	\tau(-\varepsilon) = 0, \quad \kappa (-\varepsilon) = \kappa_0, \quad \xi(-\varepsilon) =0,
\end{equation}
where $\kappa_0 := w_0(0)$. By Lemma~\ref{existence} in Section~\ref{sec:boot} and standard ODE theory, the initial value problem for \eqref{modul_xt} with the initial data \eqref{modul_init} admits a unique local $C^1$ solution $(\tau,\kappa,\xi)$. Define $T_*$ as the first time such that $\tau(t)=t$, i.e.,
\begin{equation}\label{T-star}
	T_* := \inf\{ t \in [-\ve,\infty) : \tau(t) = t\}.
\end{equation}
In the proof of Theorem~\ref{mainthm}, we show that $T_*$ is well-defined and is the finite blow-up time of $\partial_x w$.

We introduce the self-similar variables
\begin{equation}\label{ys}
	y(x,t) = \frac{x- \xi(t)}{(\tau(t)- t)^{5/2}}, \quad s(t) = - \log{(\tau(t) - t)},
\end{equation}
and define new functions $W(y,s)$, $Z(y,s)$, $Q(y,s)$, and $\tilde{G}(y,s)$ as
\begin{equation}\label{WZQG}
	w(x,t) - \kappa(t) = e^{-3s/2 }W(y,s), \quad z(x,t) = Z(y,s), \quad q(x,t) = Q(y,s), \quad G(x,t) = \tilde{G}(y,s).
\end{equation}
By a straightforward calculation, one has
\begin{equation}\label{y-x}
\partial_{t} = \left( - \dot{\xi} e^{5s/2} + \frac{5}{2} y (1-\dot{\tau}) e^s \right) \partial_y + (1-\dot{\tau})e^s \partial_s, \quad \partial_x = e^{5s/2} \partial_y.
\end{equation}
Using the change of variables \eqref{ys}--\eqref{WZQG} in \eqref{rSV'}, we obtain
\begin{subequations}
\begin{align}
\label{eq:W} & W_s - \frac{3}{2}W + U^W W_y = - \frac{e^{s/2}\dot{\kappa}}{1-\dot{\tau}} + \frac{8e^{s/2}(\tilde{G}-Q)}{3(1-\dot{\tau})}, \\
\label{eq:Z} & Z_s + U^Z Z_y = \frac{8 e^{-s} (\tilde{G}-Q)}{3(1-\dot{\tau})} ,
\end{align}
\end{subequations}
where
\begin{subequations}
\begin{align}
\label{UW} U^W (y,s) & :=   \frac{5}{2} y +  \frac{W}{1-\dot{\tau}} + \frac{e^{3s/2}Z}{3(1-\dot{\tau})} + \frac{ e^{3s/2}(\kappa - \dot{\xi})}{1-\dot{\tau}}, \\
\label{UZ} U^Z (y,s) & := \frac{5}{2} y + \frac{W}{3(1-\dot{\tau})} + \frac{e^{3s/2}Z}{1-\dot{\tau}} +  \frac{e^{3s/2} ( \kappa - 3 \dot{\xi})}{3(1-\dot{\tau})}.
\end{align}
\end{subequations}
Applying $\partial_y^n$ to \eqref{eq:W} for $n=1,2,3,4$, we have
\begin{subequations}
	\begin{align}
		& \label{Wy1} \left( \partial_s + 1 + \frac{W_y}{2(1-\dot{\tau})} - \frac{4e^{3s/2} Z_y }{3(1-\dot{\tau})} \right) W_y +  U^W W_{yy} = -\frac{8 e^{s/2} Q_y}{3(1-\dot{\tau})} + \frac{e^{3s} Z_y^2}{2(1-\dot{\tau})} , \\
		& \label{Wy2} \left( \partial_s + \frac{7}{2} + \frac{2W_y}{1-\dot{\tau}} - \frac{e^{3s/2} Z_y }{1-\dot{\tau}} \right) W_{yy} +  U^W  \partial_y^3 W = - \frac{8 e^{s/2} Q_{yy}}{3(1-\dot{\tau})} + \frac{e^{3s} Z_y Z_{yy}}{1-\dot{\tau}}  + \frac{4e^{3s/2} Z_{yy} W_y}{3(1-\dot{\tau})}, \\
		& \label{Wy3} \left( \partial_s + 6 + \frac{3W_y}{1-\dot{\tau}} - \frac{2e^{3s/2} Z_y }{3(1-\dot{\tau})} \right) \partial_y^3 W +  U^W  \partial_y^4 W = - \frac{8 e^{s/2} \partial_y^3 Q}{3(1-\dot{\tau})} - \frac{2 W_{yy}^2}{1-\dot{\tau}} \\
		& \nonumber \quad  + \frac{7e^{3s/2} W_{yy}Z_{yy} }{3(1-\dot{\tau})} + \frac{4e^{3s/2} W_y \partial_y^3 Z}{3(1-\dot{\tau})} + \frac{e^{3s} Z_{yy}^2}{1-\dot{\tau}} + \frac{e^{3s} Z_y \partial_y^3 Z}{1-\dot{\tau}}  , \\
		& \label{Wy4} \left( \partial_s + \frac{17}{2} + \frac{4W_y}{1-\dot{\tau}} - \frac{e^{3s/2} Z_y }{3(1-\dot{\tau})} \right) \partial_y^4 W +  U^W  \partial_y^5 W = - \frac{8 e^{s/2} \partial_y^4 Q}{3(1-\dot{\tau})}  - \frac{7 W_{yy} \partial_y^3 W}{1-\dot{\tau}} \\
		& \nonumber \quad   + \frac{3e^{3s/2}  \partial_y^3 W Z_{yy}}{1-\dot{\tau}} + \frac{11e^{3s/2} W_{yy} \partial_y^3 Z}{3(1-\dot{\tau})}  + \frac{4e^{3s/2} W_y \partial_y^4 Z}{3(1-\dot{\tau})} + \frac{3e^{3s} Z_{yy} \partial_y^3 Z}{1-\dot{\tau}} + \frac{e^{3s} Z_y \partial_y^4 Z}{1-\dot{\tau}}.
	\end{align}
\end{subequations}
Similarly, we apply $\partial_y^n$ to \eqref{eq:Z} for $n=1,2,3,4$ to obtain
\begin{equation} \label{eq:Z_high}
\partial_s \partial_y^n Z + D_n^Z \partial_y^n Z + U^Z \partial_y^{n+1} Z = F_n^Z, \quad n=1,2,3,4,
\end{equation}
where
\begin{equation} \label{D^Z}
\begin{aligned}
D_1^Z (y,s) &:= \frac{5}{2}, \\
D_3^Z(y,s) &:= \frac{15}{2} + \frac{3e^{3s/2} Z_y}{1-\dot{\tau}} - \frac{2W_y}{3(1-\dot{\tau})},
\end{aligned}
\qquad
\begin{aligned}
D_2^Z(y,s) &:= 5 + \frac{2e^{3s/2} Z_y}{1-\dot{\tau}} - \frac{W_y}{1-\dot{\tau}}, \\
D_4^Z(y,s) &:= 10 + \frac{4e^{3s/2} Z_y}{1-\dot{\tau}} - \frac{W_y}{3(1-\dot{\tau})},
\end{aligned}
\end{equation}
and
\begin{equation} \label{F^Z}
\begin{split}
F_1^Z(y,s) & := - \frac{8e^{-s}Q_y}{3(1-\dot{\tau})} + \frac{e^{-3s/2}W_y^2}{2(1-\dot{\tau})} + \frac{4W_yZ_y}{3(1-\dot{\tau})} - \frac{e^{3s/2}Z_y^2}{2(1-\dot{\tau})}   , \\
F_2^Z(y,s) & := - \frac{8 e^{-s} Q_{yy}}{3(1-\dot{\tau})} + \frac{e^{-3s/2}W_y W_{yy}}{1-\dot{\tau}} + \frac{4W_{yy}Z_y}{3(1-\dot{\tau})}, \\
F_3^Z(y,s) & := - \frac{8 e^{-s} \partial_y^3 Q}{3(1-\dot{\tau})}  + \frac{e^{-3s/2} W_{yy}^2}{1-\dot{\tau}} + \frac{e^{-3s/2}W_y \partial_y^3 W}{1-\dot{\tau}} + \frac{4 \partial_y^3 W Z_y }{3(1-\dot{\tau})} + \frac{7W_{yy}Z_{yy}}{3(1-\dot{\tau})} - \frac{2e^{3s/2}Z_{yy}^2}{1-\dot{\tau}}, \\
F_4^Z(y,s) & := - \frac{8 e^{-s} \partial_y^4 Q}{3(1-\dot{\tau})}  + \frac{3e^{-3s/2} W_{yy} \partial_y^3 W}{1-\dot{\tau}} + \frac{e^{-3s/2}W_y \partial_y^4 W}{1-\dot{\tau}}  \\
& \quad + \frac{11  \partial_y^3 W Z_{yy} }{3(1-\dot{\tau})} + \frac{3W_{yy} \partial_y^3 Z}{1-\dot{\tau}} + \frac{4  \partial_y^4 W Z_y}{3(1-\dot{\tau})} - \frac{7e^{3s/2}Z_{yy} \partial_y^3 Z}{1-\dot{\tau}}.
\end{split}
\end{equation}

We observe that, under the assumption \eqref{init_w0}, $(\tau,\kappa,\xi)(t)$ solves the ODE system \eqref{modul_xt}--\eqref{modul_init} if and only if the following constraints hold for all $s\ge s_0$, where $s_0:=s(-\varepsilon)$:
\begin{equation}\label{constraint}
	W(0,s)=0,\quad W_y(0,s)=-2,\quad W_{yy}(0,s)=0.
\end{equation}
By the change of variables \eqref{ys} and the definition \eqref{WZQG}, the system \eqref{modul_xt} can be rewritten as
\begin{subequations}\label{modul}
	\begin{align}
		\dot{\tau} & = - \frac{4 e^{s/2} Q_y(0,s)}{3} + \frac{e^{3s} Z_y^2(0,s)}{4} -  \frac{4e^{3s/2} Z_y(0,s) }{3} , \label{tau}\\
		\dot{\kappa} & = \frac{2 e^{-s/2}}{\partial_y^3 W (0,s)} \left( - \frac{8 e^{s/2} Q_{yy}(0,s)}{3}+e^{3s} Z_y(0,s) Z_{yy}(0,s)  - \frac{8e^{3s/2} Z_{yy}(0,s)}{3} \right) + \frac{8 (\tilde{G}-Q)(0,s)}{3}, \label{kap}\\
		\dot{\xi} & = - \frac{e^{-3s/2}}{\partial_y^3 W (0,s)} \left( - \frac{8 e^{s/2} Q_{yy}(0,s)}{3} + e^{3s} Z_y(0,s) Z_{yy}(0,s)  - \frac{8e^{3s/2} Z_{yy}(0,s)}{3} \right) + \frac{Z(0,s)}{3} + \kappa(t). \label{xi}
	\end{align}
\end{subequations}
Evaluating \eqref{eq:W}, \eqref{Wy1}, and \eqref{Wy2} at $y=0$ and using \eqref{modul}, we obtain a closed ODE system for $(W(0,s),W_y(0,s),W_{yy}(0,s))$, for which $(0,-2,0)$ is a solution. Since the initial condition \eqref{init_w0} yields \eqref{constraint} at $s=s_0$, uniqueness for this ODE system implies \eqref{constraint} for all $s \geq s_0$. Conversely, if \eqref{constraint} holds for all $s \geq s_0$, then $\partial_s W(0,s)=\partial_s W_y(0,s)=\partial_s W_{yy}(0,s)=0$. Substituting these relations into \eqref{eq:W}, \eqref{Wy1}, and \eqref{Wy2} evaluated at $y=0$ recovers \eqref{modul}, and hence \eqref{modul_xt}.

\subsection{Bootstrap argument} \label{sec:boot}

We first state the local existence of solutions to \eqref{rSV'}.

\begin{lemma}[Local existence and blow-up criterion] \label{existence}
Let the initial data $(w_0,z_0)$ satisfy \eqref{init_h}--\eqref{IC}. Then there exists a constant $T_0>-\veps$ such that the initial value problem for \eqref{rSV'} has a unique solution $(w,z)(x,t)$ satisfying
\begin{equation*}
(w-2\sqrt{h_*},z+2\sqrt{h_*})\in C([-\veps,T_0];H^5(\mathbb{R}))\cap C^1([-\veps,T_0];H^4(\mathbb{R}))
\end{equation*}
and
\begin{equation*}
\inf_{x \in \mathbb{R}}{(w-z)(x,t)} >0 \quad \text{for all } t \in [-\varepsilon,T_0].
\end{equation*}
Let $[-\varepsilon,T_{\mathrm{max}})$ be the maximal interval of existence of the solution. If $(w_0,z_0)$ satisfies \eqref{E0_bound} and $T_{\mathrm{max}}< \infty$, then
\begin{equation*}
\limsup_{t \to T_{\mathrm{max}}^-} {\lVert (w_x,z_x)(\cdot,t) \rVert_{L^\infty}} = \infty.
\end{equation*}
\end{lemma}

The first assertion in this lemma follows from the local existence result for \eqref{rSV1} in \cite[Theorem~3.1]{LPP}, together with the equivalence
\begin{equation*}
\inf_{x\in\mathbb{R}} h(x,t)>0 \iff \inf_{x\in\mathbb{R}} (w-z)(x,t)>0
\end{equation*}
and the fact that the change of variables \eqref{wz} preserves $H^k$ regularity for $k=1,5$ as long as $h$ is bounded away from zero. The blow-up criterion follows from \cite[Corollary~4.2]{LPP}.

We introduce the bootstrap assumptions on the self-similar time interval $[s_0,\sigma_1]$, where $\sigma_1>s_0$ will be specified in Remark~\ref{bootstrap_start}. For all $s\in[s_0,\sigma_1]$ (with the corresponding time $t\in[-\varepsilon,T_{\sigma_1}]$; see Remark~\ref{T_sigma} for the definition of $T_{\sigma_1}$) and all $y\in\mathbb{R}$, we assume that $(W,Z)$ and $\tau$ satisfy:
\begin{subequations}\label{boots_ass}
\begin{align}
&  \lvert \dot{\tau} (t)\rvert  \leq \frac{8}{\sqrt{h_*}} \varepsilon, \label{dottau} \\
& \lVert Z_{y} (\cdot,s) \rVert_{L^\infty} \leq \frac{5}{\sqrt{h_*}} e^{-5s/2}, \label{lem_Zy} \\
& |Z_y(y,s)| \leq \frac{10}{\sqrt{h_*}}\frac{e^{-3s/2}}{1+y^{2/5}},  \label{Zy_dec} \displaybreak[2] \\
& \lvert W_y(y,s) - \overline{W}'(y) \rvert \leq \frac{y^2}{1000(1+y^2)}, \label{Wy_bound} \\
& \lvert W_y(y,s) - \overline{W}'(y) \rvert \leq \frac{6}{13(1+y^{2/5})}, \label{Wy_dec}\\
& \lvert W_{yy}(y,s) \rvert \leq \frac{M^{1/8}\lvert y \rvert}{(1+y^2)^{1/2}}, \label{Wyy_bound}\\
& \lvert \partial_y^3 W(0,s) - 256 \rvert \leq 1, \label{Wy3_0}\\
& \lVert \partial_y^3 W(\cdot,s) \rVert_{L^\infty} \leq M^{3/4}, \label{Wy3_bound}\\
& \lVert \partial_y^4 W(\cdot,s) \rVert_{L^\infty} \leq M,\label{Wy4_bound}
\end{align}
\end{subequations}
where $M>0$ is a sufficiently large constant to be chosen later.

The bootstrap assumptions \eqref{Wy_bound}--\eqref{Wy4_bound} imply that $W_y$ is close to $\overline{W}'$ near $y=0$ in $C^2$-norm, consistent with the constraints \eqref{constraint} at the origin, and that $W_y$ remains uniformly bounded in $C^3$-norm. The asymptotic behavior of $W$ near $y=0$ plays an important role in resolving the degeneracy of the damping term there in the transport-type equations arising in the closure of the bootstrap argument.

\begin{remark} [Initial conditions for $(W,Z)$] \label{init_rmk}
We verify that, under the assumptions \eqref{init_w0}--\eqref{init_wx_weight}, \eqref{init_zx_weight}, and \eqref{E0_bound} on $(w_0,z_0)$, the corresponding initial data $(W,Z)(\cdot,s_0)$ and $\dot{\tau}(-\varepsilon)$ satisfy the bootstrap assumptions \eqref{dottau}--\eqref{Wy4_bound} at $s=s_0$.

First, under \eqref{init_w0}--\eqref{init_wx_weight} and \eqref{init_zx_weight}, the functions $W$ and $Z$ satisfy the following bounds in the self-similar variables:
\begin{subequations}\label{EP-W-IC}
\begin{align}
		& \lVert Z_{y} (\cdot,s_0) \rVert_{L^\infty} \leq \frac{1}{\sqrt{h_*}} \veps^{5/2}, \label{lem_Zy_init} \\
& |Z_y(y,s_0)| \leq \frac{1}{\sqrt{h_*}}\frac{\veps^{3/2}}{1+y^{2/5}}, \label{Zy_weight_init} \\
& |W_y(y,s_0)-\overline W'(y)| \leq \frac{y^2}{3000(1+y^2)}, \label{W-y2} \\
&|W_y(y,s_0)-\overline{W}'(y)| \le \frac{\Theta}{1+y^{2/5}},\label{Utildey_init} \\
&|\partial_y^3 W(0, s_0)-256| = 0,  \label{1D2} \\
&\|\partial_y^4W(\cdot,s_0)\|_{L^\infty}\leq 2^{13}. \label{1D4}
\end{align}
\end{subequations}
We then apply the 1D Gagliardo--Nirenberg inequality and use \eqref{init_wbound} to obtain
\begin{equation*}
\begin{split}
\|\partial_x^2w_0\|_{L^\infty}
&\leq \sqrt{2}\|\partial_x^2w_0\|_{L^2}^{1/2}
\|\partial_x^3w_0\|_{L^2}^{1/2}
\leq 2^{9/2}\varepsilon^{-7/2},\\
\|\partial_x^3w_0\|_{L^\infty}
&\leq \sqrt{2}\|\partial_x^3w_0\|_{L^2}^{1/2}
\|\partial_x^4w_0\|_{L^2}^{1/2}
\leq 2^9\varepsilon^{-6},
\end{split}
\end{equation*}
or, equivalently,
\begin{subequations}
	\begin{align}
		&\|W_{yy}(\cdot,s_0)\|_{L^\infty}\leq 2^{9/2}, \label{1D3} \\
		&\|\partial_y^3W(\cdot,s_0)\|_{L^\infty}\leq 2^9. \label{1D5} 
	\end{align}
\end{subequations}
A Taylor expansion of $W_{yy}(y,s_0)$ about $y=0$, together with \eqref{1D2}--\eqref{1D4}, yields
\begin{equation*}
|W_{yy}(y,s_0)|\leq|y||\partial_y^3W(0,s_0)|+\frac{y^2}{2}\|\partial_y^4W(\cdot,s_0)\|_{L^{\infty}}\leq 256|y|+2^{12}y^2. 
\end{equation*}
Combining this with \eqref{1D3}, we obtain
\begin{equation}\label{1D3'}
|W_{yy}(y,s_0)|\leq \min\left\{ 256|y|+2^{12}y^2 , 2^{9/2} \right\}\leq \frac{M^{1/8}|y|}{4(1+y^2)^{1/2}},\quad y \in \mathbb{R}
\end{equation}
for sufficiently large $M>0$. Finally, we apply \eqref{GQ_bound} from Lemma~\ref{lem_Gq} and \eqref{lem_Zy_init} to \eqref{tau} to obtain
\begin{equation}\label{tau_init}
	|\dot{\tau}(-\veps)|\leq C e^{s_0/2}|Q_y(0,s_0)| + C e^{3s_0}|Z_y(0,s_0)|^2 + \frac{4e^{3s_0/2}}{3} |Z_y(0,s_0)|\leq C\veps^2 + \frac{4}{3\sqrt{h_*}}\veps \leq \frac{3}{2\sqrt{h_*}} \veps
\end{equation}
for sufficiently small $\varepsilon>0$. From \eqref{EP-W-IC}--\eqref{tau_init}, we conclude that the initial data $(W,Z)(y, s_0)$ and $\dot{\tau}(-\varepsilon)$ satisfy \eqref{dottau}--\eqref{Wy4_bound}, provided $M>0$ is sufficiently large and $\varepsilon>0$ is sufficiently small.
\end{remark}

\begin{remark}[The initial bootstrap interval] \label{bootstrap_start}
Using the strict initial bounds established in Remark~\ref{init_rmk}, we show that the bootstrap estimates \eqref{boots_ass} hold on $[s_0,\sigma_1]$ for some $\sigma_1>s_0$.  By Lemma~\ref{existence} and the continuity of the modulations, the bounds \eqref{dottau}--\eqref{lem_Zy} and \eqref{Wy3_0}--\eqref{Wy4_bound} hold on some $[s_0,\sigma]$. Moreover, using \eqref{constraint} and Taylor's formula near $y=0$, one can verify that \eqref{Wy_bound} and \eqref{Wyy_bound} also hold on a sufficiently short time interval.

On $|y|\leq1$, the estimates \eqref{Zy_dec} and \eqref{Wy_dec} follow directly from \eqref{lem_Zy} and \eqref{Wy_bound}, respectively. To treat $|y|\geq1$, we multiply the equations for $W_y-\overline W'$ and $e^{3s/2}Z_y$, obtained from \eqref{Wy1} and \eqref{eq:Z_high} with $n=1$, by $\min\{1+|y|^{2/5},R\}$, where $R\geq1$. We use the local bounds from Lemma~\ref{existence} and the decay of $\overline W'$ and $\overline W''$, and estimate the terms involving $Q_y$ using the Green function bound in Appendix~\ref{App:K}. Integrating along the characteristic curves associated with $U^W$ and $U^Z$ and applying Gronwall's inequality, we obtain
\begin{equation*}
\begin{split}
& \|\min\{1+|y|^{2/5},R\}(W_y(\cdot,s)-\overline W')\|_{L^\infty} \leq\Theta+C(s-s_0),\\
& e^{3s/2}\|\min\{1+|y|^{2/5},R\}Z_y(\cdot,s)\|_{L^\infty} \leq\frac{1}{\sqrt{h_*}}+C(s-s_0)
\end{split}
\end{equation*}
for $s\in[s_0,\sigma]$, where $C>0$ is independent of $R$. For similar calculations with the weight $1+|y|^{2/5}$, see the proof of Lemma~\ref{mainprop_1-p} for the $W$-equation and the paragraph preceding Lemma~\ref{lem_Zyweight} for the $Z$-equation. Letting $R\to\infty$ and using $\Theta<6/13$, we obtain \eqref{Zy_dec} and \eqref{Wy_dec} on $[s_0,\sigma_1]$ for some $\sigma_1\in(s_0,\sigma]$ sufficiently close to $s_0$.
\end{remark}

\begin{remark} \label{T_sigma}
We collect several bounds derived from \eqref{boots_ass} that will be used in the subsequent analysis. For $\varepsilon>0$ sufficiently small, \eqref{dottau} yields
	\begin{equation} \label{dottau1}
		\frac{1}{1-\dot{\tau}} = 1+ \frac{\dot{\tau}}{1-\dot{\tau}} \leq 1 + \varepsilon^{1/2}.
	\end{equation}
Moreover, \eqref{dottau} implies that $\tau(t)-t$ is strictly decreasing, and hence there exists a unique $T_{\sigma_1}$ such that $e^{-\sigma_1} = \tau(T_{\sigma_1}) - T_{\sigma_1}$. By \eqref{modul_init}, \eqref{ys}, and \eqref{dottau}, we obtain
\begin{equation} \label{Tsigma1}
T_{\sigma_1} = \tau(T_{\sigma_1}) - e^{-\sigma_1} \leq \left| \int_{-\varepsilon}^{T_{\sigma_1}} \dot{\tau}(t) \, dt \right| = \left| \int_{s_0}^{\sigma_1} \dot{\tau}(t) \frac{e^{-s}}{1-\dot{\tau}} \, ds \right| \leq C \varepsilon \int_{s_0}^{\sigma_1} e^{-s} \, ds \leq C \varepsilon.
\end{equation}	
Combining \eqref{4.1} with \eqref{Wy_bound}, we deduce a uniform bound for $W_y$:
	\begin{equation}\label{Uy1}
		|W_y(y,s)|\leq |\overline{W}'(y)|+|W_y(y,s)-\overline{W}'(y)|\leq 2-\frac{6y^2}{5(1+y^2)}+\frac{y^2}{1000(1+y^2)}\leq 2.
	\end{equation}
\end{remark}

We now close the bootstrap argument.

\begin{proposition}\label{Boot}
Fix $\textstyle \Theta\in (50^{-1/5},\frac{6}{13})$. There exist constants $\varepsilon_0>0$, $M>0$, and $C>0$ such that for each $\varepsilon\in(0,\varepsilon_0)$ the following statement holds.

Let $T_0>-\varepsilon$. Suppose that $(w,z)$ is the unique solution to \eqref{rSV'} with initial data satisfying \eqref{init_h}--\eqref{init_zx_weight} and \eqref{E0_bound} on $[-\varepsilon,T_0]$, and that $(\tau,\kappa,\xi)$ is the unique solution to \eqref{modul_xt}--\eqref{modul_init} on $[-\varepsilon,T_0]$ with $\tau(T_0)>T_0$. Fix $\sigma_1\in(s_0,s(T_0)]$, where $s_0:=s(-\varepsilon)$, and let $T_{\sigma_1}$ be the corresponding time as in Remark~\ref{T_sigma}. If the functions $(W,Z)$ defined by \eqref{ys}--\eqref{WZQG} and the modulation $\tau$ satisfy the assumptions \eqref{boots_ass} for all $s\in[s_0,\sigma_1]$ and $y\in\mathbb{R}$, then we have
\begin{equation} \label{dottau_close}
|\dot{\tau}(t)| \leq \frac{7}{\sqrt{h_*}} e^{-s}
\end{equation}
for all $t\in[-\varepsilon,T_{\sigma_1}]$, and
\begin{subequations} \label{boots_close}
\begin{align}
& \| Z_{y}(\cdot, s) \|_{L^\infty} \leq \frac{4}{\sqrt{h_*}} e^{-5s/2}, \label{lem_Zy_close} \\
& |Z_y(y,s)| \leq \frac{8}{\sqrt{h_*}}\frac{ e^{-3s/2}}{1 + y^{2/5}}, \label{Zyweight_close} \displaybreak[2] \\
& |W_y(y,s) - \overline{W}'(y)| \leq \frac{y^2}{1500(1 + y^2)}, \label{Utildey_close} \\
& |W_y(y,s) - \overline{W}'(y)| \leq \frac{\frac{6}{13} - \theta}{1 + y^{2/5}}, \label{Utildey_M-str} \displaybreak[2] \\
& |W_{yy}(y,s)| \leq \frac{M^{1/8} |y|}{2(1 + y^2)^{1/2}}, \label{Wy2_close} \\
& |\partial_y^3 W(0,s) - 256| \leq C \varepsilon, \label{Wy30_close} \\
& \| \partial_y^3 W(\cdot, s) \|_{L^\infty} \leq \frac{M^{3/4}}{2}, \label{Wy3_close} \\
& \| \partial_y^4 W(\cdot, s) \|_{L^\infty} \leq \frac{M}{2}. \label{Wy4_close}
\end{align}
\end{subequations}
for all $s\in[s_0,\sigma_1]$ and $y\in\mathbb{R}$, where $\theta>0$ is defined in \eqref{init_wx_dec}.
\end{proposition}

\begin{proof}
The proof is deferred to Section~\ref{sec3}, where it is decomposed into a sequence of lemmas. The desired estimates are established through Lemmas~\ref{lem:tau_decay}--\ref{mainprop_1-p}. For the reader’s convenience, we provide the following list: \eqref{dottau_close} is obtained in Lemma~\ref{lem:tau_decay}, \eqref{lem_Zy_close} in Lemma~\ref{lem:Zy}, \eqref{Zyweight_close} in Lemma~\ref{lem_Zyweight}, \eqref{Utildey_close} in Lemma~\ref{lem:Wtildey}, \eqref{Utildey_M-str} in Lemma~\ref{mainprop_1-p}, \eqref{Wy2_close} in Lemma~\ref{Wy2_lem}, \eqref{Wy30_close} in Lemma~\ref{lem:Wy30}, \eqref{Wy3_close} in Lemma~\ref{lem:Wy3}, and \eqref{Wy4_close} in Lemma~\ref{lem:Wy4}.
\end{proof}

\subsection{Proof of Theorem~\ref{mainthm}}\label{C13_subsec-p}
We split the proof into several steps.

	Step 1:
We first establish the existence and the bound \eqref{z_result}. Under the initial assumptions \eqref{init_h}--\eqref{init_zx_weight} and \eqref{E0_bound}, Lemma~\ref{existence} and Remarks~\ref{init_rmk}--\ref{T_sigma} imply that there exists $\sigma_1>s_0$ such that the rSV system \eqref{rSV'} has a unique solution $(w,z)$ on $[-\varepsilon,T_{\sigma_1}]$ satisfying
\begin{equation*}
(w-2\sqrt{h_*}, z+2\sqrt{h_*}) \in C([-\varepsilon,T_{\sigma_1}]; H^5(\mathbb{R})) \cap C^1([-\varepsilon,T_{\sigma_1}]; H^4(\mathbb{R})),
\end{equation*}
the modulation $(\tau,\kappa,\xi)$ is well-defined on $[-\varepsilon,T_{\sigma_1}]$ with $\tau$ satisfying \eqref{dottau}, and the corresponding functions $(W,Z)$ satisfy the bootstrap assumptions \eqref{lem_Zy}--\eqref{Wy4_bound} on $[s_0,\sigma_1]$. We then obtain the improved bootstrap bounds \eqref{dottau_close}--\eqref{boots_close} by Proposition~\ref{Boot}. By these bounds and a standard continuation argument (cf. Lemma~\ref{existence}), the bootstrap interval extends to all $s\geq s_0$. Since $|\dot{\tau}|<1/2$ for sufficiently small $\varepsilon>0$, we have $0<dt/ds=e^{-s}/(1-\dot{\tau})\leq2e^{-s}$, so $t(s)$ has a finite limit as $s\to\infty$. This limit is $T_*$ in \eqref{T-star}, since $\tau(t(s))-t(s)=e^{-s}\to0$. Hence, the solution $(w,z)$ and the modulations $(\tau,\kappa,\xi)$ exist on $[-\varepsilon,T_*)$, while all the bootstrap assumptions remain valid. By \eqref{constraint}, $w_x(\xi(t),t)=-2e^s\to-\infty$ as $t\nearrow T_*$, so $T_*$ is the blow-up time of $\partial_x w$. We therefore obtain \eqref{sol_reg}, from which \eqref{sol_reg1} follows by Sobolev embedding and the structure of the rSV system \eqref{rSV'}. Finally, from \eqref{lem_Zy_close} and \eqref{WZbound}, we deduce that
\begin{equation*}
\|z(\cdot,t)\|_{C^1} = \|z(\cdot,t)\|_{L^{\infty}}+e^{5s/2}\|Z_y(\cdot,s)\|_{L^{\infty}}\leq C.
\end{equation*}
This proves $\textstyle \sup_{t<T_*} \left[ z(\cdot, t) \right]_{C^1(\mathbb{R})}<\infty$.

		Step 2: We prove that $\textstyle \sup_{t<T_*} \left[ w(\cdot, t) \right]_{C^{3/5}(\Omega)}<\infty$ for any bounded open set $\Omega\subset \mathbb{R}$. By Step~1, the estimate \eqref{Utildey_M-str} in Proposition~\ref{Boot} holds globally in $s$, i.e.,  
		\begin{equation*}
			\sup_{y \in\mathbb{R}}  \left( (y^{2/5}+1)|W_y(y,s)-\overline{W}'(y)| \right) \le \frac{6}{13}-\theta \quad  \text{ for all } s\ge s_0.
		\end{equation*}
		From this together with \eqref{4.0} and \eqref{4.0'}, it follows that
		\begin{equation}\label{Wydec_fin_2}
			|W_y (y,s) | \le | W_y (y, s) - \UU'(y) | + | \UU'(y) | \le \frac{C}{1+y^{2/5}}+|\overline{W}'(y) |\leq \frac{C}{1+y^{2/5}}
		\end{equation}
		for all $s\ge s_0$ and $y\in\mathbb{R}$, where $C>0$ is a generic constant.
		Using this, for any $y, \wt{y}\in\mathbb{R}$ and for all $s\ge s_0$, we have 
		\begin{equation}\label{not-zero}
			{\frac{|W(y,s)-W(\wt{y},s)|}{{|y-\wt{y}|}^{3/5}}} =  \frac{1}{{|y-\wt{y}|}^{3/5}} {\left|\int_{\wt{y}}^y W_y(\hat{y},s) d\hat{y} \right|} \le   \frac{C}{{|y-\wt{y}|}^{3/5}}  \left| \int_{\wt{y}}^{y}{(1+\hat{y}^2)^{-1/5}\,d\hat{y}} \right| \lesssim 1,
		\end{equation} 
		where the estimate is uniform in $y, \wt{y} \in \mathbb{R}$ and $s\ge s_0$. 
		Consider any two points $x\neq \wt{x} \in \Omega \subset \mathbb{R}$. 
		By the change of variables \eqref{ys}--\eqref{WZQG} together with \eqref{not-zero}, we have 
		\begin{equation*}\label{Eq_holder}
			\frac{|w(x,t) - w(\tilde x, t ) | }{ | x - \tilde x|^{3/5}} =   \frac{|W(y,s)-W(\wt{y},s)|}{{|y-\wt{y}|}^{3/5}} \lesssim 1
		\end{equation*}
		for all $t\in[-\ve, T_*)$. This yields the desired bound.

			Step 3:
			From Step~1, we find that $\xi(t)$ and $\kappa(t)$ converge to $\xi(T_*)\in\mathbb{R}$ and $\kappa(T_*)\in\mathbb{R}$, respectively, as $t\nearrow T_*$. We let $x_* : = \xi(T_*)$. In this step, we prove \eqref{blow-up-result1}, i.e., $\textstyle \lim_{t\nearrow T_*} \left[ w (\cdot, t) \right]_{C^{\alpha}(\Omega)} = \infty$ if $x_* \in \Omega$; and $ \textstyle \limsup_{t\nearrow T_*} \left[ w (\cdot, t) \right]_{C^{\alpha}(\Omega)} < \infty$ if $x_* \notin \overline{\Omega}$ for any $\textstyle \alpha\in(\frac35,1]$ and for any bounded open set $\Omega$.
	
			We note from  \eqref{ys} and \eqref{WZQG} that   
			\begin{equation*}
				\frac{|w(x,t)-w(\wt{x} ,t)|}{|x-\wt{x} |^\alpha}=  e^{(\frac{5}{2}\alpha-\frac{3}{2})s}\frac{|W(y,s)-W(\wt{y},s)|}{|y-\wt{y}|^{\alpha}}. 
			\end{equation*} 
			Let $\wt{y}=0$ and fix $y \in(-\eta_1,\eta_1)\setminus\{0\}$ for some $0<\eta_1 \ll1$. 
			By the mean value theorem,
			\begin{equation*}
				\frac{|W(y,s)-W(0,s)|}{|y|^{\alpha}}=|W_y(\overline{y},s)||y|^{1-\alpha} 
			\end{equation*}
			for some $\overline{y}\in (-\eta_1,\eta_1)$.
			From \eqref{constraint} and \eqref{Wyy_bound}, we see that $\textstyle |W_y(\overline{y},s)| \geq \frac12$ for $\eta_1\ll1$, which implies that $|W_y(\overline{y},s)||y|^{1-\alpha}\ge c_0$ for some $c_0>0$. 
			For this fixed $y$, let $x_y(t):=\xi(t)+e^{-5s/2}y$ and $\widetilde{x}_y(t):=\xi(t)$ be the corresponding points in the original variables. If $x_*\in\Omega$, then $x_y(t),\widetilde{x}_y(t)\to x_*$ as $t\nearrow T_*$. Since $\Omega$ is open, both points belong to $\Omega$ for all $t$ sufficiently close to $T_*$. Therefore,
			\begin{equation}\label{LB-blow}
				\begin{split} 
					[w(\cdot, t)]_{C^{\alpha}(\Omega)}
					& \geq    e^{(\frac{5}{2}\alpha-\frac{3}{2})s}  \frac{|W(y,s)-W(0,s)|}{|y|^{\alpha}} 
					\\
					&=  e^{(\frac{5}{2}\alpha-\frac{3}{2})s} |W_y(\overline{y},s)||y|^{1-\alpha}  \ge  c_0     e^{(\frac{5}{2}\alpha-\frac{3}{2})s}. 
				\end{split}
			\end{equation}
			Since $\textstyle \alpha>\frac35$, we see that $[w(\cdot, t)]_{C^{\alpha}(\Omega)}  \to \infty$ as $t\nearrow T_*$, i.e., $s\to \infty$. 
			From this, we conclude that $x_* =\xi(T_*)$ is the blow-up location.

				On the other hand, let $\Omega$ be a bounded open set such that $x_*\notin\overline{\Omega}$. We prove that $\textstyle \limsup_{t\nearrow T_*}[w(\cdot,t)]_{C^\alpha(\Omega)} < \infty$. Set
				\begin{equation*}
					d_*:=d(x_*,\Omega)=\inf_{x\in\Omega}|x-x_*|>0.
				\end{equation*}
				Since $\xi(t)\to x_*$ as $t\nearrow T_*$, for all $t$ sufficiently close to $T_*$ we have
				\begin{equation*}
					|x-\xi(t)|\geq \frac12d_*
				\end{equation*}
				for all $x\in\Omega$. Hence, with $Y_t(x):=e^{5s/2}(x-\xi(t))$,
				\begin{equation*}
					|Y_t(x)|\geq \frac12d_*e^{5s/2}=:c_1e^{5s/2}
				\end{equation*}
				for all $x\in\Omega$. Fix $x,\widetilde{x}\in\Omega$ with $x\neq\widetilde{x}$, and set $y:=Y_t(x)$ and $\widetilde{y}:=Y_t(\widetilde{x})$.

				Suppose first that $y$ and $\widetilde{y}$ have the same sign. Then every $\widehat{y}$ between $y$ and $\widetilde{y}$ satisfies $|\widehat{y}|\geq c_1e^{5s/2}$. By \eqref{Wydec_fin_2}, for any $ \textstyle \alpha\in(\frac35,1]$,
				\begin{equation*}
					|W_y(\widehat{y},s)|\leq C|\widehat{y}|^{-2/5}=C|\widehat{y}|^{\alpha-1}|\widehat{y}|^{3/5-\alpha}\leq Ce^{(\frac32-\frac52\alpha)s}|\widehat{y}|^{\alpha-1}.
				\end{equation*}
				Therefore,
				\begin{equation*}
					e^{(\frac52\alpha-\frac32)s}\frac{|W(y,s)-W(\widetilde{y},s)|}{|y-\widetilde{y}|^\alpha}\leq C\frac{\left|\int_{\widetilde{y}}^y|\widehat{y}|^{\alpha-1}\,d\widehat{y}\right|}{|y-\widetilde{y}|^\alpha}\leq C,
				\end{equation*}
				where we used $\big||y|^\alpha-|\widetilde{y}|^\alpha\big|\leq |y-\widetilde{y}|^\alpha$ for $0<\alpha\leq1$.

				Suppose next that $y$ and $\widetilde{y}$ have opposite signs. Then
				\begin{equation*}
					|x-\widetilde{x}|=|x-\xi(t)|+|\widetilde{x}-\xi(t)|\geq d_*,
				\end{equation*}
				and therefore $|y-\widetilde{y}|\geq d_*e^{5s/2}$. Using \eqref{not-zero}, we obtain
				\begin{equation*}
					e^{(\frac52\alpha-\frac32)s}\frac{|W(y,s)-W(\widetilde{y},s)|}{|y-\widetilde{y}|^\alpha}\leq Ce^{(\frac52\alpha-\frac32)s}|y-\widetilde{y}|^{3/5-\alpha}\leq C.
				\end{equation*}
				Combining the two cases and using \eqref{ys}--\eqref{WZQG}, we conclude that
				\begin{equation*}
					\sup_{\substack{x,\widetilde{x}\in\Omega\\x\neq\widetilde{x}}}\frac{|w(x,t)-w(\widetilde{x},t)|}{|x-\widetilde{x}|^\alpha}\leq C
				\end{equation*}
				for all $t$ sufficiently close to $T_*$. Hence we have
				\begin{equation*}
					\limsup_{t\nearrow T_*}[w(\cdot,t)]_{C^\alpha(\Omega)}<\infty.
				\end{equation*}
			This also implies that the blow-up location $x_*$ is unique. 
			
			Step 4: We shall prove the temporal blow-up rate \eqref{blow-up-rate1}, i.e., 
			for $x_*\in \Omega$ and $\textstyle \alpha \in (\frac35,1]$,
			\begin{equation*} \left[ w(\cdot, t) \right]_{C^\alpha(\Omega)}\sim (T_*-t)^{-\frac{5\alpha-3}{2}}
			\end{equation*} 
			for all $t$ sufficiently close to $T_*$.
			By symmetry, we may assume that $\wt{y}\leq y$. We first note that  
			\begin{equation*}
				|W(y,s)-W(\wt{y},s)|\leq \int^y_{\wt{y}}|W_y(\hat{y},s)|\,d\hat{y} \leq C\int^y_{\wt{y}}(1+\hat{y}^2)^{(\alpha-1)/2}\,d\hat{y} \leq C|y-\wt{y}|^{\alpha},
			\end{equation*}
			where in the second inequality, we used \eqref{Wydec_fin_2}, i.e.,  
			\begin{equation*}
				|W_y(y,s)|\leq C(1+y^2)^{-1/5}\leq C(1+y^2)^{(\alpha-1)/2}
			\end{equation*} 
			for $\textstyle \alpha \in (\frac35,1]$.
			Then again from \eqref{ys} and \eqref{WZQG}, we have for any $x, \wt{x} \in \mathbb{R}$, 
			\begin{equation}\label{u_Calp2}
				\frac{|w(x,t)-w(\wt{x}, t )|}{|x-\wt{x} |^\alpha}=e^{(\frac{5}{2}\alpha-\frac{3}{2})s}\frac{|W(y,s)-W(\wt{y},s)|}{|y-\wt{y}|^{\alpha}}\leq Ce^{(\frac{5}{2}\alpha-\frac{3}{2})s}
			\end{equation}
			for some $C>0$. 
			Note from \eqref{dottau_close} that
			\begin{equation*}
				\frac{1}{2}(T_*-t)\leq \tau(t)-t=\int^{T_*}_t \left( 1-\dot{\tau}(t') \right)  \, dt'\leq \frac{3}{2}(T_*-t).
			\end{equation*}
			This, together with \eqref{LB-blow} and \eqref{u_Calp2}, yields the desired blow-up rate of $w$ as 
			\begin{equation}\label{pf-br}
				   [w(\cdot,t)]_{C^\alpha} \sim(T_*-t)^{-(\frac{5}{2}\alpha-\frac{3}{2})}.
			\end{equation} 
			In particular, from \eqref{pf-br}, we see that the blow-up rate of $\|\partial_xw(\cdot,t)\|_{L^{\infty}}$ is $(T_*-t)^{-1}$.
			This completes the proof of Theorem~\ref{mainthm}. 
			\qed

\section{Closure of bootstrap assumptions}\label{sec3}

We improve the bootstrap assumptions \eqref{dottau}--\eqref{Wy4_bound} and thereby prove Proposition~\ref{Boot}. In Section~\ref{sec:pre}, we first establish preliminary estimates on \(\widetilde G\), \(Q\), the modulation parameters, and the derivatives of \(Z\), which are used in the subsequent proof. The two main inputs are the refined decay estimates for higher derivatives of \(Z\), obtained in Lemma~\ref{Lemma:Z_high}, and the weighted estimate for the nonlocal term \(Q_y\), proved in Lemma~\ref{lem_Qdec} using the Green-function bounds of Appendix~\ref{App:K}. Section~\ref{sec:3.2} closes the bootstrap assumptions on \(\dot\tau\) and \(Z_y\), while Section~\ref{sec:3.3} closes the remaining estimates for \(W\) and its derivatives.

\subsection{Preliminary estimates} \label{sec:pre}

We begin by recalling two lemmas from \cite{LPP}, omitting their proofs for brevity, which will be used in the proof of Lemma~\ref{lem_Gq}. 

\begin{lemma} [Landau--Kolmogorov inequality] \label{LK}
Let $\phi\in C^2(\mathbb{R})$ satisfy $\|\phi\|_{L^{\infty}}<\infty$ and $\|\phi''\|_{L^{\infty}}<\infty$. Then
\begin{equation*}
\|\phi'\|^2_{L^{\infty}}\leq 2\|\phi\|_{L^{\infty}}\|\phi''\|_{L^{\infty}}.
\end{equation*}
\end{lemma}

\begin{lemma} [\cite{LPP}, Lemma~4.5]  \label{w_LPP}
Let $s>\tfrac{3}{2}$ and $h_*>0$. Suppose that $h-h_*\in H^s(\mathbb{R})$ with $0<h_{\mathrm{min}}\leq h\leq h_{\mathrm{max}}<\infty$ for some constants $h_{\mathrm{min}}$, $h_{\mathrm{max}}$. Then, for any $\psi\in L^1(\mathbb{R}) \cap C_{\mathrm{lim}}(\mathbb{R})$, we have
\begin{equation*}
\|\mathcal{I}_h^{-1}\partial_x\psi\|_{L^{\infty}}\leq C(h_{\mathrm{min}},h_{\mathrm{max}})\|\psi\|_{L^1}.
\end{equation*}
Here, $\mathcal{I}_h$ is defined as in \eqref{I_h}, and $C_{\mathrm{lim}}(\mathbb{R})$ denotes the Banach space of continuous functions $\phi:\mathbb{R}\rightarrow\mathbb{R}$ having finite limits as $x \to \pm\infty$.
\end{lemma}

\begin{lemma}\label{h_max}
Let $(w,z)$ be the unique solution to \eqref{rSV'} with initial data satisfying \eqref{init_h}--\eqref{init_zx_weight} and \eqref{E0_bound}, as in Proposition~\ref{Boot}. Then it holds that
	\begin{equation} \label{hbd}
		0<h_{\mathrm{min}}\leq h(x,t) \leq h_{\mathrm{max}}
	\end{equation}
	for all $x\in\mathbb{R}$ and $t\in[-\veps,T_{\sigma_1}]$, where $\textstyle h_{\mathrm{min}}:= \frac{h_*}{2}$ and $\textstyle h_{\mathrm{max}}:= \frac{1+\sqrt{3}}{2}h_*$.
\end{lemma}

\begin{proof}
		From \eqref{def_E} and the inequality $\textstyle \frac{1}{2}(a^2+b^2)\geq \pm ab$, it follows that for each $t\geq -\veps$,
		\begin{equation*}
			\begin{split}
				E_0=E(t)&\geq \int_{\mathbb{R}}\frac{1}{2}(h-h_*)^2+\frac{1}{2}h^2h_x^2\,dx
				\\
				&\geq \int^x_{-\infty}(h-h_*)hh_x\,dx-\int^{\infty}_x (h-h_*)hh_x\,dx = \frac{1}{3}(h-h_*)^2(2h+h_*).
			\end{split}
		\end{equation*}
		Observe that the function $(x-h_*)^2(2x+h_*)$ is strictly decreasing on the interval  $x\in(0,h_*)$, and strictly increasing on $x>h_*$. Therefore, combining \eqref{E0_bound} with $\textstyle \lim_{|x|\rightarrow\infty}h(x,t)=h_*$, we deduce that
	\begin{equation*}
		0<h_{\mathrm{min}}\leq h(x,t)\leq h_{\mathrm{max}}<\infty \quad \text{for all }(x,t)\in\mathbb{R}\times [-\veps,T_{\sigma_1}],
	\end{equation*} where $h_{\mathrm{min}}\in (0,h_*)$ and $h_{\mathrm{max}}\geq h_*$ are the solutions of 
	\begin{equation*}
		\frac{h_*^3}{6}=\frac{1}{3}(x-h_*)^2(2x+h_*),
	\end{equation*} and they are explicitly given by $\textstyle h_{\mathrm{min}}=\frac{h_*}{2}$, $\textstyle h_{\mathrm{max}}=\frac{1+\sqrt{3}}{2}h_*$.
\end{proof}

\begin{lemma}\label{lem_Gq}
	Let $(w,z)$ be as in Lemma~\ref{h_max}. There exists a constant $C>0$ such that
	\begin{equation} \label{Gq_bound}
		\lVert G (\cdot,t) \rVert_{L^\infty}, \, \lVert q(\cdot,t) \rVert_{L^\infty}, \, \|q_x(\cdot,t)\|_{L^{\infty}} \leq C
	\end{equation}
	for all $t \in [-\varepsilon,T_{\sigma_1}]$, or equivalently,
	\begin{equation}\label{GQ_bound}
		\lVert \tilde{G} (\cdot,s) \rVert_{L^\infty}, \, \lVert Q(\cdot,s) \rVert_{L^\infty}, \, \|e^{5s/2}Q_y(\cdot,s)\|_{L^{\infty}} \leq C
	\end{equation}
	for all $s \in [s_0,\sigma_1]$.
\end{lemma}

\begin{proof}
	From the definition \eqref{G} of $G$ and the conservation of energy \eqref{def_E}--\eqref{Energy}, we obtain the estimate
	\begin{equation} \label{Gest}
		\begin{split}
			|G(x,t)|&=\left|\int^x_{-\infty}u_x^2-\frac{h_x^2}{4h}\,dx'\right|
			\leq \int_{\mathbb{R}}u_x^2 +\frac{h_x^2}{h}\,dx' \leq \frac{2}{ h_{\mathrm{min}}^3}E_0\leq \frac{8}{3},
		\end{split}
	\end{equation}
	which implies the bound for $G$ stated in \eqref{Gq_bound}. Note that the last inequality holds by \eqref{E0_bound} and Lemma~\ref{h_max}.
	 To estimate $q(x,t)$, we recall the identity \eqref{G-q}:
	\begin{equation*}
		G-q=-\frac{1}{2} \mathcal{I}_h^{-1} \partial_x \left( 2 h^3 u_x^2-\frac{1}{2}h^2h_x^2 \right).
	\end{equation*}
	Note that $(h-h_*,u)(\cdot,t) \in H^5(\mathbb{R}) \cap C^4(\mathbb{R})$ for $t \in [-\varepsilon,T_{\sigma_1}]$ by Lemma~\ref{existence}. Moreover, by a similar calculation to \eqref{Gest}, we have
	\begin{equation*}
		\left\|2h^3u_x^2-\frac{1}{2}h^2h_x^2\right\|_{L^1} \leq 2\int_{\mathbb{R}}h^3\left(u_x^2+\frac{h_x^2}{h}\right)\,dx \leq 4E_0.
	\end{equation*}
	It then follows from Lemma~\ref{w_LPP} that
	\begin{equation}\label{G-q_temp}
		\|G-q\|_{L^{\infty}}\leq C.
	\end{equation}
	Combining this with the bound \eqref{Gest}, we obtain
	\begin{equation*}
		\lVert q(\cdot,t) \rVert_{L^\infty} \leq \lVert G \rVert_{L^\infty} + \lVert G-q \rVert_{L^\infty} \leq C.
	\end{equation*}
	
	By the definition \eqref{G} of $q$ with \eqref{I_h}, we find that $q$ satisfies
	\begin{equation} \label{qeq}
		q- h^{-1}(h^3q_x)_x=G.
	\end{equation}
	We introduce a change of variables by setting $\textstyle x' := \int^x_{0} \frac{1}{h^3(z,t)}\,dz$, i.e., 
	\begin{equation*}
		\frac{d}{dx'} = h^3 \frac{d}{dx}
	\end{equation*}
	and, after rearranging the resulting equation, we obtain
	\begin{equation*}
		q_{x'x'} = -  h^4 ( G-q ).
	\end{equation*}
	Incorporating the bound \eqref{G-q_temp} into the equation above, we deduce
	\begin{equation*}
		|q_{x'x'}(x,t)| \leq  h_{\mathrm{max}}^4 \lVert G-q \rVert_{L^\infty} \leq C(  h_{\mathrm{max}}, E_0).
	\end{equation*} 
	This, together with Lemma~\ref{LK}, yields
	\begin{equation*}
		\begin{split}
			(h_{\mathrm{min}})^3\|q_x(\cdot,t)\|_{L^{\infty}}\leq \|q_{x'}(\cdot,t)\|_{L^{\infty}} \leq (2\|q\|_{L^{\infty}}\|q_{x'x'}\|_{L^{\infty}})^{1/2} \leq C,
		\end{split}
	\end{equation*}
	which completes the estimate for $q_x$. The bounds \eqref{GQ_bound} follow immediately from \eqref{Gq_bound} and \eqref{WZQG}.
\end{proof}

\begin{lemma}\label{wz_lem}
Under the same assumptions in Proposition~\ref{Boot}, there exists a constant $C>0$ such that
\begin{equation}\label{WZbound}
\|w(\cdot,t)\|_{L^{\infty}} \leq C, \quad \|z(\cdot,t)\|_{L^{\infty}}\leq C
\end{equation}
for all $t \in [-\varepsilon, T_{\sigma_1}]$.
\end{lemma}

\begin{proof}
Let $\psi_w$ be the characteristic curve satisfying
\begin{equation*}
\partial_t\psi_w(x,t)=w(\psi_w(x,t),t)+\frac{1}{3}z(\psi_w(x,t),t), \quad \psi_w(x,-\veps)=x.
\end{equation*}
Then, \eqref{rSV'1} can be rewritten as
	\begin{equation*}
		\frac{d}{dt}w(\psi_w(x,t),t)=\frac{8}{3} \left( G(\psi_w(x,t),t) - q(\psi_w(x,t),t) \right).
	\end{equation*}
	Integrating it along $\psi_w$ and applying Lemma~\ref{lem_Gq}, we obtain
	\begin{equation} \label{winf}
		\|w(\cdot,t)\|_{L^{\infty}}\leq \lVert w_0 \rVert_{L^\infty} + C(T_{\sigma_1} + \varepsilon).
	\end{equation}
Combining \eqref{winf} and \eqref{Tsigma1}, we have the desired bound on $\lVert w(\cdot,t) \rVert_{L^\infty}$. Similarly, we define $\psi_z$ to satisfy 
\begin{equation*}
\partial_t\psi_z(x,t)=z(\psi_z(x,t),t)+\frac{1}{3}w(\psi_z(x,t),t), \quad \psi_z(x,-\varepsilon) = x.
\end{equation*}
Then, by rewriting the equation \eqref{rSV'2} along the characteristic curve $\psi_z$ and integrating, we obtain
	\begin{equation*}
		\|z(\cdot,t)\|_{L^{\infty}}\leq \|z_0\|_{L^{\infty}}+C(T_{\sigma_1}+\veps) \leq C,
	\end{equation*}
	where we used $\lVert z_0 \rVert_{L^\infty} \leq 1 + 2\sqrt{h_*}$ from \eqref{init_wbound}.
\end{proof}

We derive decay bounds for the higher-order derivatives of \( Z \) and \( Q \), which arise in the forcing terms in~\eqref{Wy2}--\eqref{Wy4}. The estimates for \( \partial_y^n Z \) will be improved later in this section.

\begin{lemma} \label{lem:ZQ}
Under the same assumptions in Proposition~\ref{Boot}, there exists a constant $C>0$ such that
\begin{equation}\label{Z_high}
	\lVert \partial_y^n Z (\cdot,s) \rVert_{L^\infty} \leq C e^{-3s/2}, \quad n=2,3,4
\end{equation}
and
\begin{equation} \label{Q_high}
\lVert \partial_y^n Q(\cdot,s) \rVert_{L^\infty} \leq C e^{-4s}, \quad n=2,3,4
\end{equation}
for all $s \in [s_0,\sigma_1]$.
\end{lemma}

\begin{proof}
To establish \eqref{Z_high} and \eqref{Q_high}, we sequentially estimate higher-order derivatives of $Z$ and $Q$, alternating between them as their estimates are mutually dependent.

For estimates on the derivatives of $q$, we introduce the change of variables $\textstyle x':= \int^x_{0} \frac{1}{h^3(z,t)}\,dz$, i.e., 
\begin{equation} \label{chvar}
\frac{d}{dx'} = h^3 \frac{d}{dx}
\end{equation}
as in the proof of Lemma~\ref{lem_Gq}. Then, by \eqref{G}, $q$ satisfies 
\begin{equation} \label{eq:qx'x'}
q_{x'x'} = - h^4 (G-q).
\end{equation}
Moreover, the relation 
\begin{equation} \label{rel:qx'x'} 
q_{x'x'} = 3h^5 h_x q_x + h^6 q_{xx} 
\end{equation} 
holds by \eqref{chvar}. Applying the bound \eqref{Gq_bound} to \eqref{eq:qx'x'} and \eqref{rel:qx'x'}, we obtain
\begin{equation*}
|q_{x'x'}| \leq  h_{\mathrm{max}}^4 \lVert G-q \rVert_{L^\infty} \leq C
\end{equation*}
and
\begin{equation} \label{qxxbd} 
|q_{xx}| \leq h_{\mathrm{min}}^{-6} |q_{x'x'}| + 3h_{\mathrm{min}}^{-1} |h_x| |q_x| \leq C(1 + |h_x|),
\end{equation}
respectively. To estimate $|h_x|$, we observe from \eqref{wz}, \eqref{WZbound}, \eqref{Uy1}, and \eqref{lem_Zy} that
\begin{equation} \label{hxbd}
	|h_x(x,t)|=\frac{|w-z||w_x-z_x|}{8}\leq C\left|e^sW_y-e^{5s/2}Z_y\right| \leq Ce^s.
\end{equation}
Combining the above estimates, we arrive at the desired bound
\begin{equation} \label{Qyybd}
\lVert Q_{yy} \rVert_{L^\infty} = \lVert e^{-5s} q_{xx} \rVert_{L^\infty} \leq  C e^{-4s}.
\end{equation}

To estimate $Z_{yy}$, we consider \eqref{eq:Z_high} with \eqref{D^Z} and \eqref{F^Z} for $n=2$. Thanks to \eqref{dottau1}, \eqref{lem_Zy} and \eqref{Uy1}, we obtain the lower bound
\begin{equation*}
D_2^Z \geq 5 - \left( 2 e^{3s/2} \lvert Z_y \rvert + \lvert W_y \rvert \right) (1+\varepsilon^{1/2}) \geq \frac{5}{2}
\end{equation*}
for sufficiently small $\varepsilon>0$. Using \eqref{dottau1}, \eqref{Uy1}, \eqref{Wyy_bound}, \eqref{Qyybd} and \eqref{lem_Zy}, we also find
\begin{equation*}
\lvert F_2^Z \rvert \leq C \left( e^{-5s} + e^{-3s/2}  + e^{-5s/2} \right) \leq C e^{-3s/2}.
\end{equation*}
We integrate the equation along the characteristic curve \(\psi_Z\) defined by \(\partial_s \psi_Z(y,s) = U^Z(\psi_Z,s)\) with \(\psi_Z(y,s_0) = y\), obtaining
\begin{equation*}
Z_{yy}(\psi_Z(y,s),s) = Z_{yy}(y,s_0) e^{-\int_{s_0}^s (D_2^Z \circ \psi_Z) \, ds'} + \int_{s_0}^s e^{-\int_{s'}^s (D_2^Z \circ \psi_Z) \, ds''} (F_2^Z \circ \psi_Z) \, ds'.
\end{equation*}
Using the lower bound on \(D_2^Z\) and the estimate for \(F_2^Z\), we deduce
\begin{equation} \label{pr_Zyy}
\begin{aligned}
\lVert Z_{yy}(\cdot,s) \rVert_{L^\infty} 
&\leq \lVert Z_{yy}(\cdot,s_0) \rVert_{L^\infty} e^{-5(s-s_0)/2} + C \int_{s_0}^s e^{-5(s-s')/2} e^{-3s'/2} \, ds' \\
&\leq \varepsilon^{5/2} e^{-5s/2} + C e^{-3s/2} \\
&\leq C e^{-3s/2},
\end{aligned}
\end{equation}
where in the second inequality we used \eqref{init_wbound}.

Differentiating \eqref{eq:qx'x'} and \eqref{rel:qx'x'} with respect to $x'$ and using \eqref{chvar}, we have
\begin{equation} \label{eq:qx'x'x'}
\partial_{x'}^3 q = - 4  h^6 h_x (G-q) - h^7 (G_x - q_x)
\end{equation}
and
\begin{equation} \label{rel:qx'x'x'}
\partial_{x'}^3 q = (15h^7 h_x^2 + 3h^8 h_{xx})q_x + 9 h^8 h_x q_{xx} + h^9 \partial_x^3 q,
\end{equation}
respectively. It then follows from \eqref{eq:qx'x'x'}, together with the bounds \eqref{Gq_bound}, \eqref{hxbd} and \eqref{hbd}, that
\begin{equation*}
|\partial_{x'}^3 q | \leq C \left( e^{s} + |G_x| \right).
\end{equation*}
Substituting this bound into \eqref{rel:qx'x'x'}, and using \eqref{Gq_bound}, \eqref{hxbd}, \eqref{qxxbd} and \eqref{hbd}, we deduce
\begin{equation} \label{qxxxbd}
\begin{split}
|\partial_x^3 q | & = | h^{-9} \partial_{x'}^3 q - (15h^{-2} h_x^2 + 3 h^{-1} h_{xx})q_x - 9 h^{-1} h_x q_{xx} | \leq C \left( e^{2s} + |G_x| + |h_{xx}| \right).
\end{split}
\end{equation}
Using \eqref{WZbound}, \eqref{Uy1}, \eqref{lem_Zy}, \eqref{Wyy_bound} and \eqref{pr_Zyy}, we obtain the following estimate for $h_{xx}$:
\begin{equation} \label{hxxbd}
\begin{split}
|h_{xx}| & = \left| \frac{(w_x-z_x)^2}{8} + \frac{(w-z)(w_{xx}-z_{xx})}{8} \right| \\
& \leq C \left( e^{2s} |W_y|^2 + e^{7s/2}|W_y||Z_y| + e^{5s} |Z_y|^2 + e^{7s/2}|W_{yy}| + e^{5s} |Z_{yy}| \right) \\
& \leq C e^{7s/2}.
\end{split}
\end{equation}
Similarly, we have
\begin{equation} \label{Gxbd}
\begin{split}
|G_x| & = \left| \frac{3}{16} w_x^2 + \frac{5}{8} w_xz_x + \frac{3}{16} z_x^2 \right| \\
& \leq C \left( e^{2s} |W_y|^2 + e^{7s/2}|W_y||Z_y| + e^{5s}|Z_y|^2 \right) \\
& \leq C e^{2s},
\end{split}
\end{equation}
where the first identity is given by differentiating \eqref{G} in $x$. Recalling $\eqref{y-x}$, the relation between the variables $y$ and $x$, we conclude
\begin{equation} \label{Qyyybd}
\| \partial_{y}^3 Q\|_{L^\infty} = \| e^{-15s/2}  \partial_x^3 q \|_{L^\infty} \leq C e^{-4s}.
\end{equation}

As before, applying \eqref{dottau1}, \eqref{lem_Zy} and \eqref{Uy1} to $D_3^Z$ in \eqref{D^Z}, we obtain the lower bound
\begin{equation*}
D_3^Z \geq \frac{15}{2} - \left( 3 e^{3s/2} \lvert Z_y \rvert + \frac{2\lvert W_y \rvert}{3} \right) (1+\varepsilon^{1/2}) \geq 6
\end{equation*}
for sufficiently small $\varepsilon>0$. Furthermore, using \eqref{lem_Zy}, \eqref{Wyy_bound}, \eqref{Wy3_bound}, \eqref{dottau1}, \eqref{Uy1}, \eqref{pr_Zyy} and \eqref{Qyyybd} in $F_3^Z$ from \eqref{F^Z}, we obtain the estimate
\begin{equation*}
\lvert F_3^Z \rvert \leq C e^{-3s/2}.
\end{equation*}
Integrating the equation for $\partial_y^3Z$ in \eqref{eq:Z_high} along the characteristic curve $\psi_Z(y,s)$, we obtain
\begin{equation} \label{pr_Zy3}
\begin{split}
\lVert \partial_y^3 Z(\cdot,s) \rVert_{L^\infty} & \leq \lVert \partial_y^3 Z (\cdot,s_0) \rVert_{L^\infty} e^{-6(s-s_0)} + C \int_{s_0}^s e^{-6(s-s')} e^{-3s'/2} \, ds' \\
& \leq \varepsilon^{3/2}e^{-6s} + C e^{-3s/2} \leq C e^{-3s/2}.
\end{split}
\end{equation}

Differentiating \eqref{eq:qx'x'x'} with respect to $x'$ and using \eqref{chvar}, we have
\begin{equation*}
\partial_{x'}^4 q = \left( - 24 h^8 h_x^2 - 4 h^9h_{xx} \right) (G-q) - 11 h^9 h_x (G_x - q_x) - h^{10} (G_{xx} - q_{xx}),
\end{equation*}
which, combined with the bounds \eqref{hbd}, \eqref{Gq_bound}, \eqref{hxbd}, \eqref{Qyybd}, \eqref{hxxbd} and \eqref{Gxbd}, yields
\begin{equation} \label{qx'4b}
| \partial_{x'}^4 q | \leq C ( e^{7s/2} + |G_{xx}| ).
\end{equation}
We differentiate \eqref{rel:qx'x'x'} with respect to $x'$ and use \eqref{qx'4b} to obtain a bound for $\partial_x^4q$:
\begin{equation*}
\begin{split}
|\partial_x^4 q| & = \big| h^{-12}\partial_{x'}^4 q - 18 h^{-1}h_x \partial_x^3 q - (87 h^{-2}h_x^2 + 12 h^{-1}h_{xx})q_{xx} \\
& \qquad - (105h^{-3}h_x^3 + 54 h^{-2}h_xh_{xx} + 3 h^{-1}\partial_x^3 h)q_x \big| \\
& \leq C ( e^{9s/2} + |G_{xx}| + |\partial_x^3 h| ),
\end{split}
\end{equation*}
where we have also used the bounds \eqref{hbd}, \eqref{hxbd}, \eqref{hxxbd}, \eqref{Gq_bound}, \eqref{Qyybd}, and \eqref{Qyyybd}. To estimate $\partial_x^3h$, we use \eqref{WZbound}, \eqref{Uy1}, \eqref{Wyy_bound}, \eqref{Wy3_bound}, \eqref{lem_Zy}, \eqref{pr_Zyy}, and \eqref{pr_Zy3}, leading to
\begin{equation*}
\begin{split}
|\partial_{x}^3 h| & =  \left| \frac{3 \left(w_x-z_x \right) \left(w_{xx}-z_{xx} \right)}{8} + \frac{\left(w-z \right)\left(\partial_x^3 w - \partial_x^3 z \right)}{8} \right| \\
& \leq C \left( e^{9s/2} |W_y||W_{yy}| + e^{6s} |W_y||Z_{yy}| + e^{6s} |Z_y| |W_{yy}| + e^{15s/2}|Z_y| |Z_{yy}| + e^{6s} |\partial_y^3 W| + e^{15s/2} |\partial_y^3 Z| \right)  \\
& \leq C e^{6s}.
\end{split}
\end{equation*}
Similarly, using the identity \eqref{G}, we obtain
\begin{equation*}
\begin{split}
| G_{xx} | & = \left| \frac{3}{8}w_xw_{xx} + \frac{5}{8}w_{xx}z_x + \frac{5}{8} w_xz_{xx} + \frac{3}{8} z_xz_{xx}  \right| \\
& \leq C \left( e^{9s/2}|W_y||W_{yy}| + e^{6s} |W_{yy}||Z_y| + e^{6s}|W_y||Z_{yy}| + e^{15s/2}|Z_y||Z_{yy}| \right) \\
& \leq C e^{9s/2}.
\end{split}
\end{equation*}
Combining the above estimates, we have
\begin{equation} \label{Qy4bd}
\lVert \partial_y^4 Q \rVert_{L^\infty} = \lVert e^{-10s} \partial_x^4 q \rVert_{L^\infty} \leq C e^{-4s}.
\end{equation}

Finally, we estimate $\partial_y^4 Z$. Using \eqref{dottau1}, \eqref{Uy1}, and \eqref{lem_Zy}, we observe that $D_4^Z$, defined in \eqref{D^Z}, satisfies the lower bound
\begin{equation*}
D_4^Z \geq 10 - \left( 4 e^{3s/2} \lvert Z_y \rvert + \frac{\lvert W_y \rvert}{3} \right) (1+\varepsilon^{1/2}) \geq 9
\end{equation*}
for sufficiently small $\varepsilon>0$, and the forcing term $F_4^Z$ in \eqref{F^Z} admits the upper bound
\begin{equation*}
\lvert F_4^Z \rvert \leq C \left( e^{-3s/2} + e^{-5s} + e^{-5s/2} + e^{-7s/2} \right) \leq C e^{-3s/2}.
\end{equation*}
In obtaining the estimate for $F_4^Z$, we have also used \eqref{dottau1}, \eqref{Uy1}, \eqref{lem_Zy}, \eqref{Wyy_bound}, \eqref{Wy3_bound}, \eqref{Wy4_bound}, \eqref{Qy4bd}, \eqref{pr_Zyy}, and \eqref{pr_Zy3}. Integrating the equation for $\partial_y^4 Z$ in \eqref{eq:Z_high} along the characteristic curve $\psi_Z(y,s)$ defined previously, we conclude that
\begin{equation*}
\begin{split}
\lVert \partial_y^4 Z(\cdot,s) \rVert_{L^\infty} & \leq \lVert \partial_y^4 Z (\cdot,s_0) \rVert_{L^\infty} e^{-9(s-s_0)} + C \int_{s_0}^s e^{-9(s-s')} e^{-3s'/2} \, ds' \\
& \leq \varepsilon e^{-9s} + C e^{-3s/2} \leq C e^{-3s/2},
\end{split}
\end{equation*}
as desired.
This completes the proof. 
\end{proof}

\begin{lemma}\label{xilem}
Under the same assumptions in Proposition~\ref{Boot}, it holds for all $s\in[s_0,\sigma_1]$ (equivalently, $t \in [-\varepsilon,T_{\sigma_1}]$) that
	\begin{equation}\label{kap_xi}
		\lvert \dot{\kappa}(t) \rvert \leq C, \quad \lvert \kappa(t) \rvert \leq C, \quad \lvert \dot{\xi}(t) \rvert \leq C, \quad \lvert \xi(t) \rvert \leq C
	\end{equation}
	for some constant $C>0$.
\end{lemma}

\begin{proof}
Recalling \eqref{kap}, we deduce from \eqref{Wy3_0}, \eqref{lem_Zy}, \eqref{Z_high}, \eqref{Q_high}, and \eqref{GQ_bound} that
	\begin{equation*}
		\begin{split}
			|\dot{\kappa}|&\leq \frac{2 e^{-s/2}}{|\partial_y^3 W (0,s)|} \left( \frac{8 e^{s/2} |Q_{yy}(0,s)|}{3} + e^{3s} |Z_y(0,s)| |Z_{yy}(0,s)| + \frac{8e^{3s/2} |Z_{yy}(0,s)|}{3} \right) \\
			& \quad + \frac{8 |\tilde{G}(0,s)-Q(0,s)|}{3}
			\\
			&\leq  C e^{-s/2} \left( e^{-7s/2} + e^{-s} + 1 \right) + C
			\leq C.
		\end{split}
	\end{equation*}
Hence, using \eqref{dottau1}, we obtain
	\begin{equation*}
		|\kappa(t)|\leq |\kappa_0|+\int^t_{-\varepsilon}|\dot{\kappa}|\,dt' \leq  |\kappa_0|+\int^s_{s_0}|\dot{\kappa}|\frac{e^{-s'}}{1-\dot{\tau}}\,ds' \leq |\kappa_0|+\int^s_{s_0}Ce^{-s'}\,ds' \leq C(\kappa_0).
	\end{equation*}
Similarly, thanks to \eqref{xi}, \eqref{lem_Zy}, \eqref{Z_high}, \eqref{Q_high}, \eqref{WZbound}, and \eqref{Wy3_0}, we find
	\begin{equation*}
	\begin{split}
		|\dot{\xi}| & \leq  Ce^{-3s/2} \left( e^{-7s/2} + e^{-s} + 1 \right)  + \frac{|Z(0,s)|}{3} + |\kappa| \leq C.
	\end{split}
	\end{equation*}
Integrating $\dot{\xi}$ along $t$, we obtain $|\xi(t)|\leq C$. 
\end{proof}

We now refine the decay estimates on \( \lVert \partial_y^n Z(\cdot,s) \rVert_{L^\infty} \) for \( n = 2,3,4 \), established in Lemma~\ref{lem:ZQ}. The bounds on $\partial_y^n Z$, obtained via integration along the characteristic flow of the equations for $Z$, are not sufficient to close the bootstrap argument. To refine these bounds, we rewrite some forcing terms in the equations for $\partial_y^n Z$ along the characteristic flow associated with the equation for $W$, which allows us to obtain sharper decay estimates. This approach follows the idea introduced in \cite[Section~5.2]{LPP}, adapted to the present self-similar framework. For this purpose, we introduce the characteristic curves $\psi_W$ and $\psi_Z$, defined by
\begin{equation} \label{curves}
\begin{split}
& \partial_s \psi_W(y,s) = U^W(\psi_W(y,s),s) = \left( \frac{5}{2}y + \frac{W}{1-\dot{\tau}} + \frac{e^{3s/2}Z}{3(1-\dot{\tau})} + \frac{e^{3s/2}(\kappa - \dot{\xi})}{1-\dot{\tau}} \right) (\psi_W(y,s),s), \\
& \partial_s \psi_Z(y,s) = U^Z(\psi_Z(y,s),s) = \left( \frac{5}{2}y + \frac{W}{3(1-\dot{\tau})} + \frac{e^{3s/2}Z}{1-\dot{\tau}} + \frac{e^{3s/2}(\kappa - 3 \dot{\xi})}{3(1-\dot{\tau})} \right) (\psi_Z(y,s),s)
\end{split}
\end{equation}
with $\psi_W(y,s_0) = \psi_Z(y,s_0) = y$ for all $y \in \mathbb{R}$ and $s \in [s_0,\sigma_1]$.

The following lemma provides improved decay rates of $\lVert \partial_y^n Z(\cdot,s) \rVert_{L^\infty}$ for $n=2,3,4$, together with an integral estimate that will be used later:

\begin{lemma} \label{Lemma:Z_high}
Under the same assumptions in Proposition~\ref{Boot}, there exists a constant $C>0$ such that
\begin{equation}\label{Z_high'}
\lVert \partial_y^n Z (\cdot,s) \rVert_{L^\infty} \leq C e^{-5s/2}, \quad n=2,3,4
\end{equation}
and
\begin{equation} \label{int_Wy2}
\int_{s_0}^s e^{s'} W_y^2(\psi_Z(y,s'),s') \, ds' \leq \frac{4\sqrt{2}}{\sqrt{h_*}} + C \varepsilon
\end{equation}
for all $s \in [s_0,\sigma_1]$ and $y\in\mathbb{R}$.
\end{lemma}

\begin{proof}

We provide the proof of \eqref{Z_high'} for $n=2$, together with \eqref{int_Wy2}. Since the estimates for the third- and fourth-order derivatives are lengthy and somewhat repetitive, they are deferred to Appendix~\ref{App:Z}.

Recall the definition of $D_2^Z$ and $F_2^Z$ in \eqref{D^Z} and \eqref{F^Z}. It follows from \eqref{dottau1}, \eqref{lem_Zy}, and \eqref{Uy1} that
\begin{equation*}
D_2^Z \geq 5 - \left( 2 e^{3s/2} \lvert Z_y \rvert + \lvert W_y \rvert \right) (1+\varepsilon^{1/2}) \geq \frac{8}{3}
\end{equation*}
for sufficiently small $\varepsilon>0$. Moreover, by \eqref{dottau1}, \eqref{Wyy_bound}, \eqref{Q_high}, \eqref{lem_Zy},  and Young's inequality, it holds that
\begin{equation*}
\lvert F_2^Z \rvert \leq C \left( e^{-3s/2}W_y^2 + e^{-3s/2}W_{yy}^2 + e^{-5s/2} \right).
\end{equation*}
Integrating the equation for $Z_{yy}$ in \eqref{eq:Z_high} along the characteristic curve \(\psi_Z\) in \eqref{curves}, and using the above estimates, we obtain
\begin{equation} \label{pr_Zyy'}
\begin{aligned}
\lvert Z_{yy}(\psi_Z(y,s),s)\rvert &\leq \lVert Z_{yy}(\cdot,s_0) \rVert_{L^\infty} e^{-8(s-s_0)/3} + C \int_{s_0}^s e^{-8(s-s')/3} e^{-5s'/2} \, ds' \\
& \quad + C \int_{s_0}^s e^{-8(s-s')/3} e^{-3s'/2} ( W_y^2 + W_{yy}^2)(\psi_Z(y,s'),s') \, ds'  \\
&\leq C e^{-5s/2} + C e^{-5s/2} \int_{s_0}^s e^{s'} (W_y^2 + W_{yy}^2)(\psi_Z(y,s'),s') \, ds'.
\end{aligned}
\end{equation}
In the second inequality, we used \eqref{init_wbound} and
\begin{equation*}
e^{-8s/3} \int_{s_0}^s e^{7s'/6} (W_y^2 + W_{yy}^2) \, ds' \leq e^{-5s/2} \int_{s_0}^s e^{s'} (W_y^2 + W_{yy}^2) \, ds'.
\end{equation*}
In order to complete the proof of \eqref{Z_high'} for $n=2$, we now establish \eqref{int_Wy2} and
\begin{equation} \label{int_Wyy2}
\int_{s_0}^s e^{s'} W_{yy}^2(\psi_Z(y,s'),s') \, ds' \leq C \varepsilon^{1/5}
\end{equation}
for sufficiently small $\varepsilon>0$, which together with \eqref{pr_Zyy'} yields the desired bound
\begin{equation} \label{Zyy_proof}
\lVert Z_{yy}(\cdot,s) \rVert_{L^\infty} \leq C e^{-5s/2}.
\end{equation}

We first make a change of variables so that the integral in the last inequality of \eqref{pr_Zyy'}, originally taken along $\psi_Z$, is expressed in terms of the curve $\psi_W$ defined in \eqref{curves}. For this, we fix $y_z \in \mathbb{R}$ and, for each $s \in [s_0,\sigma_1]$, define $y_w(s)$ such that
\begin{equation} \label{rel:psi_ZW}
\psi_Z(y_z,s) = \psi_W(y_w(s),s).
\end{equation}
This change of variables is justified by showing that $\partial_y \psi_W (y,s) >0$ for all $y \in \mathbb{R}$ and $s \in [s_0,\sigma_1]$. Differentiating the equation for $\psi_W$ in \eqref{curves} with respect to $y$ yields
\begin{equation} \label{psiWy}
\partial_s  (\psi_W)_y = \left( \frac{5}{2} + \frac{W_y}{1-\dot{\tau}} + \frac{e^{3s/2}Z_y}{3(1-\dot{\tau})} \right) (\psi_W(y,s),s) \cdot (\psi_W)_y.
\end{equation}
Solving this ODE, we obtain for each $y$ that
\begin{equation*}
(\psi_W)_y(y,s) = (\psi_W)_y(y,s_0) \exp \left( \int_{s_0}^s \left( \frac{5}{2} + \frac{W_y}{1-\dot{\tau}} + \frac{e^{3s/2}Z_y}{3(1-\dot{\tau})} \right) (\psi_W(y,s'),s') \, ds' \right).
\end{equation*}
Since $\partial_y \psi_W(y,s_0) = 1 >0$ and, by the bounds \eqref{Uy1}, \eqref{lem_Zy}, and \eqref{dottau1},
\begin{equation*}
\frac{5}{2} - |W_y| (1+\varepsilon^{1/2}) - \frac{e^{3s/2}|Z_y|}{3} (1+\varepsilon^{1/2}) \geq \frac{1}{4}
\end{equation*}
for sufficiently small $\varepsilon>0$, we conclude that $\partial_y \psi_W(y,s) >0$ for all $s \in [s_0,\sigma_1]$.

We denote $\overline{y}(s) := \psi_Z(y_z,s)$. From the relation \eqref{rel:psi_ZW} and the equations \eqref{curves} for $\psi_Z$ and $\psi_W$, we compute
\begin{equation*}
\begin{split}
\frac{\partial \psi_W}{\partial y}(y_w(s),s) \frac{d y_w(s)}{ds} & = \left( \frac{5}{2}y + \frac{W}{3(1-\dot{\tau})} + \frac{e^{3s/2}Z}{1-\dot{\tau}} + \frac{e^{3s/2}(\kappa - 3 \dot{\xi})}{3(1-\dot{\tau})} \right) (\psi_Z(y_z,s),s) \\
& \qquad - \left( \frac{5}{2}y + \frac{W}{1-\dot{\tau}} + \frac{e^{3s/2}Z}{3(1-\dot{\tau})} + \frac{e^{3s/2}(\kappa - \dot{\xi})}{1-\dot{\tau}} \right) (\psi_W(y_w(s),s),s) \\
& = - \frac{2}{3} e^{3s/2} \left( \frac{e^{-3s/2} W}{1-\dot{\tau}} - \frac{Z}{1-\dot{\tau}} + \frac{\kappa}{1-\dot{\tau}} \right) (\overline{y}(s),s).
\end{split}
\end{equation*}
Since $e^{-3s/2}W + \kappa - Z = w-z = 4\sqrt{h} >0$ by \eqref{wz} and \eqref{hbd}, the right-hand side is strictly negative. Thus, using \eqref{dottau}, we obtain
\begin{equation} \label{psiW'bd}
\left| \frac{\partial \psi_W}{\partial y}(y_w(s),s) \frac{d y_w(s)}{ds} \right| \geq \frac{8}{3} e^{3s/2} \frac{\sqrt{h_{\mathrm{min}}}}{1+|\dot{\tau}|} \geq \frac{4}{3} e^{3s/2} \sqrt{h_{\mathrm{min}}}
\end{equation}
for sufficiently small $\varepsilon>0$. Moreover, we find that $\tfrac{dy_w}{ds} < 0$ and therefore $y_w(s)\le y_z$ for all $s\in[s_0,\sigma_1]$. Let $\overline{s}(y)$ denote the inverse function of $y_w(s)$, so that $y_w(\overline{s}(\xi)) = \xi$ for all $\xi \in [y_w(s),y_z]$. This, together with \eqref{psiW'bd}, allows us to estimate the integral in \eqref{int_Wy2} as
\begin{equation} \label{Wy2bd}
\begin{split}
& \int_{s_0}^s e^{s'} W_y^2(\psi_Z(y_z,s'),s') \, ds'  = \int_{s_0}^s e^{s'} W_y^2(\psi_W(y_w(s'),s'),s') \, ds' \\
& \quad \leq \frac{3}{4\sqrt{h_{\mathrm{min}}}} \int_{y_w(s)}^{y_z} e^{-\overline{s}(\xi)/2} W_y^2 (\psi_W(\xi,\overline{s}(\xi)), \overline{s}(\xi)) \frac{\partial \psi_W}{\partial \xi} (\xi, \overline{s}(\xi)) \, d\xi \\
& \quad = \frac{3}{4\sqrt{h_{\mathrm{min}}}} \int_{y_w(s)}^{y_z} \left(e^{-s_0/2}  W_y^2(\xi,s_0) + \int_{s_0}^{\overline{s}(\xi)} \frac{d}{ds'} \left( e^{-s'/2} (W_y \circ \psi_W)^2 \frac{ \partial \psi_W}{\partial \xi} \right)(\xi,s') \, ds' \right) \, d\xi \\
& \quad \leq \frac{3}{4\sqrt{h_{\mathrm{min}}}} \left( \lVert \partial_x w_0 \rVert_{L^2(\mathbb{R})}^2 + J_1 \right)\leq \frac{4\sqrt{2}}{\sqrt{h_*}} + \frac{3\sqrt{2}}{4\sqrt{h_*}} J_1,
\end{split}
\end{equation}
where we used $\textstyle \int e^{-s_0/2} W_y^2(y,s_0) \, dy = \int (\partial_x w_0)^2 \, dx$, and set
\begin{equation*}
J_1 \coloneqq \int_{y_w(s)}^{y_z} \int_{s_0}^{\overline{s}(\xi)} \frac{d}{ds'} \left( e^{-s'/2} (W_y \circ \psi_W)^2 \frac{\partial \psi_W}{\partial \xi} \right)(\xi,s') \, ds' d\xi.
\end{equation*}
In the last inequality, we have also used $\textstyle h_{\mathrm{min}} = \frac{h_*}{2}$ from Lemma~\ref{h_max} and
	\begin{equation}\label{w0H1}
		\|\partial_xw_0\|_{L^2(\mathbb{R})}^2=\left\|\partial_x u_0 + \frac{\partial_x h_0}{\sqrt{h_0}}\right\|^2_{L^2(\mathbb{R})} \leq 2 \int_\mathbb{R} \left( ( \partial_x u_0 )^2 + \frac{(\partial_x h_0)^2}{h_0} \right) \, dx \leq \frac{4}{h^3_{\mathrm{min}}}E_0 \leq \frac{16}{3},
	\end{equation}
which follows from \eqref{def_E}--\eqref{E0_bound} and \eqref{hbd}. Similarly, using \eqref{init_wbound}, we obtain the corresponding bound for $W_{yy}^2$:
\begin{equation} \label{Wyy2bd}
\begin{split}
\int_{s_0}^s e^{s'} W_{yy}^2(\psi_Z(y_z,s'),s') \, ds' \leq \frac{3}{4\sqrt{h_{\mathrm{min}}}} \left( \varepsilon^{5} \lVert \partial_x^2w_{0} \rVert_{L^2(\mathbb{R})}^2 + J_2 \right) \leq \frac{3}{4\sqrt{h_{\mathrm{min}}}} \left( C \varepsilon^{1/2} + J_2 \right),
\end{split}
\end{equation}
where
\begin{equation*}
J_2 \coloneqq \int_{y_w(s)}^{y_z} \int_{s_0}^{\overline{s}(\xi)} \frac{d}{ds'} \left( e^{-s'/2} (W_{yy} \circ \psi_W)^2 \frac{\partial \psi_W}{\partial \xi} \right)(\xi,s') \, ds' d\xi.
\end{equation*}

For notational convenience, we henceforth write $f \circ \psi_W$ simply as $f$, for any sufficiently regular function $f$. Substituting from \eqref{Wy1} and \eqref{psiWy}, and applying the bounds \eqref{lem_Zy}, \eqref{dottau1}, and \eqref{GQ_bound}, we estimate the integrands of $J_1$ and $J_2$ as follows: for $J_1$,
\begin{equation} \label{J1_1}
\begin{split}
\frac{d}{ds'} \left( e^{-s'/2} W_y^2 \frac{\partial \psi_W}{\partial \xi} \right) & = -\frac{1}{2} e^{-s'/2} W_y^2 \frac{\partial \psi_W}{\partial \xi} + 2 e^{-s'/2} W_y \frac{d W_y}{ds'} \frac{\partial \psi_W}{\partial \xi} + e^{-s'/2} W_y^2 \frac{\partial^2 \psi_W}{\partial s' \partial \xi} \\
& = e^{-s'/2} \frac{\partial \psi_W}{\partial \xi} W_y \left( \frac{3e^{3s'/2} Z_y W_y}{1-\dot{\tau}} - \frac{16e^{s'/2} Q_y}{3(1-\dot{\tau})} + \frac{e^{3s'} Z_y^2}{1-\dot{\tau}} \right) \\
& \leq C e^{-3s'/2} \frac{\partial \psi_W}{\partial \xi} W_y^2 + C e^{-5s'/2} \frac{\partial \psi_W}{\partial \xi}|W_y|
\end{split}
\end{equation}
and for $J_2$, using \eqref{Wy2} and \eqref{psiWy},
\begin{equation} \label{J2_1}
\begin{split}
\frac{d}{ds'} \left( e^{-s'/2} W_{yy}^2 \frac{\partial \psi_W}{\partial \xi} \right) & = -\frac{1}{2} e^{-s'/2} W_{yy}^2 \frac{\partial \psi_W}{\partial \xi} + 2 e^{-s'/2} W_{yy} \frac{d W_{yy}}{ds'} \frac{\partial \psi_W}{\partial \xi} + e^{-s'/2} W_{yy}^2 \frac{\partial^2 \psi_W}{\partial s' \partial \xi} \\
& = e^{-s'/2} \frac{\partial \psi_W}{\partial \xi} W_{yy} \bigg( -5W_{yy} - \frac{3W_y W_{yy}}{1-\dot{\tau}} + \frac{7e^{3s'/2} Z_y W_{yy}}{3(1-\dot{\tau})} \\
& \qquad \qquad \qquad \qquad \quad - \frac{16e^{s'/2} Q_{yy}}{3(1-\dot{\tau})} + \frac{2e^{3s'} Z_y Z_{yy}}{1-\dot{\tau}} + \frac{8e^{3s'/2} Z_{yy} W_y}{3(1-\dot{\tau})} \bigg).
\end{split}
\end{equation}

We first estimate the integrals of the terms appearing in the last inequality of \eqref{J1_1}. For the first term, we apply Fubini's theorem and the change of variables $y = \psi_W(\xi,s')$ to obtain
\begin{equation} \label{J1_first}
\begin{split}
\int_{y_w(s)}^{y_z} \int_{s_0}^{\overline{s}(\xi)} \left( e^{-3s'/2} \frac{\partial \psi_W}{\partial \xi} W_y^2 \right) \, ds' d\xi & = \int_{s_0}^s e^{-3s'/2} \bigg( \int_{\psi_W(y_w(s),s')}^{\psi_Z(y_z,s')} W_y^2(y,s') \, dy \bigg) \, ds' \\
& \leq \int_{s_0}^s e^{-s'} \bigg( \int_{\mathbb{R}} e^{-s'/2} W_y^2(y,s') \, dy \bigg) \, ds' \\
& \leq C E_0 \int_{s_0}^{\infty} e^{-s'} \, ds' \leq C \varepsilon,
\end{split}
\end{equation}
where we used \eqref{Energy} and \eqref{hbd}. For the second term, applying Fubini's theorem and the Cauchy--Schwarz inequality, we obtain
\begin{equation*}
\begin{split}
\int_{y_w(s)}^{y_z} \int_{s_0}^{\overline{s}(\xi)} \left| e^{-5s'/2} \frac{\partial \psi_W}{\partial \xi} W_y \right| \, ds' d\xi & = \int_{s_0}^s e^{-5s'/2} \bigg( \int_{\psi_W(y_w(s),s')}^{\psi_Z(y_z,s')} |W_y(y,s')| \, dy \bigg) \, ds' \\
& \leq C \int_{s_0}^s e^{-9s'/4} \left( \psi_Z(y_z,s') - \psi_W(y_w(s),s') \right)^{1/2} \, ds',
\end{split}
\end{equation*}
where we used \eqref{Energy} and \eqref{hbd} in the last inequality. To estimate $\psi_Z(y_z,s') - \psi_W(y_w(s),s')$, we note from \eqref{UW}, \eqref{UZ}, \eqref{WZbound}, and \eqref{kap_xi} that
\begin{equation*}
\left| U^W(y,s') - \frac{5}{2}y \right| + \left| U^Z(y,s') - \frac{5}{2}y \right| \leq C e^{3s'/2}.
\end{equation*}
Since $\psi_W(y_w(s),s) = \psi_Z(y_z,s)$, integration of \eqref{curves} yields
\begin{equation} \label{I_length}
0 \leq \psi_Z(y_z,s') - \psi_W(y_w(s),s') \leq C e^{5s'/2} \int_{s'}^s e^{-r} \, dr \leq C e^{3s'/2}
\end{equation}
for $s_0 \leq s' \leq s$. Substituting \eqref{I_length} into the estimate for the second term, we obtain
\begin{equation} \label{J1_second}
\begin{split}
\int_{y_w(s)}^{y_z} \int_{s_0}^{\overline{s}(\xi)} \left| e^{-5s'/2} \frac{\partial \psi_W}{\partial \xi} W_y \right| \, ds' d\xi & \leq C \int_{s_0}^s e^{-9s'/4} e^{3s'/4} \, ds' \\
& \leq C \int_{s_0}^{\infty} e^{-3s'/2} \, ds' \leq C \varepsilon.
\end{split}
\end{equation}
Combining \eqref{J1_first} and \eqref{J1_second} with \eqref{J1_1}, we conclude that $J_1 \leq C \varepsilon$, which, together with \eqref{Wy2bd}, gives the desired estimate \eqref{int_Wy2}.

We next estimate $J_2$. For this, we estimate each term on the right-hand side of \eqref{J2_1}. The bounds \eqref{lem_Zy}, \eqref{Uy1}, and \eqref{dottau1} give
\begin{equation*}
-5 - \frac{3W_y}{1-\dot{\tau}} + \frac{7e^{3s'/2}Z_y}{3(1-\dot{\tau})} \leq -\frac{3}{4}, \quad |y| \geq 1
\end{equation*}
for sufficiently small $\varepsilon>0$. This, together with \eqref{Wyy_bound} on $|y|\leq1$, yields
\begin{equation} \label{J2_first}
\begin{split}
& \int_{s_0}^s e^{-s'/2} \int_{\psi_W(y_w(s),s')}^{\psi_Z(y_z,s')} \left( -5 - \frac{3W_y}{1-\dot{\tau}} + \frac{7e^{3s'/2}Z_y}{3(1-\dot{\tau})} \right) W_{yy}^2 \, dy \, ds' \\
& \quad \leq -\frac{3}{4} \int_{s_0}^s e^{-s'/2} \int_{\psi_W(y_w(s),s')}^{\psi_Z(y_z,s')} W_{yy}^2 \, dy \, ds' + C\varepsilon^{1/2}.
\end{split}
\end{equation}
For the remaining terms, we use \eqref{lem_Zy}, \eqref{Z_high}, \eqref{Q_high}, and \eqref{dottau1}, together with Young's inequality, to obtain
\begin{equation} \label{J2_second}
\begin{split}
& \left| W_{yy} \left( -\frac{16e^{s'/2}Q_{yy}}{3(1-\dot{\tau})} + \frac{2e^{3s'}Z_yZ_{yy}}{1-\dot{\tau}} + \frac{8e^{3s'/2}Z_{yy}W_y}{3(1-\dot{\tau})} \right) \right| \\
& \quad \leq C|W_{yy}|\left( e^{-7s'/2} + e^{-s'} + |W_y| \right)  \leq \frac{1}{4}W_{yy}^2 + C\left( W_y^2 + e^{-2s'} \right).
\end{split}
\end{equation}
For the term $e^{-2s'}$, \eqref{I_length} gives
\begin{equation*}
\begin{split}
\int_{s_0}^s e^{-s'/2} \int_{\psi_W(y_w(s),s')}^{\psi_Z(y_z,s')} e^{-2s'} \, dy \, ds' & = \int_{s_0}^s e^{-5s'/2} \left( \psi_Z(y_z,s') - \psi_W(y_w(s),s') \right) \, ds' \\
& \leq C \int_{s_0}^{\infty} e^{-s'} \, ds' \leq C\varepsilon.
\end{split}
\end{equation*}
Applying Fubini's theorem to \eqref{J2_1} and combining \eqref{J2_first} and \eqref{J2_second} with this bound, we obtain
\begin{equation} \label{J2_bound}
\begin{split}
J_2 & \leq -\frac{1}{2} \int_{s_0}^s e^{-s'/2} \int_{\psi_W(y_w(s),s')}^{\psi_Z(y_z,s')} W_{yy}^2 \, dy \, ds'  + C \int_{s_0}^s e^{-s'/2} \int_{\psi_W(y_w(s),s')}^{\psi_Z(y_z,s')} W_y^2 \, dy \, ds' + C\varepsilon^{1/2}.
\end{split}
\end{equation}
Since the first term on the right-hand side of \eqref{J2_bound} is nonpositive, it remains to estimate the integral of $W_y^2$. By \eqref{asymp-y-infty} and \eqref{Wy_dec}, we have $|W_y(y,s')|\leq C(1+|y|^{2/5})^{-1}$. Hence, using \eqref{I_length}, we obtain
\begin{equation} \label{I_Wy2}
\begin{split}
\int_{\psi_W(y_w(s),s')}^{\psi_Z(y_z,s')} W_y^2(y,s') \, dy & \leq C \int_0^{\psi_Z(y_z,s')-\psi_W(y_w(s),s')} \frac{dy}{(1+y^{2/5})^2} \\
& \leq C \left( \psi_Z(y_z,s') - \psi_W(y_w(s),s') \right)^{1/5} \leq C e^{3s'/10}.
\end{split}
\end{equation}
Combining \eqref{J2_bound} and \eqref{I_Wy2}, we have
\begin{equation*}
J_2 \leq C\varepsilon^{1/2} + C \int_{s_0}^s e^{-s'/5} \, ds' \leq C\varepsilon^{1/5}.
\end{equation*}
Together with \eqref{Wyy2bd}, this proves \eqref{int_Wyy2}, uniformly in $y_z$.

We also record an estimate for use in Appendix~\ref{App:Z}. Combining \eqref{J2_bound} and \eqref{I_Wy2} with $J_2\geq-C\varepsilon^{1/2}$ from \eqref{Wyy2bd}, we obtain
\begin{equation} \label{int_Wyy2_I}
\int_{s_0}^s e^{-s'/2} \int_{\psi_W(y_w(s),s')}^{\psi_Z(y_z,s')} W_{yy}^2 \, dy \, ds' \leq C\varepsilon^{1/5}.
\end{equation}
Finally, using \eqref{int_Wy2} and \eqref{int_Wyy2} in \eqref{pr_Zyy'} and taking the supremum over $y\in\mathbb{R}$, we obtain \eqref{Zyy_proof} for sufficiently small $\varepsilon>0$.
\end{proof}

\begin{lemma}
	Under the same assumptions in Proposition~\ref{Boot}, it holds for all $s\in[s_0,\sigma_1]$ (equivalently, $t \in [-\varepsilon,T_{\sigma_1}]$) that
	\begin{equation}\label{Z}
		\left\lvert e^{3s/2}(\kappa - \dot{\xi}) + \frac{e^{3s/2}Z(y,s)}{3} \right\rvert \leq C \left( 1 + \lvert y \rvert  \right) e^{-s}
	\end{equation}
	for some constant $C>0$. 	Furthermore, we have
	\begin{equation}\label{UW_far}
		\inf_{|y|>1, s\in[s_0, \sigma_1]} U^W(y,s) \frac{y}{|y|}  \geq \frac{1}{8},
	\end{equation}
	where $U^W(y,s)$ is defined in \eqref{UW}. 
\end{lemma}

\begin{proof}
	By the fundamental theorem of calculus with \eqref{lem_Zy}, we have
	\begin{equation*}
	\lvert Z(y,s) - Z(0,s) \rvert \leq C \lvert y \rvert e^{-5s/2}.
	\end{equation*}
	Also, by \eqref{xi} with the bounds \eqref{Wy3_0}, \eqref{lem_Zy}, \eqref{Z_high'}, and \eqref{Q_high}, we deduce that
	\begin{equation*}
	\begin{split}
		\left\lvert \kappa - \dot{\xi} + \frac{Z(0,s)}{3} \right\rvert &  = \left\lvert \frac{e^{-3s/2}}{\partial_y^3 W (0,s)} \left( - \frac{8 e^{s/2} Q_{yy}(0,s)}{3} + e^{3s} Z_y(0,s) Z_{yy}(0,s) - \frac{8e^{3s/2} Z_{yy}(0,s)}{3} \right) \right\rvert \leq C e^{-5s/2}.
	\end{split}
	\end{equation*}
	Combining them with \eqref{dottau1}, we obtain
	\begin{equation*}
	\begin{split}
		\left\lvert \frac{e^{3s/2}(\kappa - \dot{\xi})}{1-\dot{\tau}} + \frac{e^{3s/2}Z}{3(1-\dot{\tau})} \right\rvert & \leq 2 e^{3s/2} \left( \left\lvert \kappa - \dot{\xi} + \frac{Z(0,s)}{3} \right\rvert + \left\lvert \frac{Z(y,s) - Z(0,s)}{3} \right\rvert \right) \leq C \left( 1 + \lvert y \rvert  \right) e^{-s} .
	\end{split}
	\end{equation*}
	
	Lastly, using \eqref{dottau1}, \eqref{constraint}, \eqref{Uy1} and \eqref{Z} for \eqref{UW}, one can check that  
	\begin{equation*}
		\begin{split}
			U^W(y,s) &\geq \frac{5}{2}y-\frac{2y}{1-\dot{\tau}}-\frac{C(1+y)e^{-s}}{1-\dot{\tau}}
			\ge  \frac{y}{4}-C\veps \quad \text{ for }y\geq 0,
		\end{split}
	\end{equation*}
	\begin{equation*}
		\begin{split}
			U^W(y,s) &\leq \frac{5}{2}y-\frac{2y}{1-\dot{\tau}}+\frac{C(1-y)e^{-s}}{1-\dot{\tau}}
			\leq \frac{y}{4}+C\veps \quad \text{ for }y<0.
		\end{split}
	\end{equation*}
	This proves the assertion \eqref{UW_far}.
\end{proof}

\begin{lemma}\label{lem_Qdec}
Under the same assumptions in Proposition~\ref{Boot}, there exists a constant $C>0$ such that
\begin{equation} \label{Qy_weight}
y^{4/5}|Q_y|\leq C e^{-s/2}
\end{equation}
for all $y\in\mathbb{R}$ and $s\in[s_0,\sigma_1]$.
\end{lemma}

\begin{proof}

By \eqref{ys}, it is enough to show that
\begin{equation}\label{qx_weight}
|x-\xi(t)|^{4/5}|q_x(x,t)|\leq C
\end{equation}
for some positive constant $C>0$. From \eqref{G} and \eqref{I_h}, the function $q$ satisfies
\begin{equation}\label{eq_q}
q- h^{-1}(h^3q_x)_x = G.
\end{equation}
Differentiating both sides of \eqref{eq_q} with respect to $x$, we obtain
\begin{equation*}
(1-\mathcal{L})q_x=G_x,
\end{equation*}
where $\mathcal{L} \coloneq  \partial_x \circ h^{-1} \partial_x \circ h^3 $. Using \eqref{G}, \eqref{WZQG}, \eqref{Wy_dec}, and \eqref{Zy_dec}, we estimate
\begin{equation}\label{Gx_temp}
\begin{split}
|G_x(x,t)|&\leq \frac{3}{16}e^{2s}W_y^2+\frac{5}{8}e^{7s/2}|W_y||Z_y|+\frac{3}{16}e^{5s}Z_y^2 \\
& \leq Ce^{2s}y^{-4/5}=C|x-\xi(t)|^{-4/5}.
\end{split}
\end{equation}

Let $K(x,z)$ be the Green function associated with the operator $(1-\mathcal{L})$, so that
\begin{equation} \label{repqx}
q_x(x,t)=\int_{\mathbb{R}}K(x,z)G_x(z,t)\,dz,
\end{equation}
see Appendix~\ref{App:K}. To estimate the quantity in \eqref{qx_weight}, we use the representation \eqref{repqx} and split the integral into two regions. Set $ \textstyle A(t):=\{z\in\mathbb{R}: |x-z|< \frac{|x-\xi(t)|}{2} \}$. This decomposition allows us to bound the left-hand side of \eqref{qx_weight} by
\begin{equation}\label{claim1}
\underbrace{|x-\xi(t)|^{4/5}\left| \int_AK(x,z)G_x(z,t)\,dz\right|}_{=:I_A}+\underbrace{|x-\xi(t)|^{4/5}\left| \int_{A^c}K(x,z)G_x(z,t)\,dz\right|}_{=:I_{A^c}}\leq C,
\end{equation}
which suffices to establish \eqref{qx_weight}. Since $\textstyle \frac{|x-\xi(t)|}{2} \leq |z-\xi(t)|$ for $z\in A$, we estimate $I_A$ as
\begin{equation}\label{IA}
\begin{split}
I_A & \leq \left|\int_A K(x,z)|x-\xi(t)|^{4/5}G_x(z,t)\,d z\right| \\
& \leq C\int_A |K(x,z)|\frac{|x-\xi(t)|^{4/5}}{|z-\xi(t)|^{4/5}}\,dz \leq C\int_\mathbb{R} | K(x,z) | \,dz \leq C,
\end{split}
\end{equation}
where we used \eqref{Gx_temp} and \eqref{Kbound}. Next, for $z\in A^c$, we observe that $2|x-z|\geq |x-\xi(t)|$. Applying \eqref{Kbound}, we obtain
\begin{equation}\label{IAC}
\begin{split}
I_{A^c} & \leq 2^{4/5}\int_{A^c} |K(x,z)| |x-z|^{4/5} |G_x(z,t)|\,dz \\
& \leq C \sup_{(x,z)\in\mathbb{R}^2}\left( |K(x,z)| |x-z|^{4/5} \right) \int_\mathbb{R} |G_x(z,t)| \, dz \leq C \int_\mathbb{R} |G_x(z,t)| \, dz. 
\end{split}
\end{equation}
We recall from \eqref{G} that $\textstyle G_x=u_x^2-\frac{h_x^2}{4h}$. Together with \eqref{def_E} and \eqref{Energy}, this gives
\begin{equation}
\int_{\mathbb{R}} \left| G_x(z,t) \right| \,dz \leq \frac{1}{h_{\mathrm{min}}^3} \int_{\mathbb{R}} \left| h^3\left(u_x^2-\frac{h_x^2}{4h}\right)(z,t) \right| \,dz \leq C.
\end{equation}
Substituting this into \eqref{IAC}, we obtain
\begin{equation}\label{IAC'}
I_{A^c}\leq C.
\end{equation}
Therefore, \eqref{claim1} follows from \eqref{IA} and \eqref{IAC'}.
	
\end{proof}

\subsection{Proof of (\ref{dottau_close}) and (\ref{lem_Zy_close})--(\ref{Zyweight_close}) in Proposition~\ref{Boot}} \label{sec:3.2}

We are now ready to prove Proposition~\ref{Boot}. We first close the bootstrap assumption for $\dot{\tau}$.

\begin{lemma} \label{lem:tau_decay}
Under the same assumptions in Proposition~\ref{Boot}, we have
\begin{equation} \label{tau_decay}
\lvert \dot{\tau}(t) \rvert \leq \frac{7}{\sqrt{h_*}} e^{-s}
\end{equation}
for all $t \in [-\varepsilon, T_{\sigma_1}]$.
\end{lemma}

\begin{proof}
Applying the bounds \eqref{GQ_bound} and \eqref{lem_Zy} to \eqref{tau}, we obtain
\begin{equation*}
\begin{split}
|\dot{\tau}| & \leq \frac{4}{3} \left| e^{s/2} Q_y(0,s)\right| + \frac{1}{4} \left| e^{3s} Z_y^2(0,s)\right| + \frac{4}{3} \left| e^{3s/2} Z_y(0,s) \right| \\
& \leq C e^{-2s} + \frac{20}{3\sqrt{h_*}} e^{-s} \leq \left( \frac{20}{3\sqrt{h_*}} + C \varepsilon \right) e^{-s} \leq \frac{7}{\sqrt{h_*}}e^{-s}
\end{split}
\end{equation*}
for sufficiently small $\varepsilon>0$. 
\end{proof}

We next close \eqref{lem_Zy}.

\begin{lemma} \label{lem:Zy}
Under the same assumptions in Proposition~\ref{Boot}, we have
\begin{equation} \label{lemm_Zy}
\begin{split}
& \lVert Z_{y} (\cdot,s) \rVert_{L^\infty} \leq \frac{4}{\sqrt{h_*}} e^{-5s/2}
\end{split}
\end{equation}
for all $s \in [s_0,\sigma_1]$.
\end{lemma}

\begin{proof}
We recall the equation \eqref{eq:Z_high} with \eqref{D^Z}--\eqref{F^Z} for $n=1$:
\begin{equation} \label{Zyeq}
\partial_s Z_y + D_1^Z Z_y + U^Z Z_{yy} = F_1^Z,
\end{equation}
where
\begin{equation*}
D_1^Z(y,s) = \frac{5}{2}, \quad F_1^Z(y,s) = - \frac{8e^{-s}Q_y}{3(1-\dot{\tau})}+ \frac{e^{-3s/2}W_y^2}{2(1-\dot{\tau})}+ \frac{4W_y Z_y}{3(1-\dot{\tau})} - \frac{e^{3s/2}Z_y^2}{2(1-\dot{\tau})}.
\end{equation*}
To estimate the third term in $F_1^Z$, we apply Young's inequality together with the bound \eqref{lem_Zy}:
\begin{equation*}
|Z_y||W_y| \leq \frac{h_*^{1/3}}{2} Z_y^{2/3}W_y^2 + (2 h_*^{1/3} )^{-1} Z_y^{4/3} \leq \frac{5^{2/3}}{2}  e^{-5s/3} W_y^2 + \frac{5^{4/3}}{2  h_*} e^{-10s/3}.
\end{equation*}
This yields
\begin{equation*}
\begin{split}
|F_1^Z|  \leq C e^{-10s/3} + \left( \frac{1}{2} + C \veps^{1/6}  \right) e^{-3s/2} W_y^2,
\end{split}
\end{equation*}
where we also used \eqref{dottau}, \eqref{lem_Zy}, \eqref{GQ_bound}, and \eqref{dottau1}. Integrating the equation \eqref{Zyeq} in \eqref{eq:Z_high} along the characteristic curve $\psi_Z$ defined in \eqref{curves}, we obtain
\begin{equation*}
\begin{split}
| Z_y(\psi_Z(y,s),s) | & \leq | Z_{y}(y,s_0) | e^{-5(s-s_0)/2} + C \int_{s_0}^s e^{ -5(s-s')/2} e^{-10s'/3} \, ds' \\
& \quad + \left( \frac{1}{2} + C \varepsilon^{1/6} \right) \int_{s_0}^s e^{-5(s-s')/2} e^{-3s'/2} W_y^2(\psi_Z(y,s'),s') \, ds' \\
& \leq \frac{1}{\sqrt{h_*}} e^{-5s/2} + C \varepsilon^{5/6} e^{-5s/2} + \left( \frac{1}{2}  + C \varepsilon^{1/6} \right) e^{-5s/2} \int_{s_0}^s e^{s'} W_y^2(\psi_Z(y,s'),s') \, ds',
\end{split}
\end{equation*}
using the initial bound \eqref{lem_Zy_init} in the last inequality. Finally, applying \eqref{int_Wy2} from Lemma~\ref{Lemma:Z_high}, we obtain the desired bound \eqref{lemm_Zy}:
\begin{equation*}
\begin{split}
| Z_y(\psi_Z(y,s),s) | & \leq \left( \frac{1}{\sqrt{h_*}} + \frac{2\sqrt{2}}{\sqrt{h_*}} + C \varepsilon^{1/6} \right) e^{-5s/2}  \leq \frac{4}{\sqrt{h_*}} e^{-5s/2}
\end{split}
\end{equation*}
for sufficiently small $\varepsilon >0$.
\end{proof}

To prove \eqref{Zyweight_close}, we let $\zeta(y,s):=(y^{2/5}+1)e^{3s/2}Z_y(y,s)$. Then, the equation for $Z_y$ in \eqref{eq:Z_high} can be rewritten in terms of $\zeta$ for $y\ne0$ as
\begin{equation}\label{zeta_eq}
	\partial_s\zeta + D^\zeta \zeta + U^Z \zeta_y = F^\zeta_1+F^\zeta_2,
\end{equation}	
where $U^Z$ is defined in \eqref{UZ}, and 
\begin{equation*}
	\begin{split}
		&D^\zeta(y,s):= 1 -\frac{y^{2/5}}{y^{2/5}+1},
		\\
		&F^\zeta_1(y,s):= - \zeta \left(\frac{e^{3s/2} Z_y}{2(1-\dot{\tau})} - \frac{4 W_y}{3(1-\dot{\tau})} - \frac{2}{5y^{3/5}(y^{2/5}+1)} \left( \frac{e^{3s/2} ( \kappa - 3 \dot{\xi})}{3(1-\dot{\tau})} +\frac{e^{3s/2}Z}{1-\dot{\tau}} +  \frac{W}{3(1-\dot{\tau})}\right) \right),
		\\
		&F^\zeta_2(y,s):= (y^{2/5}+1)\frac{W_y^2}{2(1-\dot{\tau})} - (y^{2/5}+1)\frac{8 e^{s/2}Q_y}{3(1-\dot{\tau})}.
	\end{split}
\end{equation*}

We now close the assumption \eqref{Zy_dec} in the next lemma.

\begin{lemma}\label{lem_Zyweight}
	Under the same assumptions in Proposition~\ref{Boot}, it holds that
	\begin{equation*}
		|Z_y(y,s)|\leq \frac{8}{\sqrt{h_*}}\frac{e^{-3s/2}}{1+y^{2/5}}
	\end{equation*}
	for all $y \in \mathbb{R}$ and $s \in [s_0,\sigma_1]$.
\end{lemma}

\begin{proof}
We evaluate $\zeta$ along the characteristic curve $\psi_Z$ defined in \eqref{curves} and split the proof into the cases $|\psi_Z(y,s)|\leq e^{5s/2}$ and $|\psi_Z(y,s)|>e^{5s/2}$. To this end, we first estimate $\zeta$ on $|y|\leq e^{5s/2}$ and $D^\zeta$, $F^\zeta_1$, and $F^\zeta_2$ on $|y|\geq e^{5s/2}$.

By \eqref{lem_Zy}, we have
\begin{equation}\label{EP2_1D1-p'}
|\zeta(y,s)|\leq (y^{2/5}e^{-s}+e^{-s})\|e^{5s/2}Z_y(\cdot,s)\|_{L^{\infty}} < \frac{6}{\sqrt{h_*}} \quad \text{for }|y|\leq e^{5s/2}.
\end{equation}
By the definition of $D^\zeta$, we have
\begin{equation}\label{D_zeta'}
D^\zeta = 1- \frac{y^{2/5}}{y^{2/5}+1} \geq 0 \quad \text{for all }(y,s)\in\mathbb{R}\times[s_0,\sigma_1].
\end{equation}
We next estimate $F^\zeta_1$ and $F^\zeta_2$ on $|y|\geq e^{5s/2}$. In the following, all inequalities are understood to hold in this region.

By \eqref{Wy_dec} and \eqref{4.0'}, we have
\begin{equation}\label{Wy_dec'}
|W_y|\leq |\wt{W}_y|+|\overline{W}'|\leq \frac{C}{y^{2/5}}\leq Ce^{-s}.
\end{equation}
Hence, using $W(0,s)=0$, we obtain
\begin{equation}\label{Wy_temp}
\begin{split}
\left|\frac{W}{y}\right| \leq \frac{1}{|y|}\left|\int_0^y |W_y(y',s)|\,dy'\right| \leq \frac{C}{|y|}\int_0^{|y|}\frac{1}{y'^{2/5}}\,dy' = C|y|^{-2/5}\leq Ce^{-s}.
\end{split}
\end{equation}
We also observe that, by \eqref{WZbound} and \eqref{kap_xi},
\begin{equation}\label{modul_temp}
\left|\frac{e^{3s/2}}{y^{3/5}(y^{2/5}+1)}\left(\frac{\kappa-3\dot{\xi}}{3}+Z\right)\right| \leq Ce^{3s/2}|y|^{-1}\leq Ce^{-s}.
\end{equation}
It follows from \eqref{Wy_dec'}--\eqref{modul_temp} and \eqref{dottau}--\eqref{Zy_dec} that $\|F^\zeta_1\|_{L^{\infty}(|y|\geq e^{5s/2})}\leq Ce^{-s}$.
In addition, by \eqref{GQ_bound} and \eqref{Qy_weight}, we have
\begin{equation*}
|(y^{2/5}+1)e^{s/2}Q_y|\leq y^{2/5}e^{s/2}|Q_y|+e^{s/2}|Q_y|\leq Cy^{-2/5}+Ce^{-2s}\leq Ce^{-s}.
\end{equation*}
Hence, using \eqref{Uy1} and \eqref{Wy_dec'}, we obtain $\|F^\zeta_2\|_{L^{\infty}(|y|\geq e^{5s/2})}\leq Ce^{-s}$. Combining the above estimates yields
\begin{equation}\label{F_zeta'}
\begin{split}
\|F^\zeta_1\|_{L^{\infty}(|y|\geq e^{5s/2})}+\|F^\zeta_2\|_{L^{\infty}(|y|\geq e^{5s/2})}&\leq Ce^{-s}.
\end{split}
\end{equation}

We now fix $y\in\mathbb{R}$ and apply these estimates along $\psi_Z$. If $|\psi_Z(y,s)|\leq e^{5s/2}$ for all $s\in[s_0,\sigma_1]$, then the desired bound along $\psi_Z$ follows from \eqref{EP2_1D1-p'}. Otherwise, by continuity, there is a nonempty interval on which $|\psi_Z(y,r)|>e^{5r/2}$. Let $(s_1,s_2)$ be any maximal such interval, with $s_0\leq s_1<s_2\leq\sigma_1$. At the left endpoint, either $s_1=s_0$ or $|\psi_Z(y,s_1)|=e^{5s_1/2}$. Thus, \eqref{Zy_weight_init} and \eqref{EP2_1D1-p'} give $|\zeta(\psi_Z(y,s_1),s_1)|\leq 6/\sqrt{h_*}$. For any fixed $s\in(s_1,s_2)$, integrating \eqref{zeta_eq} along $\psi_Z$ from $s_1$ to $s$ and using \eqref{D_zeta'} and \eqref{F_zeta'}, we obtain
\begin{equation*}
\begin{split}
|\zeta(\psi_Z(y,s),s)| & \leq |\zeta(\psi_Z(y,s_1),s_1)|e^{-\int_{s_1}^s D^\zeta(\psi_Z(y,r),r)\,dr} \\
& \quad + \int_{s_1}^s e^{-\int_r^s D^\zeta(\psi_Z(y,r'),r')\,dr'}\left|(F^\zeta_1+F^\zeta_2)(\psi_Z(y,r),r)\right|\,dr \\
& \leq |\zeta(\psi_Z(y,s_1),s_1)|+C\int_{s_1}^s e^{-r}\,dr \leq \frac{6}{\sqrt{h_*}}+C\varepsilon<\frac{8}{\sqrt{h_*}}
\end{split}
\end{equation*}
for sufficiently small $\varepsilon>0$. Since $y$, $(s_1,s_2)$, and $s$ are arbitrary, the two cases give the desired bound along every characteristic curve, with $s=s_0$ and $s=\sigma_1$ included by continuity. Since $\psi_Z(\cdot,s)$ maps $\mathbb{R}$ onto $\mathbb{R}$, this completes the proof.
\end{proof}

\subsection{Proof of (\ref{Utildey_close})--(\ref{Wy4_close}) in Proposition~\ref{Boot}} \label{sec:3.3}

We close the remaining bootstrap assumptions \eqref{Wy_bound}--\eqref{Wy4_bound} to complete the proof of Proposition~\ref{Boot}. In the proofs of the subsequent lemmas, the constant $M>0$ is chosen sufficiently large first. All smallness conditions on $\varepsilon$ are then understood to possibly depend on this choice of $M$.

\begin{lemma} \label{lem:Wy30}
Under the same assumptions in Proposition~\ref{Boot}, there exists a constant $C>0$ such that
\begin{equation}\label{str_U3-p}
	\lvert \partial_y^3 W(0,s) - 256 \rvert \leq C \varepsilon
\end{equation}
for all $s \in [s_0,\sigma_1]$.
\end{lemma}

\begin{proof}
Plugging $y=0$ into \eqref{Wy3}, we have
\begin{equation*}
\begin{split}
\partial_s \partial_y^3 W(0,s) &= \left(\frac{6\dot{\tau}}{1-\dot{\tau}} + \frac{2e^{3s/2} Z_y (0,s)}{3(1-\dot{\tau})}\right) \partial_y^3 W(0,s) -\left(\frac{e^{3s/2}Z(0,s)}{3(1-\dot{\tau})}+\frac{e^{3s/2}(\kappa-\dot{\xi})}{1-\dot{\tau}}\right) \partial_y^4W(0,s) + F_3^W(0,s),
\end{split}
\end{equation*}
where
\begin{equation*}
F_3^W(0,s) := - \frac{8 e^{s/2} \partial_y^3 Q(0,s)}{3(1-\dot{\tau})} + \frac{e^{3s} Z_{yy}^2(0,s)}{1-\dot{\tau}} + \frac{e^{3s} Z_y(0,s) \partial_y^3 Z(0,s)}{1-\dot{\tau}}  - \frac{8e^{3s/2} \partial_y^3 Z(0,s)}{3(1-\dot{\tau})}.
\end{equation*}
First, by \eqref{tau_decay}, \eqref{dottau1}, \eqref{lem_Zy}, and \eqref{Wy3_0}, the following bound holds:
\begin{equation*}
\left| \frac{6\dot{\tau}}{1-\dot{\tau}} + \frac{2e^{3s/2} Z_y (0,s)}{3(1-\dot{\tau})}\right| \left|\partial_y^3 W(0,s) \right| \leq C e^{-s}.
\end{equation*}
Also, applying the bounds \eqref{Wy3_0}, \eqref{lem_Zy}, \eqref{Z_high'}, and \eqref{Q_high} to the equation \eqref{xi}, we obtain
\begin{equation} \label{kxiZ}
\begin{split}
	\left\lvert \kappa - \dot{\xi} + \frac{Z(0,s)}{3} \right\rvert &  = \left\lvert \frac{e^{-3s/2}}{\partial_y^3 W (0,s)} \left( e^{3s} Z_y(0,s) Z_{yy}(0,s) - \frac{8 e^{s/2} Q_{yy}(0,s)}{3} - \frac{8e^{3s/2} Z_{yy}(0,s)}{3} \right) \right\rvert \\
	& \leq C e^{-5s/2}.
\end{split}
\end{equation}
Using \eqref{kxiZ}, \eqref{Wy4_bound}, and \eqref{dottau1}, we get
\begin{equation*}
	\left|\left(\frac{e^{3s/2}Z(0,s)}{3(1-\dot{\tau})}+\frac{e^{3s/2}(\kappa-\dot{\xi})}{1-\dot{\tau}}\right) \partial_y^4W(0,s)\right| \leq Ce^{-s}.
\end{equation*}
We further obtain
\begin{equation*}
\lvert F_3^W(0,s) \rvert \leq C \left( e^{-2s} + e^{-s}  \right) \leq C e^{-s}
\end{equation*}
by using the bounds \eqref{Q_high}, \eqref{Z_high'}, \eqref{lem_Zy}, and \eqref{dottau1}. Combining these estimates, we obtain
\begin{equation*}
\lvert \partial_y^3W(0,s) - 256 \rvert = \lvert \partial_y^3 W(0,s) - \partial_y^3 W(0,s_0) \rvert \leq \int_{s_0}^s \lvert \partial_s \partial_y^3 W(0,s') \rvert \, ds' \leq C \int_{s_0}^s e^{-s'} \, ds' \leq C \varepsilon.
\end{equation*}

\end{proof}

\begin{lemma} \label{lem:Wtildey}
Under the same assumptions in Proposition~\ref{Boot}, it holds that
\begin{equation*}
\lvert W_y(y,s) - \overline{W}'(y) \rvert \leq \frac{y^2}{1500(1+y^2)}
\end{equation*}
for all $y \in \mathbb{R}$ and $s \in [s_0,\sigma_1]$.
\end{lemma}

\begin{proof}
Let $\widetilde{W} := W - \overline{W}$. Then, from \eqref{Wy1} and \eqref{Weq}, we have the equation for $\wt{W}_y(y,s)$:
\begin{multline}\label{Wytilde}
	 \partial_s \widetilde{W}_y + \left( 1 + \frac{\widetilde{W}_y + 2 \overline{W}'}{2(1-\dot{\tau})} - \frac{4e^{3s/2} Z_y }{3(1-\dot{\tau})} \right) \widetilde{W}_y + U^W \widetilde{W}_{yy} 
	= - \left( \frac{\dot{\tau} \overline{W}'}{2(1-\dot{\tau})} - \frac{4e^{3s/2} Z_y }{3(1-\dot{\tau})} \right) \overline{W}' 
	\\
	- \left( \frac{\widetilde{W} + \dot{\tau}\overline{W}}{1-\dot{\tau}} + \frac{e^{3s/2}(\kappa - \dot{\xi})}{1-\dot{\tau}} + \frac{e^{3s/2}Z}{3(1-\dot{\tau})} \right) \overline{W}'' + \frac{e^{3s} Z_y^2}{2(1-\dot{\tau})} - \frac{8 e^{s/2} Q_y}{3(1-\dot{\tau})},
\end{multline}
where $U^W$ is defined in \eqref{UW}. Defining a new weighted function
\begin{equation*}
V(y,s) := \frac{y^2 + 1}{y^2} \widetilde{W}_y(y,s),
\end{equation*}
we obtain from \eqref{Wytilde} that, for $y\ne0$,
\begin{equation*}
\partial_s V + D^V V + U^W V_y = F^V + \int_\mathbb{R} V(y',s) K^V(y,s;y') \, dy',
\end{equation*}
where
\begin{equation*}
\begin{split}
D^V(y,s) & := 1 + \frac{\widetilde{W}_y + 2 \overline{W}'}{2(1-\dot{\tau})} - \frac{4e^{3s/2} Z_y }{3(1-\dot{\tau})} + \frac{2}{y(1+y^2)} U^W, \\
F^V(y,s) & := - \frac{y^2 + 1}{y^2} \left( \frac{\dot{\tau} \overline{W}'}{2(1-\dot{\tau})} - \frac{4e^{3s/2} Z_y }{3(1-\dot{\tau})} \right) \overline{W}' \\
& \quad - \frac{y^2 + 1}{y^2} \left( \frac{ \dot{\tau}W}{1-\dot{\tau}} + \frac{e^{3s/2}(\kappa - \dot{\xi})}{1-\dot{\tau}} + \frac{e^{3s/2}Z}{3(1-\dot{\tau})} \right) \overline{W}'' + \frac{y^2+1}{y^2} \left( \frac{e^{3s} Z_y^2}{2(1-\dot{\tau})} - \frac{8 e^{s/2} Q_y}{3(1-\dot{\tau})} \right), \\
K^V(y,s;y') & := - \frac{y^2+1}{y^2} \overline{W}''(y) \mathbb{I}_{[0,y]} (y') \frac{y'^2}{1+y'^2},
\end{split}
\end{equation*}
where $\mathbb{I}_{[0,y]}$ denotes the oriented indicator of the interval from $0$ to $y$, namely
\begin{equation*}
\mathbb{I}_{[0,y]}(y') =
\begin{cases}
1, & 0\leq y'\leq y,\quad y>0,\\
-1, & y\leq y'\leq 0,\quad y<0,\\
0, & \text{otherwise}.
\end{cases}
\end{equation*}

First, let us estimate $V(y,s)$ near $y=0$. By the Taylor expansion, we write
\begin{equation*}
	W_y(y,s) = - 2 + \frac{y^2}{2} \partial_y^3 W(0,s) + \frac{y^3}{6} \partial_y^4 W(y',s) \quad \text{ for some } |y'|<|y|<l,
\end{equation*} 
where $l :=1/(1000M)>0$ for sufficiently large $M>0$. Using \eqref{Wy4_bound} and \eqref{str_U3-p}, we obtain
\begin{equation*}
	\begin{split}
		| W_y(y, s) + 2 - 128 y^2 | 
		& = \left| \frac{y^2}{2}\left(\partial_y^3W( 0, s) - 256 \right)+\frac{y^3}{6}\partial_y^4W(y',s)\right|
		\leq C\veps  y^2 + \frac{M}{6}  |y|^3  \leq y^2\left( C \veps+ \frac{ M | l | }{ 6  } \right)
	\end{split}
\end{equation*}
for all $| y | \le l$. By this together with \eqref{W_taylor_small}, we deduce that
\begin{equation*}
\begin{split}
|W_y(y, s) -\overline{W}'(y) |&\leq y^2\left(C\veps+\frac{M|l|}{6}\right)+C\cdot y^4 \\
& \leq y^2\left(C\veps+\frac{M|l|}{6}+C\cdot l^2\right)\leq \frac{ y^2}{3000(1+y^2)} \quad \text{for}\quad |y|\leq l,
	\end{split}
\end{equation*}
where we used $Ml/6=1/6000$ in the last inequality. This yields
\begin{equation}\label{V-l-p}
	|V(y,s)|\leq \frac{1}{3000}  \quad \text{for} \quad |y|\leq l.
\end{equation}

Next, we derive a positive lower bound for $D^V(y,s)$ in $|y|\geq l$. For this, we use the estimate \eqref{num_2}: 
\begin{equation}\label{DV0}
	1+\overline{W}'+\frac{2}{1+y^2}\left(\frac{5}{2}+\frac{\overline{W}}{y}\right)\geq \frac{y^2}{5(1+y^2)}, \quad y\in\mathbb{R}.
\end{equation}
To apply the above inequality, we split $D^V(y,s)$ as follows: 
\begin{equation}\label{DV}
	\begin{split}
		D^V(y,s) & = 1 + \frac{\widetilde{W}_y + 2 \overline{W}'}{2(1-\dot{\tau})} - \frac{4 e^{3s/2}Z_y}{3(1-\dot{\tau})} + \frac{2}{y(1+y^2)} \left( \frac{e^{3s/2}(\kappa - \dot{\xi})}{1-\dot{\tau}} + \frac{5}{2}y + \frac{ \widetilde{W} + \overline{W}}{1-\dot{\tau}} + \frac{e^{3s/2}Z}{3(1-\dot{\tau})} \right) \\
		& \geq \left(1 + \overline{W}' + \frac{2}{1+y^2} \left( \frac{5}{2} + \frac{\overline{W}}{y} \right)\right) - \left\lvert \frac{\widetilde{W}_y + 2\dot{\tau} \overline{W}'}{2(1-\dot{\tau})}\right\lvert -\left\lvert \frac{4e^{3s/2}Z_y}{3(1-\dot{\tau})} \right\lvert \\
		& \quad - \left\lvert \frac{2}{y(1+y^2)}\frac{\wt{W}}{1-\dot{\tau}}\right\lvert-\left\lvert \frac{2}{y(1+y^2)} \left( \frac{\dot{\tau} \overline{W}}{1-\dot{\tau}} + \frac{e^{3s/2}(\kappa - \dot{\xi})}{1-\dot{\tau}}  + \frac{e^{3s/2}Z}{3(1-\dot{\tau})} \right) \right\rvert.
	\end{split}
\end{equation}
We estimate each term in the right-hand side of the inequality. By \eqref{tau_decay}, \eqref{dottau1}, \eqref{Wy_bound}, and the bound $|\overline{W}'|\leq 2$ from \eqref{4.0}, we obtain
\begin{equation}\label{DV1}
	\bigg\lvert \frac{\widetilde{W}_y + 2 \dot{\tau} \overline{W}'}{2(1-\dot{\tau})} \bigg\rvert \leq \frac{1+\varepsilon^{1/2}}{2} \left( \frac{y^2}{1000(1+y^2)} + O(\varepsilon) \right).
\end{equation}
Using \eqref{dottau1}, \eqref{lem_Zy}, we have 
\begin{equation}\label{DV2}
	\left|\frac{4e^{3s/2}Z_y}{3(1-\dot{\tau})}\right|\leq Ce^{-s}.
\end{equation}
From \eqref{Wy_bound} and the fact $\wt{W}(0,s)=0$, it holds that
\begin{equation*}
\lvert \widetilde{W} (y,s) \rvert \leq \left\lvert\int_0^y \lvert \widetilde{W}_y(y',s) \rvert \,dy'\right\rvert \leq \frac{1}{1000} \int_0^{|y|} \frac{y'^2}{1+y'^2} \, dy'.
\end{equation*}
Thus, combining this with \eqref{dottau1}, we obtain
\begin{equation}\label{DV3}
	\begin{split}
		\left\lvert \frac{2}{y(1+y^2)} \frac{\widetilde{W}}{1-\dot{\tau}} \right\rvert &\leq  (1+\varepsilon^{1/2}) \left\lvert \frac{2}{y(1+y^2)} \right\rvert \left\lvert \frac{1}{1000} \int_0^{|y|} \frac{y'^2}{1+y'^2} \, dy' \right\rvert 
		\\
		&\leq \left\lvert \frac{1+\veps^{1/2}}{500y(1+y^2)} \right\rvert \left\lvert \int_0^{|y|} \frac{y'^2}{1+y'^2} \, dy' \right\rvert \leq  \frac{(1+\veps^{1/2})y^2}{1500(1+y^2)}, \quad \lvert y \rvert \geq l.
	\end{split}
\end{equation}
For the last term, we use \eqref{4.00}, \eqref{dottau1}, \eqref{tau_decay}, and \eqref{Z} to obtain
	\begin{equation}\label{DV4}
		\begin{split}
			\left\lvert \frac{2}{y(1+y^2)} \left( \frac{\dot{\tau} \overline{W}}{1-\dot{\tau}} + \frac{e^{3s/2} (\kappa - \dot{\xi})}{1-\dot{\tau}}+\frac{e^{3s/2}Z}{3(1-\dot{\tau})} \right) \right\rvert 
			& \quad \leq C\left( \frac{\lvert \dot{\tau} \overline{W} \rvert}{|y|} + \left| \frac{e^{3s/2} (\kappa - \dot{\xi})}{y} +\frac{e^{3s/2}Z}{3y}\right| \right) \\
			& \quad \leq  O(\varepsilon), \quad \lvert y \rvert \geq l.
		\end{split}
	\end{equation}
Applying \eqref{DV0} and \eqref{DV1}--\eqref{DV4} to \eqref{DV}, we conclude that
\begin{equation}\label{damp-p}
	\begin{split}
		D^V(y,s) \geq \frac{y^2}{5(1+y^2)} - \frac{y^2}{2000(1+y^2)} - \frac{y^2}{1500(1+y^2)} - O(\varepsilon^{1/2}) \geq \frac{l^2}{10(1+l^2)}, \quad \lvert y \rvert \geq l.
	\end{split}
\end{equation}

We proceed to compare $K^V$ with $D^V$. From \eqref{num_6}, we have 
\begin{equation}\label{KVDV}
	\begin{split}
		\int_\mathbb{R} \lvert K^V (y,s;y') \rvert \, dy' & \leq \frac{y^2+1}{y^2} \lvert \overline{W}''(y) \rvert \int_0^{\lvert y \rvert} \frac{y'^2}{1+y'^2} \, dy' \\
		& \leq \delta \left( 1+ \overline{W}'(y) + \frac{2}{1+y^2} \left( \frac{5}{2} + \frac{\overline{W}(y)}{y} \right) - \frac{y^2}{500(1+y^2)} \right) \leq \delta D^V(y,s), \quad \lvert y \rvert \geq l
	\end{split}
\end{equation}
for some constant $\delta \in (0,1)$. Here, we used \eqref{DV}--\eqref{DV4} in the last inequality.

We now estimate the forcing term $F^V$. Using \eqref{dottau1}, we have
\begin{equation*}
\begin{split}
& \lVert F^V(\cdot,s) \rVert_{L^\infty(\lvert y \rvert \geq l)} \\
& \quad \leq \left\lVert \frac{y^2 + 1}{y^2} \left( \frac{\dot{\tau} \overline{W}'}{2(1-\dot{\tau})} - \frac{4e^{3s/2} Z_y }{3(1-\dot{\tau})} \right) \overline{W}' \right\rVert_{L^\infty(\lvert y \rvert \geq l)} + \lvert \dot{\tau} \rvert \left\lVert \frac{y^2 + 1}{y^2}  \frac{ W \overline{W}''}{1-\dot{\tau}}   \right\rVert_{L^\infty(\lvert y \rvert \geq l)} \\
& \qquad + \left\lVert \frac{y^2 + 1}{y^2} \left(  \frac{e^{3s/2}(\kappa - \dot{\xi})}{1-\dot{\tau}} + \frac{e^{3s/2}Z}{3(1-\dot{\tau})} \right) \overline{W}'' \right\rVert_{L^\infty(\lvert y \rvert \geq l)} + \left\lVert \frac{y^2+1}{y^2} \left( \frac{e^{3s} Z_y^2}{2(1-\dot{\tau})} - \frac{8 e^{s/2} Q_y}{3(1-\dot{\tau})} \right) \right\rVert_{L^\infty(\lvert y \rvert \geq l)} \\
& \quad \leq C  \left(  \lvert \dot{\tau} \rvert \lVert \overline{W}' \rVert_{L^\infty(\mathbb{R})} + e^{3s/2} \lVert Z_y \rVert_{L^\infty(\mathbb{R})} \right) \lVert \overline{W}' \rVert_{L^\infty(\mathbb{R})} + C  \lvert \dot{\tau} \rvert \left\|\frac{W}{y} \right\|_{L^\infty(|y|\geq l)}\lVert y\overline{W}'' \rVert_{L^\infty(\mathbb{R})} 
\\
& \qquad  + C\left\|\frac{e^{3s/2}(\kappa - \dot{\xi}) + \frac{1}{3}e^{3s/2}Z}{y}\right\|_{L^{\infty}(|y|\geq l)}\lVert y\overline{W}'' \rVert_{L^\infty(\mathbb{R})} + C \left( e^{3s} \lVert Z_y \rVert_{L^\infty(\mathbb{R})}^2 + e^{s/2} \lVert Q_y \rVert_{L^\infty(\mathbb{R})} \right).
\end{split}
\end{equation*}
Each term on the right-hand side is estimated separately below.

By \eqref{tau_decay}, \eqref{4.0}, and \eqref{lem_Zy}, we obtain
\begin{equation}\label{FV1}
C  \left(  \lvert \dot{\tau} \rvert \lVert \overline{W}' \rVert_{L^\infty(\mathbb{R})} + e^{3s/2} \lVert Z_y \rVert_{L^\infty(\mathbb{R})} \right) \lVert \overline{W}' \rVert_{L^\infty(\mathbb{R})}\leq Ce^{-s}.
\end{equation}
Using \eqref{tau_decay}, \eqref{4.0''}, and the bound
\begin{equation*}
\left\|\frac{W}{y} \right\|_{L^\infty(|y|\geq l)} \leq 2,
\end{equation*}
coming from \eqref{Uy1} and \eqref{constraint}, we deduce
\begin{equation}\label{FV2}
C \lvert \dot{\tau} \rvert \left\|\frac{W}{y} \right\|_{L^\infty(|y|\geq l)}\lVert y\overline{W}'' \rVert_{L^\infty(\mathbb{R})} \leq Ce^{-s}.
\end{equation}
From \eqref{Z} and \eqref{4.0''}, it follows that
\begin{equation}\label{FV3}
C\left\|\frac{e^{3s/2}(\kappa - \dot{\xi}) + \frac{1}{3}e^{3s/2}Z}{y}\right\|_{L^{\infty}(|y|\geq l)}\lVert y\overline{W}'' \rVert_{L^\infty(\mathbb{R})} \leq Ce^{-s}.
\end{equation}
Applying \eqref{lem_Zy} and \eqref{GQ_bound}, we have
\begin{equation}\label{FV4}
C \left( e^{3s} \lVert Z_y \rVert_{L^\infty(\mathbb{R})}^2 + e^{s/2} \lVert Q_y \rVert_{L^\infty(\mathbb{R})} \right)\leq Ce^{-s}.
\end{equation}
Combining \eqref{FV1}--\eqref{FV4}, we conclude that
\begin{equation}\label{FV}
	\lVert F^V(\cdot,s) \rVert_{L^\infty(\lvert y \rvert \geq l)} \leq Ce^{-s}\leq C\veps
\end{equation}
for sufficiently small $\veps>0$.

Lastly, from \eqref{W-y2} and \eqref{Wy_dec}, we have
\begin{equation}\label{Wy_claim}
|V(y,s_0)|\leq \frac{1}{3000}, \quad \limsup_{|y|\rightarrow \infty}|V (y,s)|=0.
\end{equation}
Applying Lemma~\ref{max_2} with \eqref{V-l-p}, \eqref{damp-p}, \eqref{KVDV}, \eqref{FV}, and \eqref{Wy_claim}, we deduce that
\begin{equation*}
\|V (\cdot,s)\|_{L^{\infty}(\mathbb{R})}\leq \frac{1}{1500}.
\end{equation*}
We remark that the condition \eqref{D-cond} in Lemma~\ref{max_2} holds for sufficiently small $\veps>0$.
\end{proof}

\begin{lemma}\label{Wy2_lem}
	Under the same assumptions in Proposition~\ref{Boot}, it holds that 
	\begin{equation}\label{Uyy-bd}
		|W_{yy} (y, s) |\leq  \frac{M^{1/8} |y|}{2(1+y^2)^{1/2}} 
	\end{equation}
	for sufficiently large $M>0$ and for all $y\in\mathbb{R}$ and $s\in[s_0,\sigma_1]$.
\end{lemma}

\begin{proof}
	Setting $\wt{V}(y,s) := (y^2+1)^{1/2} y^{-1} W_{yy}(y,s)$, we obtain from \eqref{Wy2} that, for $y\ne0$,
	\begin{equation*}
		\partial_s\wt{V} + D^{\wt{V}} \wt{V} + U^W \wt{V}_y=F^{\wt{V}},
	\end{equation*}
	where $U^W$ is given in \eqref{UW}, and
	\begin{subequations}
		\begin{align*}
				D^{\wt{V}} (y, s) &:= \frac{7}{2}+\frac{2W_y}{1-\dot{\tau}}-\frac{e^{3s/2}Z_y}{1-\dot{\tau}}+\frac{1}{y(1+y^2)} U^W, \\
				F^{\wt{V}} (y, s) &:= \frac{(y^2+1)^{1/2}}{y}\left(-\frac{8e^{s/2}Q_{yy}}{3(1-\dot{\tau})}+\frac{e^{3s}Z_yZ_{yy}}{1-\dot{\tau}}+\frac{4e^{3s/2}Z_{yy}W_y}{3(1-\dot{\tau})}\right).
			\end{align*}
		\end{subequations}
	
First, we consider the region $|y| \leq l$, where $\textstyle l := \frac{1}{M}$ for sufficiently large $M>0$. By the Taylor expansion, together with \eqref{Wy3_0} and \eqref{Wy4_bound}, we have
	\begin{equation*}
		|W_{yy}(y,s)|\leq |y||\partial_y^3W(0,s)|+\frac{y^2}{2}|\partial_y^4W(y',s)|\leq 257|y|+\frac{M}{2}y^2\leq \frac{M^{1/8} |y|}{4(1+y^2)^{1/2}}
	\end{equation*}
for $|y|\leq l $ and some $y'\in (-| y |, | y |)$. This gives 
	\begin{equation}\label{Uyy_local}
		\|\wt{V}(\cdot,s)\|_{L^{\infty}(|y|\leq l)}\leq  \frac{M^{1/8}}{4}.
	\end{equation}
Next, we obtain a lower bound for $D^{\wt{V}}$. From  \eqref{num_3}, we observe  that 
	\begin{equation*}
		\frac{7}{2}+2\overline{W}'+\frac{1}{1+y^2}\left(\frac{5}{2}+\frac{\overline{W}}{y}\right) \geq \frac{19y^2}{10(1+y^2)}.
	\end{equation*}
	In view of this estimate, we rewrite $D^{\wt{V}}$ as
	\begin{multline*}
		D^{\wt{V}}=\left(\frac{7}{2}+2\overline{W}'+\frac{1}{1+y^2}\left(\frac{5}{2}+\frac{\overline{W}}{y}\right)\right)
		\\
		+\frac{2(\wt{W}_y+\dot{\tau}\overline{W}')}{1-\dot{\tau}}-\frac{e^{3s/2}Z_y}{1-\dot{\tau}}+\frac{1}{y(1+y^2)}\left(\frac{\wt{W}+\dot{\tau}\overline{W}}{1-\dot{\tau}}+\frac{e^{3s/2}Z}{3(1-\dot{\tau})}+\frac{e^{3s/2}(\kappa-\dot{\xi})}{1-\dot{\tau}}\right). 
	\end{multline*}
Thanks to \eqref{dottau}, \eqref{dottau1}, \eqref{Wy_bound}, and the bound $|\overline{W}'|\leq 2$ from \eqref{4.0}, we obtain
\begin{equation*}
\left\lvert \frac{2(\widetilde{W}_y + \dot{\tau} \overline{W}')}{1-\dot{\tau}} \right\rvert \leq 2(1+\varepsilon^{1/2}) \left( \frac{y^2}{1000(1+y^2)} + C \varepsilon \right).
\end{equation*}
Also, using \eqref{dottau1} and \eqref{lem_Zy}, we have
\begin{equation*}
\left\lvert \frac{e^{3s/2} Z_y}{1-\dot{\tau}} \right\rvert \leq C e^{-s}.
\end{equation*}
From \eqref{Wy_bound}, \eqref{dottau1}, and the fact $\wt{W}(0,s)=0$, it follows that
\begin{equation*}
	\begin{split}
		\left\lvert \frac{1}{y(1+y^2)} \frac{\widetilde{W}}{1-\dot{\tau}} \right\rvert &\leq  (1+\varepsilon^{1/2}) \left\lvert \frac{1}{y(1+y^2)} \right\rvert \left\lvert \frac{1}{1000} \int_0^y \frac{y'^2}{1+y'^2} \, dy' \right\rvert 
		\\
		&\leq \left\lvert \frac{1+\varepsilon^{1/2}}{1000y(1+y^2)} \right\rvert \left\lvert \int_0^y \frac{y'^2}{1+y'^2} \, dy' \right\rvert \leq  \frac{(1+\varepsilon^{1/2})y^2}{3000(1+y^2)}, \quad \lvert y \rvert \geq l.
	\end{split}
\end{equation*}
In addition, from \eqref{4.00}, \eqref{dottau}, \eqref{tau_decay}, and \eqref{Z}, we obtain
\begin{equation*}
	\begin{split}
		\left\lvert \frac{1}{y(1+y^2)} \left( \frac{\dot{\tau} \overline{W}}{1-\dot{\tau}} + \frac{e^{3s/2} (\kappa - \dot{\xi})}{1-\dot{\tau}}+\frac{e^{3s/2}Z}{3(1-\dot{\tau})} \right) \right\rvert 
		& \quad \leq C\left( \frac{\lvert \dot{\tau} \overline{W} \rvert}{|y|} + \left| \frac{e^{3s/2} (\kappa - \dot{\xi})}{y} +\frac{e^{3s/2}Z}{3y}\right| \right) \\
		& \quad \leq C \varepsilon, \quad \lvert y \rvert \geq l.
	\end{split}
\end{equation*}
Hence, we have 
	\begin{equation}\label{Uyy_D}
		\begin{split}
			D^{\wt{V}}(y, s)	&\geq \frac{19y^2}{10(1+y^2)}- \frac{ 2 y^2}{1000(1+y^2)}-\frac{y^2}{3000(1+y^2)}- C\ve^{1/2} \geq \frac{y^2}{1+y^2}\quad \text{for}\quad |y|\geq l,
		\end{split}
	\end{equation}
where the last inequality holds for sufficiently small $\ve>0$. 

Using \eqref{dottau1}, \eqref{Q_high}, \eqref{lem_Zy}, \eqref{Z_high'}, and \eqref{Uy1}, we estimate $F^{\wt{V}}$ as
	\begin{equation}\label{Uyy_F}
		\begin{split}
			|F^{\wt{V}} (y, s) | 
			&\leq C\left(e^{s/2}|Q_{yy}|+e^{3s}|Z_y||Z_{yy}|+e^{3s/2}|Z_{yy}||W_y|\right)
			\\
			& \leq Ce^{-s} \quad \text{for}\quad |y|\geq l. 
		\end{split}
	\end{equation}

We now show that $\textstyle\limsup_{|y|\rightarrow \infty}|\wt{V} (y,s ) |< \frac12 M^{1/8}$ for sufficiently large $M>0$. Thanks to \eqref{Uyy_D} and \eqref{Uyy_F}, we have  
\begin{equation} \label{DtV}
	\inf_{|y|\geq 10}D^{\wt{V}}(y,s)\geq \frac{ 10^2}{1+ 10^2}=:\lambda_D
\end{equation}
and
\begin{equation} \label{FtV}
	\|F^{\wt{V}}(y,s)\|_{L^{\infty}(|y|\geq 10)}\leq Ce^{-s}. 
\end{equation}
Applying Lemma~\ref{rmk2} on $[s_0,\sigma_1]$ with $N=10$ together with \eqref{DtV}, \eqref{FtV}, and \eqref{UW_far}, we obtain
\begin{equation}\label{Uyy_far}
	\limsup_{|y|\rightarrow \infty}{|\wt{V}(y,s)|} \leq \limsup_{|y|\rightarrow \infty}{|\wt{V}(y,s_0)|e^{-\lambda_D(s-s_0)}} +  Ce^{-\lambda_D s} < \frac{M^{1/8}}{2}
\end{equation}
for sufficiently large $M>0$, where in the second inequality we used the definition of $\widetilde{V}$ with the bound \eqref{1D3'}.

Finally we apply Lemma~\ref{max_2} with \eqref{Uyy_local}--\eqref{Uyy_F} and \eqref{Uyy_far} to conclude that $\textstyle \|\wt{V} (\cdot, s) \|_{L^{\infty}}\leq \frac12 M^{1/8}$ for all $s\in[s_0,\sigma_1]$, which yields the desired estimate \eqref{Uyy-bd}.
\end{proof}

\begin{lemma} \label{lem:Wy4}
	Under the same assumptions in Proposition~\ref{Boot}, it holds that
	\begin{equation*}
	\lVert \partial_y^4 W(\cdot,s) \rVert_{L^\infty} \leq \frac{M}{2}
	\end{equation*}
	for sufficiently large $M>0$ and for all $s \in [s_0,\sigma_1]$.
\end{lemma}

\begin{proof}
	We rewrite \eqref{Wy4} in the form
	\begin{equation} \label{Wy4eq}
		\partial_s \partial_y^4 W + D^W_4 \partial_y^4 W + U^W \partial_y^5 W = F^W_4,
	\end{equation}
	where $U^W$ is defined in \eqref{UW}, and 
	\begin{equation*}
	\begin{split}
		D^W_4(y,s) & := \frac{17}{2} + \frac{4W_y}{1-\dot{\tau}} - \frac{e^{3s/2} Z_y }{3(1-\dot{\tau})}, \\
		F^W_4(y,s) & := - \frac{8 e^{s/2} \partial_y^4 Q}{3(1-\dot{\tau})}  - \frac{7 W_{yy} \partial_y^3 W}{1-\dot{\tau}} \\
		& \nonumber \quad   + \frac{3e^{3s/2}  \partial_y^3 W Z_{yy}}{1-\dot{\tau}} + \frac{11e^{3s/2} W_{yy} \partial_y^3 Z}{3(1-\dot{\tau})}  + \frac{4e^{3s/2} W_y \partial_y^4 Z}{3(1-\dot{\tau})} + \frac{3e^{3s} Z_{yy} \partial_y^3 Z}{1-\dot{\tau}} + \frac{e^{3s} Z_y \partial_y^4 Z}{1-\dot{\tau}}.
	\end{split}
	\end{equation*}
	It then follows from \eqref{Uy1}, \eqref{lem_Zy}, and \eqref{dottau1} that
	\begin{equation*}
	D^W_4 \geq \frac{17}{2} - \left( \frac{4 \lvert W_y \rvert}{1-\dot{\tau}} + \frac{e^{3s/2}\lvert Z_y \rvert}{3(1-\dot{\tau})} \right) \geq \frac{17}{2} - \left( 8 + C \varepsilon \right) (1+\varepsilon^{1/2}) \geq \frac{1}{4},
	\end{equation*}
	for sufficiently small $\veps>0$. 
	Next, using \eqref{dottau1}, \eqref{Q_high}, \eqref{Wyy_bound}, \eqref{Wy3_bound}, \eqref{Uy1}, \eqref{Z_high'}, and \eqref{lem_Zy}, we estimate $F^W_4$ as follows:
	\begin{equation*}
	\begin{split}
		| F^W_4(y,s) | & \leq Ce^{s/2} \|\partial_y^4 Q\|_{L^{\infty}}  + 7(1+\veps^{1/2})\|W_{yy}\|_{L^{\infty}}\|\partial_y^3 W\|_{L^{\infty}}\\
		& \quad +Ce^{3s/2}\left(\sum_{i=1}^3 \|\partial_y^i W\|_{L^{\infty}} \|\partial_y^{(5-i)}Z\|_{L^{\infty}}  +  e^{3s/2} \sum_{i=1}^2\|\partial_y^iZ\|_{L^{\infty}} \|\partial_y^{5-i} Z\|_{L^{\infty}}  \right)
		\\
		&\leq Ce^{-7s/2} + 7 (1+\varepsilon^{1/2})M^{7/8} +Ce^{-s} \\
		& \leq C \varepsilon +  7 (1+\varepsilon^{1/2})M^{7/8} \leq 8 M^{7/8}
	\end{split}
	\end{equation*}
	again for sufficiently small $\varepsilon>0$. Integrating \eqref{Wy4eq} along the curve $\psi_W$ defined in \eqref{curves}, and using the bounds on $D_4^W$ and $F_4^W$ derived above, together with the initial assumption $\|\partial_y^4W(\cdot,s_0)\|_{L^{\infty}}\leq 2^{13}$ from \eqref{1D4}, we obtain
	\begin{equation*}
	\begin{split}
		\lVert \partial_y^4 W(\cdot,s) \rVert_{L^\infty} & \leq \lVert \partial_y^4 W(\cdot,s_0) \rVert_{L^\infty} e^{-(s-s_0)/4} + 8M^{7/8} \int_{s_0}^s e^{-(s-s')/4} \, ds' \leq 2^{13}+32M^{7/8}.
	\end{split}
	\end{equation*}
	By choosing $M > 0$ sufficiently large so that $\textstyle 2^{13}+32M^{7/8}\leq \frac{M}{2}$, we conclude
	\begin{equation*}
		\lVert \partial_y^4 W(\cdot,s) \rVert_{L^\infty}  \leq  \frac{M}{2},
	\end{equation*}
	as desired.
\end{proof}

\begin{lemma} \label{lem:Wy3}
Under the same assumptions in Proposition~\ref{Boot}, it holds that
	\begin{equation} \label{inlemWy3}
	\lVert \partial_y^3 W(\cdot,s) \rVert_{L^\infty} \leq \frac{M^{3/4}}{2}
	\end{equation}
	for sufficiently large $M>0$ and for all $s \in [s_0,\sigma_1]$.
\end{lemma}

\begin{proof}
	Applying Lemma~\ref{LK} to $W_{yy}(\cdot,s)$ and using \eqref{Wyy_bound} and \eqref{Wy4_bound}, we have
	\begin{equation*}
		\lVert \partial_y^3 W(\cdot,s) \rVert_{L^\infty}^2
		\leq 2\lVert W_{yy}(\cdot,s) \rVert_{L^\infty}
		\lVert \partial_y^4 W(\cdot,s) \rVert_{L^\infty}
		\leq 2M^{9/8}.
	\end{equation*}
	Since $\textstyle \frac{9}{16}<\frac{3}{4}$, taking square roots yields \eqref{inlemWy3} for sufficiently large $M$.
\end{proof}

\begin{lemma}\label{dec-lem} 
	Under the same assumptions in Proposition~\ref{Boot}, it holds that  
	\begin{equation}\label{Ut_dec_p}
		\limsup_{|y|\rightarrow \infty} |(y^{2/5}+1)(W_y (y, s) -\overline{W}' (y) ) | <  \frac{6}{13} -  2  \theta,
	\end{equation}
	for all $s \in [s_0,\sigma_1]$, where $\theta$ is as in \eqref{init_wx_dec}.
\end{lemma}

\begin{proof}
	From \eqref{asymp-y-infty} with $\beta=1$, we have
	\begin{equation*}
		\begin{split}
			\limsup_{|y|\rightarrow \infty}|(y^{{2/5}}+1)(W_y-\overline{W}')|&\leq \limsup_{|y|\rightarrow \infty}|(y^{2/5}+1)W_y|+\lim_{|y|\rightarrow \infty}|(y^{2/5}+1)\overline{W}'|
			\\
			&= \limsup_{|y|\rightarrow \infty}|(y^{2/5}+1)W_y|+\frac{1}{50^{1/5}}.
		\end{split}
	\end{equation*}
	By \eqref{init_wx_dec}, we know that ${50^{-1/5}} < {6/13} - 3\theta$. Hence, to prove \eqref{Ut_dec_p}, it suffices to show that
	\begin{equation*}\label{3.00}
		\limsup_{|y|\rightarrow \infty} |(y^{2/5}+1)W_y(y,s)|< \theta. 
	\end{equation*}
	
	Let $\mu(y,s):=(y^{2/5}+1)W_y(y,s)$ and $\wt{W}(y,s):=W(y,s)-\overline{W}(y)$. Then we have from \eqref{Wy1} that, for $y\ne0$,
	\begin{equation*}
		\partial_s \mu+D^{\mu}\mu+U^W \partial_y \mu =F^{\mu},
	\end{equation*}
	where $U^W$ is defined in \eqref{UW} and
		\begin{subequations}
			\begin{align*}
				D^{\mu} (y,s) &:= 1+\frac{\wt{W}_y+\overline{W}'}{2}-\frac{2y^{2/5}}{5(1+y^{2/5})}\left(\frac{5}{2}+\frac{\wt{W}+\overline{W}}{y}\right),
				\\
				F^{\mu} (y,s) &:= (y^{2/5}+1)\left(-\frac{8e^{s/2}Q_y}{3(1-\dot{\tau})}+\frac{e^{3s}Z_y^2}{2(1-\dot{\tau})}\right) -\frac{\dot{\tau}W_y\mu}{2(1-\dot{\tau})}
				\\
				&\qquad +\left(\frac{4e^{3s/2}Z_y}{3}+\frac{2y^{2/5}}{5(1+y^{2/5})}\left(\frac{e^{3s/2}(\kappa-\dot{\xi})+\frac{1}{3}e^{3s/2}Z}{y}+\frac{\dot{\tau}W}{y}\right)\right)\frac{\mu}{1-\dot{\tau}}.
			\end{align*}
		\end{subequations}	
	Using \eqref{Wy_dec} and $\wt{W}(0,s)=0$, we have
	\begin{equation*}
		\begin{split}
			D^{\mu}(y,s)&\geq 1-\frac{3}{13(1+y^{2/5})}+\frac{\overline{W}'}{2}-\frac{2y^{2/5}}{5(1+y^{2/5})}\left(\frac{5}{2}+\frac{\overline W}{y}+\frac{6}{13y}\int_{0}^{y}{\frac{dy'}{1+{y'}^{2/5}}}\right)
			\\
			&= \frac{10}{13(1+y^{2/5})}+\frac{\overline{W}'}{2}-\frac{2y^{2/5}}{5(1+y^{2/5})}\left(\frac{\overline W}{y}+\frac{6}{13y}\int_{0}^{y}{\frac{dy'}{1+{y'}^{2/5}}}\right). 
		\end{split}
	\end{equation*}
	Let us set $\Omega:=\{y\in\mathbb{R} : |y|< 2\cdot 10^6\}$.
	Thanks to \eqref{num_1-lem} and \eqref{4.0}, we see that $D^\mu$ is non-negative for $y\in \Omega^c$, i.e.,  
	\begin{equation}\label{Dmu}
		\inf_{y\in \Omega^c}D^{\mu} (y,s) \geq 0.
	\end{equation} 
	Moreover, from \eqref{4.0'} and \eqref{Wy_dec}, we deduce the following bound:   
	\begin{equation*}
		|W_y(y,s)| \leq |\overline{W}'(y)|+|W_y(y,s)-\overline{W}'(y)| \leq \frac{C}{y^{2/5}},
	\end{equation*} 
	which implies that $\|\mu (\cdot, s) \|_{L^{\infty}}\leq C$ for some constant $C>0$. 
	Next, using \eqref{GQ_bound} and \eqref{Qy_weight}, we estimate $Q_y$ as follows:
	\begin{align*}
		y^{2/5} |Q_y| &\leq Ce^{-3s/2} \quad \text{for } |y| \leq e^{5s/2}, \\
		y^{2/5} |Q_y| &= \frac{y^{4/5} |Q_y|}{y^{2/5}} \leq Ce^{-3s/2} \quad \text{for } |y| \geq e^{5s/2}.
	\end{align*}
	Combining these, we obtain the global bound:
	\begin{equation*}
		\|(y^{2/5}+1)Q_y\|_{L^{\infty}}\leq Ce^{-3s/2}.
	\end{equation*} 
		Using the above estimates together with \eqref{dottau1}, \eqref{Uy1}, \eqref{tau_decay}, \eqref{lem_Zy}, \eqref{Zy_dec}, and \eqref{Z}, we obtain the following bound on $F^\mu$ in $ \Omega^c$:
		\begin{equation}\label{Fmu}
			\begin{split}
				&\|F^{\mu}(y,s)\|_{L^{\infty}(\Omega^c)} \\
				&\quad \leq C\left(e^{s/2}\|(y^{2/5}+1)Q_y\|_{L^{\infty}}+e^{3s}\|(y^{2/5}+1)Z_y\|_{L^{\infty}}\|Z_y\|_{L^{\infty}}\right)
				\\
				& \qquad +C\bigg(e^{3s/2}\|Z_y\|_{L^{\infty}}+\bigg\|\frac{e^{3s/2}(\kappa-\dot{\xi})+\frac{1}{3}e^{3s/2}Z}{y}\bigg\|_{L^{\infty}( \Omega^c)}+\left\|\frac{\dot{\tau}W}{y}\right\|_{L^{\infty}( \Omega^c)}\bigg)\|\mu(\cdot,s)\|_{L^{\infty}} \\
				& \qquad +C|\dot{\tau}|\,\|W_y\|_{L^{\infty}}\|\mu(\cdot,s)\|_{L^{\infty}}
				\\
				& \quad \leq 	Ce^{-s}.
			\end{split}
		\end{equation}
	Then, applying Lemma~\ref{rmk2} on $[s_0,\sigma_1]$ with $N=2\cdot10^6$ along with \eqref{UW_far}, \eqref{Dmu}, and \eqref{Fmu}, we conclude that
		\begin{equation*}
			\limsup_{|y|\rightarrow \infty}|\mu(y,s)|\leq \limsup_{|y|\rightarrow \infty}|\mu(y,s_0)|+C\veps \le \frac{\theta}{2}+C\veps < \theta
		\end{equation*}
	for sufficiently small $\ve>0$. 
	Here, we used \eqref{init_wx_dec}, that is equivalent to 
	\begin{equation*}
		\limsup_{|y|\rightarrow \infty}|\mu(y,s_0)|\leq \frac{\theta}{2}.
	\end{equation*}
\end{proof}

\begin{lemma}\label{mainprop_1-p}
	Under the same assumptions in Proposition~\ref{Boot}, it holds that  
	\begin{equation*}
		|(y^{2/5}+1) ( W_y (y, s) - \UU'(y) ) | \le \frac{6}{13}-\theta
	\end{equation*}
	for all $y\in\mathbb{R}$ and $s\in[s_0,\sigma_1]$, where $\theta$ is defined in \eqref{init_wx_dec}. 
\end{lemma}

\begin{proof}
The structure of the proof is identical to that of \cite[Lemma~4.8]{KKY}. For completeness, we present the key estimates adapted to the setting considered here.

	Let $\nu (y,s) := (y^{2/5}+1)\wt{W}_y(y,s)$, where $\wt{W} (y, s) := W(y,s) - \overline{W}(y)$. 
	From \eqref{Wy1} and \eqref{Weq}, $\nu$ satisfies 
	\begin{equation}\label{nu_eq1-p}
		\partial_s \nu + D^\nu \nu  + U^W \nu_y = F^\nu_1 +F^\nu_2 +\int_{\mathbb{R}}{\nu(y',s) K^\nu(y, s; y') \,dy'},
	\end{equation}
	where $U^W$ is defined in \eqref{UW} and
	\begin{subequations}
		\begin{align*}
			D^\nu(y,s) &:= 1+\frac{\wt{W}_y+2\overline{W}'}{2}-\frac{2y^{2/5}}{5(1+y^{2/5})}\left(\frac{5}{2}+\frac{\wt{W}+\overline{W}}{y}\right),
			\\
			F_1^\nu(y,s) &:=(y^{2/5}+1)\left(-\frac{8e^{s/2}Q_y}{3(1-\dot{\tau})}+\frac{e^{3s}Z_y^2}{2(1-\dot{\tau})}\right)
			\\
			&\quad-\left(\frac{\dot{\tau}W_y}{2}-\frac{4e^{3s/2}Z_y}{3}\right)\frac{(y^{2/5}+1)\overline{W}'}{1-\dot{\tau}}-\left(\frac{e^{3s/2}(\kappa-\dot{\xi})+\frac{1}{3}e^{3s/2}Z}{y}+\frac{\dot{\tau}W}{y}\right)\frac{y(y^{2/5}+1)\overline{W}''}{1-\dot{\tau}},
			\\
			F^\nu_2(y,s) &:= \left(-\frac{\dot{\tau}W_y}{2}+\frac{4e^{3s/2}Z_y}{3}+\frac{2y^{2/5}}{5(1+y^{2/5})}\left(\frac{e^{3s/2}(\kappa-\dot{\xi})+\frac{1}{3}e^{3s/2}Z}{y}+\frac{\dot{\tau}W}{y}\right)\right)\frac{\nu}{1-\dot{\tau}},
			\\
			K^\nu(y,s;y') &:= -(y^{2/5}+1)\overline{W}''(y)\mathbb{I}_{[0,y]}(y')\frac{1}{1+ y'^{2/5}}.
		\end{align*}
	\end{subequations}
	Here $\mathbb{I}_{[0,y]}$ denotes the oriented indicator of the interval from $0$ to $y$.
	Let us set $\Omega:=\{y\in\mathbb{R} : |y|< 2\cdot 10^6\}$. We first show that  
	\begin{equation}\label{D-K-p}
		D^\nu(y,s)\geq \int_{\mathbb{R}}|K^\nu(y,s;y')|\,dy', \quad y\in \Omega^c.
	\end{equation}
	Under the assumption \eqref{Wy_dec}, we have the bounds
	\begin{equation}\label{nu_D-p}
		\begin{split}
			D^\nu(y,s)&\geq 1-\frac{3}{13(1+y^{2/5})}+\overline{W}'-\frac{2y^{2/5}}{5(1+y^{2/5})}\left(\frac{5}{2}+\frac{\overline W}{y}+\frac{6}{13y}\int_{0}^{y}{\frac{dy'}{1+{y'}^{2/5}}}\right) 
			\\
			&=\frac{10}{13(1+y^{2/5})}+\overline{W}'-\frac{2y^{2/5}}{5(1+y^{2/5})}\left(\frac{\overline W}{y}+\frac{6}{13y}\int_{0}^{y}{\frac{dy'}{1+{y'}^{2/5}}}\right) 
			=: D^\nu_{-}(y)
		\end{split}
	\end{equation}
	and
	\begin{equation}\label{nu_K-p}
		\begin{split}
			\int_{\mathbb{R}}{|K^\nu(y, s; y')|\,dy'} &\leq (y^{2/5}+1)|\overline{W}''(y)|\int^{|y|}_{0}{\frac{dy'}{1+{y'}^{2/5}}} =: K^\nu_{+}(y).
		\end{split}
	\end{equation}
	Applying \eqref{num_1-lem} with $m_0=2\cdot 10^6$, we directly obtain \eqref{D-K-p}. Indeed, one has
	\begin{equation}
		\int_{\mathbb{R}}{|K^\nu(y, s; y')|\,dy'} \leq K^\nu_+(y) \leq D^\nu_-(y) \leq D^\nu(y,s), \quad y\in\Omega^c,
	\end{equation}

	We proceed to show that
		\begin{equation}\label{F-nu-p}
			\|F^\nu_1(\cdot,s)\|_{L^{\infty}(\Omega^c)} + \|F^\nu_2(\cdot,s)\|_{L^{\infty}(\Omega^c)} \leq Ce^{-s}.
		\end{equation}
	We first estimate $F^\nu_1$. Using \eqref{GQ_bound}, \eqref{Qy_weight}, \eqref{lem_Zy}, and \eqref{Zy_dec}, we obtain
	\begin{equation*}
		e^{s/2}\|(y^{2/5}+1)Q_y\|_{L^{\infty}}+e^{3s}\|(y^{2/5}+1)Z_y\|_{L^{\infty}}\|Z_y\|_{L^{\infty}} \leq Ce^{-s}.
	\end{equation*}
	Also, from \eqref{tau_decay}, \eqref{Uy1}, \eqref{lem_Zy}, and the boundedness of $\|(y^{2/5}+1)\overline{W}'\|_{L^{\infty}}$ from \eqref{4.0} and \eqref{4.0'}, it follows that
	\begin{equation*}
		\left(|\dot{\tau}|\|W_y\|_{L^{\infty}}+e^{3s/2}\|Z_y\|_{L^{\infty}}\right)\|(y^{2/5}+1)\overline{W}'\|_{L^{\infty}}\leq Ce^{-s}.
	\end{equation*}
	Thanks to \eqref{Z}, \eqref{tau_decay}, \eqref{Uy1}, and \eqref{4.0''}, we get
	\begin{equation*}
		\bigg( \bigg\|\frac{e^{3s/2}(\kappa-\dot{\xi})+\frac{1}{3}e^{3s/2}Z}{y} \bigg\|_{L^{\infty}(\Omega^c)}+|\dot{\tau}|\bigg\|\frac{W}{y}\bigg\|_{L^{\infty}(\Omega^c)} \bigg) \big\|y(y^{2/5}+1)\overline{W}'' \big\|_{L^{\infty}} \leq Ce^{-s}.
	\end{equation*}
	Combining these estimates and using \eqref{dottau1}, we obtain
	\begin{equation*}
		\begin{split}
			\|F^\nu_1\|_{L^{\infty}(\Omega^c)} \leq Ce^{-s}.
		\end{split}
	\end{equation*}
	Next, we use \eqref{dottau1} and \eqref{Wy_dec} to estimate $F^\nu_2$ as
	\begin{equation*}
		\begin{split}
			\|F^\nu_2\|_{L^{\infty}(\Omega^c)} &\leq C|\dot{\tau}|\|W_y\|_{L^{\infty}}+Ce^{3s/2}\|Z_y\|_{L^{\infty}} + C \bigg\|\frac{e^{3s/2}(\kappa-\dot{\xi})+\frac{1}{3}e^{3s/2}Z}{y}\bigg\|_{L^{\infty}(\Omega^c)} + |\dot{\tau}|\bigg\|\frac{W}{y}\bigg\|_{L^{\infty}(\Omega^c)}.
		\end{split}
	\end{equation*}
	Then, by applying \eqref{tau_decay}, \eqref{Uy1}, \eqref{lem_Zy}, and \eqref{Z}, we obtain
	\begin{equation*}
		\|F^\nu_2\|_{L^{\infty}(\Omega^c)}\leq Ce^{-s}.
	\end{equation*}
	The desired estimate \eqref{F-nu-p} follows from the bounds on $F_1^{\nu}$ and $F_2^{\nu}$.

At this point, the necessary estimates have been established: \eqref{D-K-p}--\eqref{F-nu-p} and the far-field estimate \eqref{Ut_dec_p} from Lemma~\ref{dec-lem}. The remainder of the proof, which involves a contradiction argument using the evolution equation \eqref{nu_eq1-p} and these bounds, proceeds as in the latter part of the proof of \cite[Lemma~4.8]{KKY}. We omit the details.
\end{proof}

\section{Outline of the proof of Theorem~\ref{mainthm2}}\label{sec4}

In this section, we outline the proof of Theorem~\ref{mainthm2}. The overall strategy closely follows that for the rSV system. We first recall a local existence result and then introduce a bootstrap framework in self-similar variables, analogous to that of Section~\ref{sec:boot}. The bootstrap assumptions are then closed by adapting the analysis of Section~\ref{sec:3.3} to the present scalar setting. Indeed, compared with the rSV case, the transport structure and damping terms are simpler, and the only additional issue is the treatment of the nonlocal terms associated with $p$. These terms, however, are easier to control than the $Q$-terms in \eqref{Wy1}--\eqref{Wy4}, while the basic ingredients needed for the argument, such as conservation of energy and uniform bounds on the solution, remain available. Once the bootstrap is closed, the proof of Theorem~\ref{mainthm2} proceeds as in the proof of Theorem~\ref{mainthm} in Section~\ref{C13_subsec-p}. We now present the main steps and key estimates needed to close the bootstrap.

For a class of equations that includes the regularized Burgers equation as a special case, local well-posedness was established in \cite{Yi, Yi1}:

\begin{lemma}[\cite{Yi, Yi1}]
For the initial data \eqref{rBic}, there exists a maximal existence time $T_* > -\varepsilon$ such that the initial value problem for \eqref{rB} admits a unique solution $v \in C([-\varepsilon,T_*);H^5(\mathbb{R})) \cap C^1([-\varepsilon,T_*);H^4(\mathbb{R}))$ satisfying
\begin{equation}
T_* < +\infty \quad \Longrightarrow \quad \lim_{t \nearrow T_*}{\inf_{x \in \mathbb{R}}{v_x(x,t)}} = -\infty.
\end{equation}
\end{lemma}

We next provide some preliminary estimates that play a role analogous to certain estimates used in the analysis of the rSV system. Define an energy $E^v(t)$ associated with the equation \eqref{rB} by
\begin{equation*}
E^v(t) := \int_\mathbb{R} (v^2 + v_x^2) \, dx.
\end{equation*}
It is straightforward to check that the energy is conserved:
\begin{equation*}
E^v(t) = E^v(-\varepsilon), \quad t \geq -\varepsilon.
\end{equation*}
Then it follows from the Sobolev inequality that
\begin{equation*}
|v(x,t)| \leq C \lVert v(\cdot,t) \rVert_{H^1} < \infty.
\end{equation*}
Using the Green function $\textstyle K_0(x) = \frac{1}{2}e^{-|x|}$ for the operator $(1-\partial_x^2)^{-1}$, we obtain
\begin{equation} \label{p_bound}
| p(x,t) | = \left| \left( K_0 * \frac{v_x^2}{2} \right) (x,t) \right|  \leq \frac{1}{4} \int_\mathbb{R} v_x^2(z,t) \, dz \leq \frac{1}{4} E^v(-\varepsilon) < \infty
\end{equation}
and, differentiating once,
\begin{equation} \label{px_bound}
|p_x(x,t)| = \left| \frac{1}{4} \int_\mathbb{R} e^{-|x-z|} \frac{x-z}{|x-z|} v_x^2(z,t) \, dz \right| \leq \frac{1}{4} E^v(-\varepsilon) < \infty.
\end{equation}

Similarly to \eqref{modul_xt}, we define dynamic modulation functions $\tau,\kappa,\xi:[-\varepsilon,T)\to\mathbb{R}$ for some $T>-\varepsilon$, satisfying
\begin{equation} \label{rBmodul}
\begin{split}
\dot{\tau}(t) & = - \frac{(\tau(t)-t)^2}{2} p(\xi(t),t),\\
\dot{\kappa}(t) & = - \frac{2 (\tau(t)-t)^{-1} p_{x}(\xi(t),t)}{ \partial_{x}^3 v(\xi(t),t)} - p_{x}(\xi(t),t),\\
\dot{\xi}(t) & = \frac{p_{x}(\xi(t),t)}{ \partial_{x}^3 v(\xi(t),t)} + \kappa(t)
\end{split}
\end{equation}
with the initial data
\begin{equation*}
\tau(-\varepsilon) =0, \quad \kappa(-\varepsilon) = \kappa_0, \quad \xi(-\varepsilon) = 0,
\end{equation*}
where $\kappa_0 := v_0(0)$. By standard ODE theory, the initial value problem for \eqref{rBmodul} with the above initial data admits a unique local $C^1$ solution $(\tau,\kappa,\xi)$. Define the self-similar variables
\begin{equation*}
y(x,t) = \frac{x-\xi(t)}{(\tau(t) - t)^{5/2}}, \quad s(t) = -\log{(\tau(t) - t)},
\end{equation*}
and new functions
\begin{equation*}
v(x,t) - \kappa(t) = e^{-3s/2}V(y,s), \quad p(x,t) = P(y,s).
\end{equation*}
Then, we obtain the equations for $V$ and $P$ as
\begin{subequations}
\begin{align}
& \label{Veq} \partial_s V - \frac{3}{2}V + U^V V_y = - \frac{e^{s/2}\dot{\kappa}}{1-\dot{\tau}} - \frac{e^{3s}P_y}{1-\dot{\tau}}, \\
& \label{Peq} \left(  e^{-5s} - \partial_y^2 \right) P = \frac{e^{-3s}V_y^2}{2},
\end{align}
\end{subequations}
where
\begin{equation*}
U^V(y,s) := \frac{5}{2}y + \frac{V}{1-\dot{\tau}} + \frac{e^{3s/2}(\kappa - \dot{\xi})}{1-\dot{\tau}}.
\end{equation*}
Applying $\partial_y^n$, $n= 1,2,3,4$, we have
\begin{equation} \label{srBderiv}
\begin{split}
& \left( \partial_s + 1 + \frac{V_y}{2(1-\dot{\tau})} \right) V_y + U^V V_{yy} = - \frac{e^{-2s}P}{1-\dot{\tau}}, \\
& \left( \partial_s + \frac{7}{2} + \frac{2V_y}{1-\dot{\tau}} \right) V_{yy} + U^V \partial_y^3 V = - \frac{e^{-2s}P_y}{1-\dot{\tau}}, \\
& \left( \partial_s + 6 + \frac{3V_y}{1-\dot{\tau}} \right) \partial_y^3 V + U^V \partial_y^4 V = - \frac{e^{-2s}P_{yy}}{1-\dot{\tau}} - \frac{2 V_{yy}^2}{1-\dot{\tau}}, \\
& \left( \partial_s + \frac{17}{2} + \frac{4V_y}{1-\dot{\tau}} \right) \partial_y^4 V + U^V \partial_y^5 V = - \frac{e^{-2s}\partial_y^3 P}{1-\dot{\tau}} - \frac{7 V_{yy} \partial_y^3 V}{1-\dot{\tau}}.
\end{split}
\end{equation}
The modulation equations \eqref{rBmodul} are chosen so that the normalization
\begin{equation*}
V(0,s)=0, \quad V_y(0,s)=-2, \quad V_{yy}(0,s)=0
\end{equation*}
is preserved.

With this setup, we proceed using a bootstrap argument similar to that employed in the rSV case. Specifically, we impose the bootstrap assumptions:
\begin{subequations}
\begin{align}
& \label{rBass1} |V_y(y,s) - \overline{W}'(y)| \leq \frac{y^2}{1000(1+y^2)}, \\
& \label{rBass2} |V_y(y,s) - \overline{W}'(y)| \leq \frac{6}{13(1+y^{2/5})}, \\
& \label{rBass3} |V_{yy}(y,s)| \leq \frac{M^{1/8}|y|}{(1+y^2)^{1/2}}, \\
& \label{rBass4} |\partial_y^3 V(0,s) - 256| \leq 1, \\
& \label{rBass5} \lVert \partial_y^3 V(\cdot,s) \rVert_{L^\infty} \leq M^{3/4}, \\
& \label{rBass6} \lVert \partial_y^4 V(\cdot,s) \rVert_{L^\infty} \leq M
\end{align}
\end{subequations}
on some self-similar time interval $[s_0,\sigma_1]$ and for all $y\in\mathbb{R}$, where $M>0$ is a sufficiently large constant and $s_0:=s(-\varepsilon)$. These are exactly the assumptions \eqref{Wy_bound}--\eqref{Wy4_bound} for the rSV system, with $W$ replaced by $V$. 

As in Remark~\ref{init_rmk}, one checks from \eqref{rBIC} that the corresponding initial profile $V(\cdot,s_0)$ satisfies \eqref{rBass1}--\eqref{rBass6} at $s=s_0$. The argument in Remark~\ref{bootstrap_start} applies here, with the weighted estimate for $Q_y$ replaced by the corresponding estimate for $P$ when verifying \eqref{rBass2}. Hence, \eqref{rBass1}--\eqref{rBass6} hold on $[s_0,\sigma_1]$ for some $\sigma_1>s_0$.

The quantities corresponding to \eqref{tau_decay}, \eqref{dottau1}, and \eqref{kxiZ} are directly controlled in the scalar setting. To this end, we first estimate $P$ and $P_y$. Writing \eqref{Peq} in convolution form, we obtain
\begin{equation} \label{P_Conv}
P(y,s)=\frac{1}{4} e^{-s/2}\int_{\mathbb R} e^{-e^{-5s/2}|y-y'|}V_y(y',s)^2\,dy'.
\end{equation}
By \eqref{rBass2} and \eqref{4.0'}, we have $|V_y(y,s)|\leq C|y|^{-2/5}$ for $y\ne0$. Hence, changing variables in the integral, we obtain
\begin{equation} \label{P_bd}
0\leq P(y,s)\leq C\int_{\mathbb R} e^{-|e^{-5s/2}y-y'|}|y'|^{-4/5}\,dy'\leq C.
\end{equation}
Differentiating the convolution formula gives
\begin{equation} \label{Py_bd}
|P_y(y,s)|\leq e^{-5s/2}P(y,s)\leq Ce^{-5s/2},
\end{equation}
where $C>0$ is independent of the initial data and $\varepsilon$.

Since \eqref{rBass4} implies $|\partial_y^3V(0,s)|\ge 255$, the modulation equations \eqref{rBmodul} and the bounds \eqref{P_bd}--\eqref{Py_bd} yield
\begin{equation*}
|\dot{\tau}(t)|\le Ce^{-2s}, \quad \left|\frac{1}{1-\dot{\tau}(t)}-1\right|\le Ce^{-2s}, \quad |e^{3s/2}(\kappa-\dot{\xi})|\le Ce^{-2s}|P_y(0,s)|\le Ce^{-9s/2}.
\end{equation*}
Moreover, by \eqref{P_bd}--\eqref{Py_bd} and the identities
\begin{equation*}
P_{yy}=e^{-5s}P-\frac12 e^{-3s}V_y^2, \quad \partial_y^3P=e^{-5s}P_y-e^{-3s}V_yV_{yy},
\end{equation*}
which follow from \eqref{Peq}, we obtain
\begin{equation*}
|P_{yy}(y,s)|+|\partial_y^3P(y,s)|\le Ce^{-3s}.
\end{equation*}
We also have the weighted bound
\begin{equation} \label{weight_p}
e^{-2s}\lVert (1+y^{2/5})P(\cdot,s)\rVert_{L^\infty}\le Ce^{-s}.
\end{equation}
Using the convolution formula \eqref{P_Conv}, \eqref{rBass2}, and \eqref{4.0'}, together with the decomposition of the integral into the regions $\textstyle |y'| \geq \frac{|y|}{2}$ and $\textstyle |y'|<\frac{|y|}{2}$, one finds that
\begin{equation*}
|P(y,s)|\le Ce^{2s}|y|^{-4/5}\qquad \text{for } |y|\ge e^{5s/2}.
\end{equation*}
Combined with the uniform bound $|P(y,s)|\le C$, this yields the weighted estimate \eqref{weight_p}.

These bounds replace the estimates on $\dot{\tau}$, $Z$, and $Q$ used in Section~\ref{sec:3.3}. With them, the analogues of Lemmas~\ref{lem:Wy30}--\ref{mainprop_1-p} follow by similar arguments as in Section~\ref{sec:3.3}. Consequently, one obtains the strict improvements
\begin{align*}
&|V_y(y,s)-\overline{W}'(y)|  \le \frac{y^2}{1500(1+y^2)}, &  &|V_y(y,s)-\overline{W}'(y)| \le \frac{\frac{6}{13}-\theta}{1+y^{2/5}}, \\
&|V_{yy}(y,s)| \le \frac{M^{1/8}|y|}{2(1+y^2)^{1/2}}, & &|\partial_y^3V(0,s)-256| \le C\varepsilon, \\
&\|\partial_y^3V(\cdot,s)\|_{L^\infty} \le \frac{M^{3/4}}{2}, & &\|\partial_y^4V(\cdot,s)\|_{L^\infty} \le \frac{M}{2}
\end{align*}
for all $s\in[s_0,\sigma_1]$ and $y\in\mathbb{R}$. This closes the bootstrap assumptions \eqref{rBass1}--\eqref{rBass6}. Then the continuation argument and the derivation of the assertions corresponding to \eqref{no-blow}--\eqref{blow-up-rate1} proceed exactly as in Section~\ref{C13_subsec-p}, completing the proof of Theorem~\ref{mainthm2}.

\appendix

\section{Admissible initial data}\label{app:admissible-data}

Let \(\varepsilon>0\) be sufficiently small. We construct initial data \((w_0,z_0)\) satisfying \eqref{init_h}--\eqref{init_zx_weight} and \eqref{E0_bound}, and show that the set of such data has non-empty relative interior in the class satisfying \eqref{IC} and \eqref{init_w0}, with respect to the topology induced by the \(H^5\) norm together with the weighted quantities appearing in \eqref{init_wx_weight}--\eqref{init_zx_weight}.

\subsection{Bounds for the reference profile \texorpdfstring{\(\overline{W}\)}{Wbar}}

We first record derivative bounds for \(\overline W\). These bounds justify the choices $C_1=2^{13}$, \(C_2=2^2\), \(C_3=2^6\), and \(C_4=2^{11}\) in \eqref{init_wbound}, as well as the corresponding constants in \eqref{rBIC}.

\begin{lemma} \label{lem:profile-norms}
For the profile \(\overline W=\overline W_1\) defined in Proposition~\ref{Profile-construct}, we have
\begin{equation*}
\|\overline W'\|_{L^\infty}=2,\quad \|\overline W''\|_{L^\infty}<2^4,\quad \|\overline W^{(3)}\|_{L^\infty}=2^8,\quad \|\overline W^{(4)}\|_{L^\infty}<2^{13},
\end{equation*}
and
\begin{equation*}
\|\overline W''\|_{L^2}<2^2,\quad \|\overline W^{(3)}\|_{L^2}<2^6,\quad \|\overline W^{(4)}\|_{L^2}<2^{11}.
\end{equation*}
\end{lemma}

\begin{proof}
For \(y>0\), set
\begin{equation*}
q(y):=-\overline W'(y),\quad A(y):=\overline W(y)+\frac52y.
\end{equation*}
Then, by Proposition~\ref{Profile-construct}, \(q(0)=2\), \(q(y)\to0\) as \(y\to\infty\), and the profile equation \eqref{Weq} can be rewritten as
\begin{equation}\label{qequ}
Aq'=-\frac12q(2-q),\quad A'=\frac52-q.
\end{equation}
From this, we find that
\begin{equation*}
\frac{d}{dy}\left(\frac{A^2q^5}{2-q}\right)=0.
\end{equation*}
Integrating it from $0$ to $y$, and using \eqref{W_taylor_small} with \(\beta=1\), we have
\begin{equation*}
A=\frac14(2-q)^{1/2}q^{-5/2}. 
\end{equation*}
Substitution into \eqref{qequ} and differentiation yield
\begin{equation*}
\overline W''=2(2-q)^{1/2}q^{7/2},\quad \overline W^{(3)}=4q^6(4q-7),\quad \overline W^{(4)}=-112q^{17/2}(2-q)^{1/2}(2q-3).
\end{equation*}
Since \(q\) decreases from \(2\) to \(0\), maximizing over \(q\in[0,2]\) gives the stated \(L^\infty\) bounds. Moreover,
\begin{equation*}
dy=-\frac{dq}{2(2-q)^{1/2}q^{7/2}}.
\end{equation*}
Using the parity of \(\overline W\) to reduce the \(L^2(\mathbb R)\) integrals to twice the integrals over \(y>0\), and then applying the change of variables together with standard beta integrals, we obtain
\begin{equation*}
\|\overline W''\|_{L^2}^2=\frac{7\pi}{2} < 2^4,\quad \|\overline W^{(3)}\|_{L^2}^2=\frac{7735\pi}{8} < 2^{12},\quad \|\overline W^{(4)}\|_{L^2}^2=963585\pi < 2^{22}.
\end{equation*}
This completes the proof.
\end{proof}

\subsection{Compatibility of the initial conditions}\label{app:data-compatibility}

To construct an admissible initial datum, it suffices to construct an even function $g\in C_c^\infty(\mathbb R)$ such that
\begin{equation}\label{g-admissible}
g(y)=\overline W'(y) \ \text{near } y=0,\quad \int_{\mathbb R}g(y)\,dy=0,\quad |g(y)-\overline W'(y)|\leq \rho b_\Theta(y),
\end{equation}
for some $\rho<1$, where
\begin{equation*}
b_\Theta(y):=\min\left\{\frac{y^2}{3000(1+y^2)},\frac{\Theta}{1+|y|^{2/5}}\right\}.
\end{equation*}
We also require that $\lVert \partial_y(g-\overline{W}')\rVert_{H^2}$ and $\lVert g-\overline{W}'\rVert_{W^{3,\infty}}$ be sufficiently small.
Once such a function $g$ has been constructed, we set
\begin{equation}\label{Phi-def}
\Phi(y):=\int_{-\infty}^y g(y')\,dy',
\end{equation}
and define
\begin{equation}\label{admissible-data-def}
w_{0,\varepsilon}(x):=2\sqrt{h_*}+\varepsilon^{3/2}\Phi\left(\frac{x}{\varepsilon^{5/2}}\right),\quad z_{0,\varepsilon}(x):=-2\sqrt{h_*}.
\end{equation}
Since $g$ is compactly supported and has zero integral, \eqref{Phi-def} gives $\Phi\in C_c^\infty(\mathbb R)$, and hence \eqref{admissible-data-def} gives \eqref{IC}. Moreover,
\begin{equation}\label{rescaled-slope-g}
\varepsilon\partial_x w_{0,\varepsilon}(\varepsilon^{5/2}y)=g(y).
\end{equation}
Thus the first condition in \eqref{g-admissible}, together with \eqref{WW-0}, gives \eqref{init_w0}, while the third condition in \eqref{g-admissible} and the fact that $\rho<1$ give \eqref{init_wx_weight}. We now construct such a function $g$.

Choose $\rho_0$ and $\rho$ such that $\textstyle \frac{50^{-1/5}}{\Theta}<\rho_0<\rho<1$. By \eqref{asymp-y-infty}, there exists $R\gg1$ such that
\begin{equation*}
|\overline W'(y)|\leq \rho_0 b_\Theta(y)\quad\text{for } |y|\geq R.
\end{equation*}
Let $\chi_R\in C_c^\infty(\mathbb R)$ be an even cutoff function such that $0\leq\chi_R\leq1$ and
\begin{equation*}
\chi_R = \begin{cases}
1 & \text{for } |y|\leq R, \\
0 & \text{for } |y|\geq 2R.
\end{cases}
\end{equation*}
Set $g_1:=\chi_R\overline W'$. Then $g_1=\overline W'$ near the origin, $g_1\in C_c^\infty(\mathbb R)$ is even, and
\begin{equation*}
|g_1(y)-\overline W'(y)|\leq \rho_0 b_\Theta(y).
\end{equation*}
Indeed, the left-hand side vanishes for $|y|\leq R$, while for $|y|\geq R$ it is bounded by $|\overline W'(y)|$. Moreover, since $\overline W'\leq0$ on $\mathbb R$ and $\overline W'<0$ near the origin, we have
\begin{equation*}
\int_{\mathbb R}g_1(y)\,dy
=\int_{\mathbb R}\chi_R(y)\overline W'(y)\,dy<0.
\end{equation*}

We now add a nonnegative correction to $g_1$ in the far field so that the resulting function satisfies the zero-integral condition in \eqref{g-admissible}. Since $b_\Theta(y)\sim\Theta |y|^{-2/5}$ as $|y|\to\infty$, we have
\begin{equation*}
\int_R^L b_\Theta(y)\,dy\to\infty \quad\text{as }L\to\infty.
\end{equation*}
By taking $R$ large enough and then choosing $L>2R$ sufficiently large, we may find an even nonnegative function $g_2\in C_c^\infty(\mathbb R)$ supported in $\{2R\leq |y|\leq L\}$ such that
\begin{equation*}
0\leq g_2(y)\leq(\rho-\rho_0)b_\Theta(y),\quad \int_{\mathbb R}g_2(y)\,dy=-\int_{\mathbb R}g_1(y)\,dy,
\end{equation*}
and such that $\lVert \partial_y(g_1+g_2-\overline{W}') \rVert_{H^2}$ and $\lVert g_1+g_2-\overline{W}' \rVert_{W^{3,\infty}}$ are arbitrarily small. Setting $g:=g_1+g_2$, we obtain an even function $g\in C_c^\infty(\mathbb R)$ satisfying \eqref{g-admissible}. Moreover, by Lemma~\ref{lem:profile-norms}, $g$ may be chosen so that the bounds corresponding to \eqref{init_wbound}, except for the bound on $\partial_x w_0$, are satisfied strictly.

It remains to verify the remaining admissibility conditions for the data defined by \eqref{Phi-def} and \eqref{admissible-data-def}. By \eqref{4.1} and the definition of $b_\Theta$,
\begin{equation*}
\overline W'(y)-b_\Theta(y)\geq -2,\quad \overline W'(y)+b_\Theta(y)<2.
\end{equation*}
Since $\rho<1$, the third condition in \eqref{g-admissible} implies
\begin{equation*}
\overline W'(y)-b_\Theta(y)\leq g(y)\leq \overline W'(y)+b_\Theta(y).
\end{equation*}
Hence
\begin{equation*}
-2\leq g(y)<2.
\end{equation*}
Since $g=\overline W'$ near the origin and $\overline W'(0)=-2$, we also have
\begin{equation*}
\|g\|_{L^\infty}=2.
\end{equation*}
For the data defined in \eqref{admissible-data-def}, we have, for $j\geq1$,
\begin{equation*}
\partial_x^jw_{0,\varepsilon}(x)
=\varepsilon^{(3-5j)/2}g^{(j-1)}\left(\frac{x}{\varepsilon^{5/2}}\right),
\end{equation*}
and therefore
\begin{equation*}
\|\partial_x^jw_{0,\varepsilon}\|_{L^2} =\varepsilon^{-(10j-11)/4}\|g^{(j-1)}\|_{L^2}.
\end{equation*}
Together with the bounds on $g$ and its derivatives obtained above, these identities give the conditions on $w_{0,\varepsilon}$ in \eqref{init_wbound}. Since $g$ is compactly supported, \eqref{rescaled-slope-g} also gives \eqref{init_wx_dec}, with limit zero. The choice $z_{0,\varepsilon}\equiv -2\sqrt{h_*}$ gives the conditions on $z_0$ in \eqref{init_wbound} and \eqref{init_zx_weight}.

Let $\delta_\varepsilon:=w_{0,\varepsilon}-2\sqrt{h_*}$. Since $\Phi\in C_c^\infty(\mathbb R)$, \eqref{admissible-data-def} gives
\begin{equation*}
\|\delta_\varepsilon\|_{L^\infty}=O(\varepsilon^{3/2}),\quad
\|\delta_\varepsilon\|_{L^2}^2=O(\varepsilon^{11/2}),\quad
\|\partial_x\delta_\varepsilon\|_{L^2}^2=O(\varepsilon^{1/2}).
\end{equation*}
Using \eqref{wz}, we obtain
\begin{equation*}
u_{0,\varepsilon}=\frac12\delta_\varepsilon,\quad
h_{0,\varepsilon}=\left(\sqrt{h_*}+\frac14\delta_\varepsilon\right)^2.
\end{equation*}
Thus \eqref{init_h} holds for all sufficiently small $\varepsilon>0$. Substitution into \eqref{def_E} gives
\begin{equation*}
E_0=O(\varepsilon^{1/2}),
\end{equation*}
and hence \eqref{E0_bound}. This proves that the initial conditions can be satisfied simultaneously.

Finally we verify the assertion about relative interior. Since $\rho<1$ in \eqref{g-admissible}, the weighted condition \eqref{init_wx_weight} is strict away from the origin. The higher-order derivative bounds above were also chosen to be strict. The only non-strict bound is
\begin{equation*}
\|\partial_xw_{0,\varepsilon}\|_{L^\infty}\leq 2\varepsilon^{-1},
\end{equation*}
and this follows from \eqref{init_wx_weight}, because
\begin{equation*}
-2\leq \overline W'(y)-b_\Theta(y)
\leq \varepsilon\partial_xw_{0,\varepsilon}(\varepsilon^{5/2}y)
\leq \overline W'(y)+b_\Theta(y)<2.
\end{equation*}
The lower bound is attained only at \(y=0\), where \eqref{init_w0} fixes the value of \(\partial_x w_{0,\varepsilon}\). Hence sufficiently small perturbations in the relative topology specified above remain admissible, provided that \eqref{init_w0} is preserved. The admissible set therefore has non-empty relative interior.

\begin{remark}\label{rem_rel}
The value \(256\) in \eqref{init_w0} and \eqref{rBIC} comes from the normalization \(\overline W=\overline W_1\). Since
\begin{equation*}
\overline W_\beta(y)=\beta^{-1/2}\overline W(\beta^{1/2}y),\quad \overline W_\beta^{(3)}(0)=256\beta,
\end{equation*}
replacing \(\overline W\) by \(\overline W_\beta\) changes the third matching condition to \(\partial_x^3w_0(0)=256\beta\varepsilon^{-6}\), and similarly for \(v_0\). For \(\beta\) in a sufficiently small open interval containing \(1\), all the strict bounds for the derivatives remain valid, \((50\beta)^{-1/5}<\Theta\), and the scaled form of \eqref{4.1} gives \(\overline W_\beta'(y)-b_\Theta(y)\geq-2\). Thus the same construction applies after \(\overline W\) is replaced by \(\overline W_\beta\) in \eqref{init_wx_weight} and \eqref{rBIC}. The bootstrap argument must then use \(\overline W_\beta\) in place of \(\overline W\) as well. Consequently, \(\varepsilon^6\partial_x^3w_0(0)\), and likewise \(\varepsilon^6\partial_x^3v_0(0)\), may vary in an open interval containing \(256\).
\end{remark}

\section{\texorpdfstring{Inequalities related to $\overline{W}$ and its derivatives}{Inequalities related to Wbar and its derivatives}}

We present some inequalities for $\overline{W}(y):=\overline{W}_\beta(y)$ with $\beta=1$, where $\overline{W}$ is the solution of the ODE problem \eqref{Weq}--\eqref{decay-infty}, that are used in the proofs in Section~\ref{sec2}--\ref{sec3}.

\begin{lemma}[\cite{KKY}, Lemma~2.2] \label{Wbar_ineq_lem}
	It holds that
	\begin{subequations}
		\begin{align}
			&2 + \overline{W}'(y) - \frac{6y^2}{5(1+y^2)} \geq 0, \quad y\in\mathbb{R}, \label{4.1}
			\\
			&1 + \overline{W}' + \frac{2}{1+y^2} \left( \frac{5}{2} + \frac{\overline{W}}{y} \right) \geq \frac{y^2}{5(1+y^2)}, \quad y\in\mathbb{R}, \label{num_2} 
			\\
			&\frac{7}{2} + 2\overline{W}' + \frac{1}{1+y^2} \left( \frac{5}{2} + \frac{\overline{W}}{y} \right) \geq \frac{19y^2}{10(1+y^2)}, \quad y\in\mathbb{R}.\label{num_3}
		\end{align}
	\end{subequations}
Furthermore, for some $\delta \in (0,1)$, it holds that
\begin{equation}\label{num_6}
	\lvert \overline{W}''(y) \rvert \frac{y^2 + 1}{y^2} \int_0^{\lvert y \rvert} \frac{y'^2}{y'^2 + 1} \, dy' \leq \delta \left( 1 + \overline{W}'(y) + \frac{2}{y^2 + 1} \left( \frac{5}{2} + \frac{\overline{W}(y)}{y} \right) - \frac{y^2}{500(1+y^2)} \right), \quad y\in \mathbb{R},
\end{equation}
and for some $m_0>0$, that
\begin{equation} \label{num_1-lem}
	\begin{split}
		& (y^{2/5} + 1) \lvert \overline{W}''(y) \rvert \int_0^{\lvert y \rvert} \frac{1}{1+y'^{2/5}} \, dy' \\
		& \quad \leq \frac{10}{13(1+y^{2/5})} + \overline{W}' - \frac{2y^{2/5}}{5(1+y^{2/5})} \left( \frac{\overline{W}}{y} + \frac{6}{13y} \int_0^y \frac{1}{1+y'^{2/5}} dy' \right), \quad \lvert y \rvert \geq m_0.
	\end{split}
\end{equation}
One can verify that \eqref{num_1-lem} holds for $m_0=2 \cdot 10^6$.
\end{lemma}

\begin{proof}
These estimates are established in \cite{KKY}. For their proofs, we refer to Lemmas~6.1--6.4 therein.
\end{proof}

\begin{remark}
	The inequality \eqref{4.1}, together with the monotonicity of $\overline{W}$ from Proposition~\ref{Profile-construct}, implies
	\begin{equation}\label{4.0}			
		-2\leq \overline{W}'(y) \leq 0, \quad y\in\mathbb{R}
	\end{equation}
	and, using also that $\overline{W}(0)=0$, we obtain
	\begin{equation} \label{4.00}
	|\overline{W}(y)| \leq 2 |y|, \quad y\in\mathbb{R}.
	\end{equation}
\end{remark}

	\begin{lemma}
		There exists a constant $C>0$ such that
		\begin{subequations}
			\begin{align}
				&y^{2/5}|\overline{W}'(y)| \leq C,  \quad y\in\mathbb{R}, \label{4.0'}
				\\
				&|y|(y^{2/5}+1)|\overline{W}''(y)|\leq C,  \quad y\in\mathbb{R}.\label{4.0''}
			\end{align}
		\end{subequations}
	\end{lemma}
	\begin{proof}
		It is known from the proof of Proposition~2.1 in \cite{KKY} that $\overline{W}$ satisfies the second-order ODE
		\begin{equation}\label{U''_TEMP}
			\overline{W}'' = 2 (2 + \overline{W}')^{1/2}(-\overline{W}')^{7/2}, \quad y\geq 0.
		\end{equation}
		By the oddness of $\overline{W}$, it suffices to consider $y\geq 0$. Since \eqref{4.0} ensures that \( -2\leq \overline{W}' \leq 0 \), the right-hand side of \eqref{U''_TEMP} is well-defined. Applying the method of separation of variables to \eqref{U''_TEMP} yields the estimate
		\begin{equation*}
			y^{2/5}|\overline{W}'(y)| \leq \left( \frac{(2 + \overline{W}')^{1/2}(2\overline{W}'^2 - 2\overline{W}' + 3)}{30} \right)^{2/5}.
		\end{equation*}
		Using \eqref{4.0}, we see that the right-hand side is uniformly bounded, approximately by $(3/10)^{2/5}$. This proves \eqref{4.0'}.
		
		To prove \eqref{4.0''}, we return to \eqref{U''_TEMP} and use both \eqref{4.0} and \eqref{4.0'} to derive
		\begin{equation*}
			|y|^{7/5}|\overline{W}''| 
			= 2(2 + \overline{W}')^{1/2}(-y^{2/5}\overline{W}')^{7/2} \leq C,
		\end{equation*}
		and, similarly,
		\begin{equation*}
			|y||\overline{W}''| 
			= 2(2 + \overline{W}')^{1/2}(-\overline{W}')(-y^{2/5}\overline{W}')^{5/2} \leq C.
		\end{equation*}
		Adding these two estimates, we obtain the desired bound:
		\begin{equation*}
			|y|(y^{2/5} + 1)|\overline{W}''| \leq C.
		\end{equation*}
	\end{proof}

\section{Maximum principles} \label{Maximum}

We present a maximum principle, a modified form of that developed in \cite{BSV}, for use in our analysis. We consider the initial value problem:
\begin{equation} \label{ivp}
\begin{split}
& \partial_s f(y,s) + D(y,s) f(y,s) + U(y,s) \partial_y f(y,s) = F(y,s) + \int_\mathbb{R} f(y',s) K(y,s;y') \, dy', \quad s \geq s_0, \, y \in \mathbb{R}, \\
& f (y,s_0) = f_0(y).
\end{split}
\end{equation}

\begin{lemma} [\cite{BKK1}] \label{max_2}
Let $\Omega\subset \mathbb{R}$ be a compact set, and let $f$ be continuous on $\mathbb{R}\times[s_0,\infty)$, satisfying \eqref{ivp} classically on $\Omega^c\times[s_0,\infty)$ and $f(\cdot,s_0)=f_0$. Assume that $s\mapsto\|f(\cdot,s)\|_{L^\infty(\mathbb{R})}$ is finite and continuous. Suppose that the following hold for all $s\ge s_0$:
\begin{subequations}
\begin{align}
& |f(y,s)|\leq d_0\quad\text{for }y\in\Omega, \label{max_2_1}\\
& \lVert f (\cdot,s_0) \rVert_{L^\infty(\mathbb{R})} \leq d_0, \label{max_2_1'}\\
& \int_\mathbb{R} \lvert K (y,s;y' ) \rvert \, dy' \leq \delta D(y,s) \quad \text{for } (y,s ) \in \Omega^c \times [s_0,\infty), \label{max_2_2}\\
& \inf_{(y,s)\in \Omega^c \times [s_0,\infty)}{D(y,s)} \geq \lambda_D >0, \label{max_2_3}\\
& \sup_{y\in\Omega^c}|F(y,s)|\leq F_0, \label{max_2_4}\\
& \limsup_{\lvert y \rvert \to \infty}{\lvert f(y,s) \rvert} < 2 d_0 \label{max_2_5}
\end{align}
\end{subequations}
for some $d_0,F_0,\lambda_D>0$ and $\delta \in (0,1)$ satisfying
\begin{equation} \label{D-cond} 
d_0 \lambda_D > \frac{F_0}{2(1-\delta)}.
\end{equation}
Then, it holds that $\lVert f(\cdot,s) \rVert_{L^\infty(\mathbb{R})} \leq 2d_0$ for all $s \geq s_0$.
\end{lemma}

\begin{proof}
The proof follows the same argument as in Lemma~SM2.1 of the supplementary material to \cite{BKK1}.
\end{proof}

Next, we consider the initial value problem:
\begin{equation} \label{ivp2}
\begin{split}
& \partial_s f(y,s) + D(y,s) f(y,s) + U(y,s) \partial_y f(y,s) = F(y,s), \quad s \geq s_0, \, y \in \mathbb{R}, \\
& f(y,s_0) = f_0(y).
\end{split}
\end{equation}
The following lemma provides decay estimates for solutions to the transport-type equation \eqref{ivp2} under suitable assumptions:

\begin{lemma}[\cite{BKK1}] \label{rmk2}
Fix $S>s_0$ and $N>0$. Let $f$ be continuous on $\mathbb{R}\times[s_0,S]$, satisfying \eqref{ivp2} classically for $|y|>N$ and $f(\cdot,s_0)=f_0$. Assume that $U$ and $\partial_yU$ are continuous on $\mathbb{R}\times[s_0,S]$, with $|U(y,s)|\leq C_S(1+|y|)$ for some $C_S>0$, and that $D$ and $F$ are continuous on $\{y\in\mathbb{R}:|y|\geq N\}\times[s_0,S]$. Suppose that
\begin{subequations}
\begin{align}
& \inf_{ \lvert y \rvert \geq N, s \in [s_0,S] }{U(y,s)\frac{y}{\lvert y \rvert}} >0, \label{rmk2_ass0} \\
& \inf_{\lvert y \rvert \geq N, s \in [s_0,S]}{D(y,s)} \geq \lambda_D, \label{rmk2_ass1} \\
& \lVert F(\cdot,s) \rVert_{L^\infty(\lvert y \rvert \geq N)} \leq F_0 e^{-s \lambda_F} \label{rmk2_ass2}
\end{align}
\end{subequations}
for some $\lambda_D,\lambda_F,F_0\geq0$ and all $s\in[s_0,S]$. Then, for all $s\in[s_0,S]$, it holds that
\begin{equation*}
	\begin{array}{l l}
		\limsup_{| y |\rightarrow \infty}|f(y,s)|\leq \limsup_{|y|\rightarrow \infty}{|f(y,s_0)|}e^{-\lambda_D(s-s_0)}+\frac{F_0}{\lambda_D-\lambda_F}e^{-s\lambda_F} & \quad \text{if } \lambda_D>\lambda_F,  \\
		\limsup_{| y |\rightarrow \infty}|f( y ,s)|\leq \limsup_{|y|\rightarrow \infty}{|f(y,s_0)|}e^{-\lambda_D(s-s_0)}+\frac{F_0e^{-s_0\lambda_F}}{\lambda_F-\lambda_D}e^{-\lambda_D(s-s_0)} & \quad \text{if } \lambda_F>\lambda_D.
	\end{array}
\end{equation*}
\end{lemma}

\begin{proof}
The proof follows the same argument as in \cite[Lemma~3.5]{BKK1}.
\end{proof}

\section{Completion of the proof of Lemma~\ref{Lemma:Z_high}} \label{App:Z}

We complete the proof of Lemma~\ref{Lemma:Z_high}, based on the intermediate results obtained earlier. The notation remains unchanged from the earlier part of the proof.

In what follows, we estimate \eqref{eq:Z_high} for $n=3,4$, where the damping term $D_n^Z$ and forcing term $F_n^Z$ are given in \eqref{D^Z} and \eqref{F^Z}, respectively. Applying \eqref{lem_Zy}, \eqref{dottau1}, and \eqref{Uy1}, we obtain
\begin{equation*}
D_3^Z \geq \frac{15}{2} - \left( 3 e^{3s/2} \lvert Z_y \rvert + \frac{2\lvert W_y \rvert}{3} \right) (1+\varepsilon^{1/2}) \geq 6
\end{equation*}
for sufficiently small $\varepsilon>0$. Additionally, using \eqref{Wyy_bound}, \eqref{Wy3_bound}, \eqref{Q_high}, and \eqref{Zyy_proof}, we estimate
\begin{equation*}
\begin{split}
\lvert F_3^Z \rvert & \leq C \left( e^{-3s/2} W_{yy}^2 + e^{-3s/2}|W_y||\partial_y^3W| + e^{-5s} + e^{-5s/2} + e^{-7s/2} \right) \\
& \leq C e^{-5s/2} + C e^{-3s/2} \left( W_{yy}^2 + W_y^2 + (\partial_y^3W)^2 \right),
\end{split}
\end{equation*}
where in the second inequality we used Young's inequality. Integrating the equation for $\partial_y^3 Z$ along the curve $\psi_Z$ defined in \eqref{curves}, we obtain
\begin{equation} \label{pr_Zy3'}
\begin{split}
\lvert \partial_y^3 Z(\psi_Z(y,s),s)\rvert & \leq \lVert \partial_y^3 Z (\cdot,s_0) \rVert_{L^\infty} e^{-6(s-s_0)} + C \int_{s_0}^s e^{-6(s-s')} e^{-5s'/2} \, ds' \\
& \quad + C \int_{s_0}^s e^{-6(s-s')} e^{-3s'/2} \left( W_{yy}^2 + W_y^2 + (\partial_y^3W)^2 \right)(\psi_Z(y,s'),s') \, ds' \\
& \leq C e^{-5s/2} + C e^{-5s/2} \int_{s_0}^s e^{s'} \left( W_{yy}^2 + W_y^2 + (\partial_y^3W)^2 \right)(\psi_Z(y,s'),s') \, ds',
\end{split}
\end{equation}
where we used
\begin{equation*}
e^{-6s }\int_{s_0}^s e^{9s'/2} \left( W_{yy}^2 + W_y^2 + (\partial_y^3W)^2 \right) \, ds' \leq e^{-5s/2}\int_{s_0}^s e^{s'} \left( W_{yy}^2 + W_y^2 + (\partial_y^3W)^2 \right) \, ds'.
\end{equation*}
Analogously to \eqref{Wy2bd}, we use \eqref{psiW'bd} and \eqref{init_wbound} to obtain
\begin{equation} \label{Wy32bd}
\begin{split}
& \int_{s_0}^s e^{s'} (\partial_y^3 W)^2 (\psi_Z(y_z,s'),s') \, ds' \leq \frac{3}{4 \sqrt{h_{\mathrm{min}}}} \left( \varepsilon^{10} \lVert \partial_x^3w_0 \rVert_{L^2}^2 + J_3 \right) \leq \frac{3}{4 \sqrt{h_{\mathrm{min}}}} \left( C\varepsilon^{1/2} + J_3 \right),
\end{split}
\end{equation}
where
\begin{equation*}
J_3 := \int_{y_w(s)}^{y_z} \int_{s_0}^{\overline{s}(\xi)} \frac{d}{ds'} \left( e^{-s'/2} (\partial_y^3 W \circ \psi_W)^2 \frac{\partial \psi_W}{\partial \xi} \right)(\xi,s') \, ds' d\xi.
\end{equation*}
By \eqref{Wy3} and \eqref{psiWy}, the integrand of $J_3$ satisfies
\begin{equation*}
\begin{split}
& \frac{d}{ds'} \left( e^{-s'/2}(\partial_y^3 W)^2 \frac{\partial \psi_W}{\partial \xi} \right) \\
& \quad = -\frac{1}{2}e^{-s'/2}(\partial_y^3 W)^2 \frac{\partial \psi_W}{\partial \xi} + 2e^{-s'/2}(\partial_y^3 W)\frac{d(\partial_y^3 W)}{ds'}\frac{\partial \psi_W}{\partial \xi} + e^{-s'/2}(\partial_y^3 W)^2\frac{\partial^2\psi_W}{\partial s'\partial\xi} \\
& \quad = e^{-s'/2}\frac{\partial \psi_W}{\partial \xi}(\partial_y^3 W)\bigg( -10\partial_y^3 W - \frac{5W_y\partial_y^3 W}{1-\dot{\tau}} + \frac{5e^{3s'/2}Z_y\partial_y^3 W}{3(1-\dot{\tau})}  -\frac{16e^{s'/2}\partial_y^3 Q}{3(1-\dot{\tau})} - \frac{4W_{yy}^2}{1-\dot{\tau}} \\
& \qquad \qquad \qquad \qquad \qquad \quad + \frac{14e^{3s'/2}W_{yy}Z_{yy}}{3(1-\dot{\tau})} + \frac{8e^{3s'/2}W_y\partial_y^3 Z}{3(1-\dot{\tau})} + \frac{2e^{3s'}Z_{yy}^2}{1-\dot{\tau}} + \frac{2e^{3s'}Z_y\partial_y^3 Z}{1-\dot{\tau}} \bigg).
\end{split}
\end{equation*}
As in the estimate of $J_2$, we apply Fubini's theorem and Young's inequality to obtain
\begin{equation} \label{J3_bound}
\begin{split}
J_3 & \leq -\frac{1}{2}\int_{s_0}^s e^{-s'/2}\int_{\psi_W(y_w(s),s')}^{\psi_Z(y_z,s')}(\partial_y^3 W)^2 \, dy \, ds' \\
& \quad + C\int_{s_0}^s e^{-s'/2}\int_{\psi_W(y_w(s),s')}^{\psi_Z(y_z,s')}(W_{yy}^2+W_y^2) \, dy \, ds' + C\varepsilon^{1/2} \\
& \leq -\frac{1}{2}\int_{s_0}^s e^{-s'/2}\int_{\psi_W(y_w(s),s')}^{\psi_Z(y_z,s')}(\partial_y^3 W)^2 \, dy \, ds' + C\varepsilon^{1/5}.
\end{split}
\end{equation}
Here we used \eqref{Wy3_bound} on $|y|\leq1$ and the bound
\begin{equation*}
-10 - \frac{5W_y}{1-\dot{\tau}} + \frac{5e^{3s'/2}Z_y}{3(1-\dot{\tau})} \leq -2, \quad |y|\geq1,
\end{equation*}
which follows from \eqref{Uy1}, \eqref{lem_Zy}, and \eqref{dottau1} for sufficiently small $\varepsilon>0$. We also used \eqref{Wyy_bound}, \eqref{lem_Zy}, \eqref{Z_high}, \eqref{Q_high}, \eqref{Zyy_proof}, and \eqref{I_length} in the first inequality, and \eqref{I_Wy2} and \eqref{int_Wyy2_I} in the last inequality.

Since the integral of $(\partial_y^3 W)^2$ in \eqref{J3_bound} is nonnegative, we have $J_3\leq C\varepsilon^{1/5}$. Substituting $J_3\leq C\varepsilon^{1/5}$ into \eqref{Wy32bd} yields
\begin{equation} \label{int_Wy32}
\int_{s_0}^s e^{s'}(\partial_y^3 W \circ \psi_Z)^2 \, ds' \leq C\varepsilon^{1/5}.
\end{equation}
Moreover, the nonnegativity of the left-hand side of \eqref{Wy32bd} implies $J_3\geq-C\varepsilon^{1/2}$. Using $J_3\geq-C\varepsilon^{1/2}$ in \eqref{J3_bound}, we also obtain
\begin{equation} \label{int_Wy32_I}
\int_{s_0}^s e^{-s'/2}\int_{\psi_W(y_w(s),s')}^{\psi_Z(y_z,s')}(\partial_y^3 W)^2 \, dy \, ds' \leq C\varepsilon^{1/5}.
\end{equation}
Substituting \eqref{int_Wy2}, \eqref{int_Wyy2}, and \eqref{int_Wy32} into \eqref{pr_Zy3'}, we obtain $|\partial_y^3 Z(\psi_Z(y,s),s)|\leq Ce^{-5s/2}$ uniformly in $y\in\mathbb{R}$. Taking the supremum over $y\in\mathbb{R}$ gives
\begin{equation} \label{Zy3_proof}
\lVert \partial_y^3 Z(\cdot,s) \rVert_{L^\infty} \leq Ce^{-5s/2}.
\end{equation}

Next, we consider the case $n=4$. By \eqref{lem_Zy}, \eqref{dottau1}, and \eqref{Uy1}, $D_4^Z$ satisfies
\begin{equation*}
D_4^Z \geq 10 - \left( 4 e^{3s/2} \lvert Z_y \rvert + \frac{\lvert W_y \rvert}{3} \right) (1+\varepsilon^{1/2}) \geq 9
\end{equation*}
for sufficiently small $\varepsilon>0$. Also, by \eqref{lem_Zy}, \eqref{Wyy_bound}, \eqref{Wy3_bound}, \eqref{dottau1}, \eqref{Zyy_proof}, \eqref{Zy3_proof}, \eqref{Q_high}, and Young's inequality, it holds that
\begin{equation*}
| F_4^Z| \leq C e^{-5s/2} + C e^{-3s/2} \left( |W_y|^2 + |W_{yy}|^2 + |\partial_y^3W|^2 + |\partial_y^4W|^2 \right).
\end{equation*}
Hence, integrating the equation for $\partial_y^4 Z$ along the curve $\psi_Z$, we obtain
\begin{equation} \label{pr_Zy4'}
\begin{split}
\lvert \partial_y^4 Z(\psi_Z(y,s),s)\rvert & \leq \lVert \partial_y^4 Z (\cdot,s_0) \rVert_{L^\infty} e^{-9(s-s_0)} + C \int_{s_0}^s e^{-9(s-s')} e^{-5s'/2} \, ds' \\
& \quad + C \int_{s_0}^s e^{-9(s-s')} e^{-3s'/2} \left( W_y^2 + W_{yy}^2 + (\partial_y^3 W)^2 + (\partial_y^4 W)^2 \right) (\psi_Z(y,s'),s') \, ds' \\
& \leq C e^{-5s/2} + C e^{-5s/2} \int_{s_0}^s e^{s'} \left( W_y^2 + W_{yy}^2 + (\partial_y^3 W)^2 + (\partial_y^4 W)^2 \right) (\psi_Z(y,s'),s') \, ds'.
\end{split}
\end{equation}
As in \eqref{Wy2bd}, we use \eqref{psiW'bd} and \eqref{init_wbound} to obtain
\begin{equation} \label{Wy42bd}
\begin{split}
\int_{s_0}^s e^{s'} (\partial_y^4 W)^2 (\psi_Z(y_z,s'),s') \, ds' \leq \frac{3}{4 \sqrt{h_{\mathrm{min}}}} \left( \varepsilon^{15} \lVert \partial_x^4 w_0 \rVert_{L^2}^2 + J_4 \right) \leq \frac{3}{4 \sqrt{h_{\mathrm{min}}}} \left( C \varepsilon^{1/2}  + J_4 \right),
\end{split}
\end{equation}
where
\begin{equation*} 
J_4 := \int_{y_w(s)}^{y_z} \int_{s_0}^{\overline{s}(\xi)} \frac{d}{ds'} \left( e^{-s'/2} (\partial_y^4 W \circ \psi_W)^2 \frac{\partial \psi_W}{\partial \xi} \right)(\xi,s') \, ds' d\xi.
\end{equation*}
By \eqref{Wy4} and \eqref{psiWy}, we have
\begin{equation*}
\begin{split}
& \frac{d}{ds'} \left(e^{-s'/2}(\partial_y^4 W)^2 \frac{\partial \psi_W}{\partial \xi}\right) \\
& \quad = e^{-s'/2} \frac{\partial \psi_W}{\partial \xi}(\partial_y^4 W) \bigg(-15 \partial_y^4 W - \frac{7W_y \partial_y^4 W}{1 - \dot{\tau}} + \frac{e^{3s'/2} Z_y \partial_y^4 W}{1 - \dot{\tau}} -\frac{16e^{s'/2} \partial_y^4 Q}{3(1 - \dot{\tau})} \\
& \qquad \qquad \qquad \qquad \qquad \quad - \frac{14W_{yy} \partial_y^3 W}{1 - \dot{\tau}} + \frac{6e^{3s'/2} \partial_y^3 W Z_{yy}}{1 - \dot{\tau}} + \frac{22e^{3s'/2} W_{yy} \partial_y^3 Z}{3(1 - \dot{\tau})} \\
& \qquad \qquad \qquad \qquad \qquad \quad + \frac{8e^{3s'/2} W_y \partial_y^4 Z}{3(1 - \dot{\tau})} + \frac{6e^{3s'} Z_{yy} \partial_y^3 Z}{1 - \dot{\tau}} + \frac{2e^{3s'} Z_y \partial_y^4 Z}{1 - \dot{\tau}} \bigg).
\end{split}
\end{equation*}
As in the estimate of $J_3$, we apply Young's inequality and Fubini's theorem to obtain
\begin{equation*}
\begin{split}
J_4 & \leq C\int_{s_0}^s e^{-s'/2}\int_{\psi_W(y_w(s),s')}^{\psi_Z(y_z,s')} \left(W_{yy}^2+(\partial_y^3 W)^2+W_y^2+e^{-2s'}\right) \, dy \, ds' \\
& \leq C\varepsilon^{1/5}+C\int_{s_0}^s \left(e^{-s'/5}+e^{-s'}\right) \, ds' \leq C\varepsilon^{1/5}.
\end{split}
\end{equation*}
Here, in the first inequality, we used \eqref{Wyy_bound}, \eqref{Z_high}, \eqref{Q_high}, \eqref{Zyy_proof}, and \eqref{Zy3_proof}, together with the bound
\begin{equation*}
-15-\frac{7W_y}{1-\dot{\tau}}+\frac{e^{3s'/2}Z_y}{1-\dot{\tau}} \leq -15+14(1+\varepsilon^{1/2})+Ce^{-s'} \leq -\frac{3}{4},
\end{equation*}
which follows from \eqref{Uy1}, \eqref{dottau1}, and \eqref{lem_Zy} for sufficiently small $\varepsilon>0$. The resulting nonpositive term involving $(\partial_y^4 W)^2$ was discarded. In the second inequality, we used \eqref{int_Wyy2_I} and \eqref{int_Wy32_I} to estimate the integrals of $W_{yy}^2$ and $(\partial_y^3 W)^2$, and \eqref{I_Wy2} and \eqref{I_length} for those of $W_y^2$ and $e^{-2s'}$, respectively.

Substituting $J_4\leq C\varepsilon^{1/5}$ into \eqref{Wy42bd}, we obtain
\begin{equation} \label{int_Wy42}
\int_{s_0}^s e^{s'}(\partial_y^4 W \circ \psi_Z)^2 \, ds' \leq C\varepsilon^{1/5}.
\end{equation}
Using \eqref{int_Wy2}, \eqref{int_Wyy2}, \eqref{int_Wy32}, and \eqref{int_Wy42} in \eqref{pr_Zy4'}, and taking the supremum over $y\in\mathbb{R}$, we obtain
\begin{equation*}
\lVert \partial_y^4 Z(\cdot,s) \rVert_{L^\infty} \leq Ce^{-5s/2}.
\end{equation*}
This completes the proof of Lemma~\ref{Lemma:Z_high}.

\section{\texorpdfstring{Bounds on the Green function $K(x,z)$}{Bounds on the Green function K(x,z)}} \label{App:K}

In this appendix, we establish exponential decay bounds on the Green function $K(x,z)$ associated with the operator
\begin{equation*}
1 - \mathcal{L}:= 1- \partial_x \circ h^{-1} \partial_x \circ h^3,
\end{equation*}
which satisfies
\begin{equation*}
(1 - \mathcal{L})K = \delta(x-z).
\end{equation*}
Throughout this section, the time variable $t$ is fixed, and we regard $h = h(x,t)$ simply as a function of $x$. For notational simplicity, we write $h(x)$.

To facilitate the analysis, we introduce a change of variables
\begin{equation} \label{zvariable}
	\tilde{x} = \int_0^x h (x') \, dx',
\end{equation}
under which $\tilde{x}$ becomes a strictly increasing bijection from $\mathbb{R}$ to $\mathbb{R}$, owing to the fact that
\begin{equation*}
h(x)>0, \quad \lim_{x \to \pm \infty}{h(x)} = h_*.
\end{equation*}
Letting $\tilde{f}(\tilde{x}) := (h^3 f)(x)$, 
one can define an operator $L$ by 
\begin{equation*}
	\begin{split}
		(1-\mathcal{L})f(x)&=\frac{\tilde{f}(\tilde{x})}{h^3(x)}- h(x)\partial_{\tilde{x}}^2\tilde{f}(\tilde{x})
		\\
		&=h(x)\left(\frac{1}{h^4(x)}- \partial_{\tilde{x}}^2\right)\tilde{f}(\tilde{x})
		\\
		&=:h(x)L\tilde{f}(\tilde{x}) \quad \text{for any }f.
	\end{split}
\end{equation*}
From the relation between the operators $(1-\mathcal{L})$ and $L$, we have $L\tilde{f}(\tilde{x}) = h^{-1}(x) g(x) =: \tilde{g}(\tilde{x})$ for any $f$ and $g$ satisfying $(1-\mathcal{L})f(x) = g(x)$. We also note that the solution $f$ decaying at both infinities can be written as $\textstyle f(x) = \int_\mathbb{R} K(x,z) g(z) \, dz$ using the Green function $K(x,z)$. Hence, we have the representation of $\tilde{f}$:
\begin{equation}\label{tildef}
	\begin{split}
		\tilde{f}(\tilde{x})&=h^3(x)f(x)=\int_{\mathbb{R}}h^3(x)K(x,z)g(z)\,dz
		\\
		&=\int_{\mathbb{R}}h^3(x)K(x,z)\tilde{g}(\tilde{z})\,d\tilde{z},
	\end{split}
\end{equation} 
where in the last equality, we used the relation $h(z)dz = d\tilde{z}$. Define the Green function $\tilde{K}(\tilde{x},\tilde{z})$ for $L$ as the solution decaying as $\tilde{x}\to\pm\infty$ to
\begin{equation} \label{defKtilde}
L \tilde{K} = \delta(\tilde{x}-\tilde{z}).
\end{equation}
Comparing \eqref{tildef} with the representation $\textstyle \tilde{f}(\tilde{x})=\int_{\mathbb{R}}\tilde{K}(\tilde{x},\tilde{z})\tilde{g}(\tilde{z})\,d\tilde{z}$, we obtain the following relation between $K(x,z)$ and $\tilde{K}(\tilde{x},\tilde{z})$:
\begin{equation}\label{K-Ktilde}
	\tilde{K}(\tilde{x},\tilde{z})=h^3(x)K(x,z).
\end{equation}

To construct $\tilde{K}(\tilde{x},\tilde{z})$, we use the classical matching method, see \cite[pp.371--376]{CH}. First we rewrite $Ly=0$ as the first-order ODE system
\begin{equation} \label{1stODE}
	Y' = \mathbb{A}(\tilde{x})Y,
\end{equation}
where $'$ denotes $d/d\tilde{x}$, and $Y = (y,y')^{\mathrm{tr}}$. The coefficient matrix $\mathbb{A}(\tilde{x})$ is given by
\begin{equation*}
\mathbb{A}(\tilde{x}) := \begin{pmatrix}
	0 & 1 \\
	h^{-4}(x) & 0
\end{pmatrix}.
\end{equation*}
It follows from $\lim_{\lvert x \rvert \to \infty}{h(x)} = h_*$ that
\begin{equation*}
\lim_{|\tilde{x}| \to \infty}{\mathbb{A}(\tilde{x})} = \begin{pmatrix}
0 & 1 \\
h_*^{-4} & 0 
\end{pmatrix} =: \mathbb{A}_\infty.
\end{equation*}
Moreover, as $h(x)-h_* \in H^5(\mathbb{R})$, we obtain $\mathbb{A}(\tilde{x}) \in C^4 \cap L^\infty$ and
\begin{equation} \label{A_L2}
\begin{split}
\int_\mathbb{R} \lVert \mathbb{A}(\tilde{x}) - \mathbb{A}_\infty \rVert_F^2 \, d\tilde{x} & =  \int_\mathbb{R} \left( \frac{1}{h^4} - \frac{1}{h_*^4} \right)^2 h \, dx =  \int_\mathbb{R} \frac{|h-h_*|^2 |h+h_*|^2 |h^2+h_*^2|^2}{h^7 h_*^8} \, dx \\
& \leq \frac{|h_{\mathrm{max}}+h_*|^2 |h_{\mathrm{max}}^2+h_*^2|^2}{ h_{\mathrm{min}}^7 h_*^8} \int_\mathbb{R} | h-h_*|^2 \, dx < \infty,
\end{split}
\end{equation}
where $\|\cdot\|_{F}$ is the Frobenius matrix norm defined as $\textstyle \|A\|_F=\sqrt{\sum_{i,j}|A_{ij}|^2}$. Theorem~1 in \cite{Pi} (or, Theorems~6' and~8 in \cite[Chap.~IV]{Co}), together with \eqref{A_L2}, ensures that the ODE system \eqref{1stODE} admits solutions $Y^\pm(\tilde{x})$ satisfying
\begin{equation*}
\begin{split}
& Y^+(\tilde{x}) = \left( v_+ + o(1) \right) e^{\lambda_+ \tilde{x} + \Theta_+(\tilde{x})} \quad \text{as } \tilde{x} \to +\infty, \\
& Y^-(\tilde{x}) = \left( v_- + o(1) \right) e^{\lambda_- \tilde{x} + \Theta_-(\tilde{x})} \quad \text{as } \tilde{x} \to -\infty,
\end{split}
\end{equation*}
where the eigenvalue $\lambda_\pm$ of $\mathbb{A}_\infty$ and the corresponding eigenvector $v_\pm$ are given by
\begin{equation*}
\lambda_\pm = \mp \frac{1}{h_*^2} \lessgtr 0 \quad \text{and} \quad v_\pm = \begin{pmatrix} 1 \\ \lambda_\pm \end{pmatrix},
\end{equation*}
respectively. The asymptotic corrections $\Theta_\pm(\tilde{x})$ take the form
\begin{equation*}
\Theta_\pm(\tilde{x}) = \mp \frac{h_*^2}{2} \int_{0}^{\tilde{x}}   \left( \frac{1}{h^4} - \frac{1}{h_*^4}  \right) \, d\tilde{x}.
\end{equation*}
Applying the H\"older inequality, together with a similar computation with \eqref{A_L2}, we find that
\begin{equation*}
| \Theta_\pm(\tilde{x}) | \leq C \left| \int_0^{\tilde{x}} \left( \frac{1}{h^4} - \frac{1}{h_*^4} \right) \, d\tilde{x}  \right| \leq C \left| \int_0^{\tilde{x}} 1 \, d\tilde{x} \right|^{1/2} \left| \int_0^{\tilde{x}} \left( \frac{1}{h^4} - \frac{1}{h_*^4} \right)^2 \, d\tilde{x} \right|^{1/2} \leq C |\tilde{x}|^{1/2}.
\end{equation*}
Thus, the solutions $Y^\pm(\tilde{x})$ decay exponentially as $ \tilde{x} \to \pm \infty$, respectively. Since 
\begin{equation*}
\lim_{\tilde{x} \to + \infty}{\mathbb{A}(\tilde{x})} = \lim_{\tilde{x} \to -\infty}{\mathbb{A}(\tilde{x})} = \mathbb{A}_\infty,
\end{equation*}
the exponential rates of growth and decay coincide at both ends. This leads to the following global bounds:
\begin{equation} \label{Ybd}
|Y^\pm (\tilde{x})| \leq C e^{\lambda_\pm \tilde{x} + \Theta_\pm(\tilde{x})}, \quad \tilde{x} \in \mathbb{R}.
\end{equation}
Note that there exist constants $a,b$ such that $Y^+(\tilde{x}) = a Y^{(+)}(\tilde{x}) + b Y^{(-)}(\tilde{x})$ for $\tilde{x} \leq -R$ with sufficiently large $R>0$, where $Y^{(\pm)}$ are growing/decaying modes as $\tilde{x} \to -\infty$ corresponding to $\lambda_\pm$. Hence,
\begin{equation*}
|Y^+(\tilde x)| \le |a|\,e^{\lambda_+\tilde x+\Theta_+(\tilde x)}
   + |b|\,e^{\lambda_-\tilde x+\Theta_-(\tilde x)}
 = e^{\lambda_+\tilde x+\Theta_+(\tilde x)}
   \Bigl(|a| + |b|\,e^{(\lambda_--\lambda_+)\tilde x+(\Theta_--\Theta_+)(\tilde x)}\Bigr).
\end{equation*}
Here $\lambda_--\lambda_+=2h_*^{-2}>0$ and $\Theta_- - \Theta_+ = -2 \Theta_+ = O(|\tilde x|^{1/2})$, so we have $|Y^+(\tilde x)| \le C\,e^{\lambda_+\tilde x+\Theta_+(\tilde x)}$. The same bound holds for $\tilde x\ge R$ by the asymptotics of $Y^+$ at $+\infty$, and on the compact interval $[-R,R]$ it follows by continuity after increasing $C$ if necessary. This proves the global estimate \eqref{Ybd} for $Y^+$, and the case of $Y^-$ is symmetric.

Next we derive the jump condition of $\tilde{K}(\tilde{x},\tilde{z})$ at $\tilde{x}=\tilde{z}$. By continuity, we have
\begin{equation} \label{Jc'}
[\tilde{K}]_{\tilde{x}=\tilde{z}} = 0,
\end{equation}
where $[\cdot]_{\tilde{x}=\tilde{z}}$ denotes the jump condition of a function at $\tilde{x}=\tilde{z}$. Integrating \eqref{defKtilde} from $\tilde{x}=\tilde{z}-\eta$ to $\tilde{x}=\tilde{z}+\eta$ and taking the limit $\eta \to 0$, we obtain
\begin{equation} \label{Jc}
[\tilde{K}_{\tilde{x}}]_{\tilde{x}=\tilde{z}} = -1.
\end{equation}

Now we are ready to construct the Green function $\tilde{K}(\tilde{x},\tilde{z})$. We set
\begin{equation*}
\tilde{K}(\tilde{x},\tilde{z}) = \begin{cases}
y_+(\tilde{x}) N^+(\tilde{z}), & \tilde{x} > \tilde{z}, \\
y_-(\tilde{x}) N^-(\tilde{z}), & \tilde{x} < \tilde{z},
\end{cases}
\end{equation*}
where $y_\pm$ is the first component of $Y^\pm$. Then, by the jump condition \eqref{Jc'}--\eqref{Jc}, we have
\begin{equation*}
\begin{split}
& y_+(\tilde{z}) N^+(\tilde{z}) = y_-(\tilde{z}) N^-(\tilde{z}), \\
& y_+'(\tilde{z}) N^+(\tilde{z}) - y_-'(\tilde{z}) N^-(\tilde{z}) = -1,
\end{split}
\end{equation*}
yielding
\begin{equation*}
N^+(\tilde{z}) = - \frac{y_-}{y_+' y_- - y_+ y_-' }, \quad N^-(\tilde{z}) = - \frac{y_+}{y_+' y_- - y_+ y_-' }.
\end{equation*}
Thus, the Green function $\tilde{K}$ is represented as
\begin{equation*}
\tilde{K}(\tilde{x},\tilde{z}) = \begin{cases}
- \frac{y_+(\tilde{x}) y_-(\tilde{z})}{ (y_+' y_- - y_+ y_-')(\tilde{z})}, & \tilde{x} > \tilde{z}, \\
- \frac{y_-(\tilde{x}) y_+(\tilde{z})}{ (y_+' y_- - y_+ y_-')(\tilde{z})}, & \tilde{x} < \tilde{z}.
\end{cases}
\end{equation*}
Since the solutions $Y^+$ and $Y^-$ are linearly independent, the Wronskian $\left(y_+' y_- - y_+ y_-' \right)$ is non-zero. Moreover, we can check that the Wronskian is a constant by the following calculation:
\begin{equation*}
\left(y_+' y_- - y_+ y_-' \right)' = y_+'' y_- - y_+y_-'' = \left( h^{-4} y_+ \right) y_- - y_+ \left( h^{-4} y_- \right) = 0,
\end{equation*}
where we used the fact that $y_\pm$ solve $Ly =0$. Using the bound \eqref{Ybd} on $Y^\pm = (y_\pm,y_\pm')^{\mathrm{tr}}$, together with $\lambda_+ = -\lambda_-$ and $\Theta_+ = - \Theta_-$, we obtain for $\tilde{x}>\tilde{z}$
\begin{equation*}
\begin{split}
|\tilde{K}(\tilde{x},\tilde{z})| & \leq C |y_+(\tilde{x})| |y_-(\tilde{z})| \\
& \leq C e^{\lambda_+ (\tilde{x}-\tilde{z})} e^{\Theta_+(\tilde{x})-\Theta_+(\tilde{z})} \\
& \leq C \exp{ \left[ -\frac{1}{2h_*^2}\left( 1 + \frac{h_*^4}{h_{\mathrm{max}}^4} \right) (\tilde{x}-\tilde{z}) \right] }  =: C e^{-\theta |\tilde{x}-\tilde{z}|}, \quad \theta>0,
\end{split}
\end{equation*}
where in the third inequality we used
\begin{equation*}
\begin{split}
\Theta_+(\tilde{x}) - \Theta_+(\tilde{z}) & = - \frac{h_*^2}{2} \int_{\tilde{z}}^{\tilde{x}} \left( \frac{1}{h^4} - \frac{1}{h_*^4} \right) \, d\tilde{x}'   \leq - \frac{h_*^2}{2} \left( \frac{1}{h_{\mathrm{max}}^4} - \frac{1}{h_*^4} \right) (\tilde{x}-\tilde{z}). 
\end{split}
\end{equation*}
A similar estimate holds for $\tilde{x} < \tilde{z}$, and hence, for all $\tilde{x},\tilde{z}\in\mathbb{R}$, there exists a constant  $\tilde{\theta}>0$ such that  
\begin{equation*}
|\tilde{K}(\tilde{x},\tilde{z})| \leq C e^{-\tilde{\theta}|\tilde{x}-\tilde{z}|}.
\end{equation*}
Using the relation between $K(x,z)$ and $\tilde{K}(\tilde{x},\tilde{z})$ given in \eqref{K-Ktilde}, we obtain
\begin{equation} \label{Kbound}
|K(x,z)| \leq C |\tilde{K}(\tilde{x},\tilde{z})| \leq C e^{-\tilde{\theta}|\tilde{x}-\tilde{z}|} \leq C e^{-\theta_0 |x-z|},
\end{equation}
where $\theta_0 := h_{\mathrm{min}} \tilde{\theta} >0$. By \eqref{A_L2}, \eqref{Energy}--\eqref{E0_bound}, and Lemma~\ref{h_max}, the constants $C$ and $\theta_0$ in \eqref{Kbound} can be chosen to depend only on $h_*$, $h_{\mathrm{min}}$, and $h_{\mathrm{max}}$, and are therefore uniform in $t$.

%
%

\end{document}